\documentclass[11pt,reqno]{amsart}

\usepackage{amssymb}
\usepackage{subfiles}
\usepackage{stackengine}
\usepackage{mathtools}
\usepackage[utf8]{inputenc}
\usepackage{color}
\usepackage{tikz}
\usepackage[a4paper,top=3cm,bottom=3cm,inner=3cm,outer=3cm]{geometry}
\usepackage{mathrsfs,amsmath,amsfonts,amsthm,graphicx,pdfpages}
\usepackage{comment}
\usepackage{graphicx}
\usepackage{adjustbox}
\usepackage[all,2cell]{xy}\UseAllTwocells\SilentMatrices
\definecolor{darkgreen}{rgb}{0,0.45,0}
\definecolor{myurlcolor}{rgb}{0.6,0,0} \definecolor{mycitecolor}{rgb}{0,0,0.8} \definecolor{myrefcolor}{rgb}{0,0,0.8}
\usepackage{hyperref} \hypersetup{colorlinks, linkcolor=myrefcolor, citecolor=darkgreen, urlcolor=myurlcolor,
final,linktoc=page}
\usepackage[capitalize]{cleveref}
\crefname{equation}{}{}
\crefname{item}{}{}
\usepackage{enumerate}

\usetikzlibrary{decorations.markings,arrows.meta,calc,fit,quotes,cd,math,external}
  
\tikzset{tick/.style={postaction={decorate,decoration={markings,
mark=at position 0.5 with {\draw[-] (0,.4ex) -- (0,-.4ex);}}}}}
\tikzset{tickx/.style={postaction={decorate,decoration={markings,mark=at position 0.5 with
{\fill circle [radius=.28ex];}}}}}

\newtheorem*{thm*}{Theorem}
\theoremstyle{remark}
\newtheorem*{rmk*}{Remark}
\newtheorem*{lem*}{Lemma}
\theoremstyle{definition}
\newtheorem*{defi*}{Definition}
\newtheorem*{cor*}{Corollary}
\theoremstyle{definition}
\newtheorem*{examples*}{Examples}
\newtheorem{prop*}{Proposition}

\theoremstyle{plain}
\newtheorem{thm}{Theorem}[section]
\theoremstyle{plain}
\newtheorem{prop}[thm]{Proposition}
\theoremstyle{remark}
\newtheorem{rmk}[thm]{Remark}
\theoremstyle{plain}
\newtheorem{lem}[thm]{Lemma}
\theoremstyle{plain}
\newtheorem{cor}[thm]{Corollary}
\theoremstyle{definition}
\newtheorem{defi}[thm]{Definition}
\theoremstyle{definition}
\newtheorem{examples}[thm]{Example}

\DeclareFontFamily{U}{mathx}{\hyphenchar\font45}
\DeclareFontShape{U}{mathx}{m}{n}{
      <5> <6> <7> <8> <9> <10>
      <10.95> <12> <14.4> <17.28> <20.74> <24.88>
      mathx10
      }{}
\DeclareSymbolFont{mathx}{U}{mathx}{m}{n}
\DeclareFontSubstitution{U}{mathx}{m}{n}
\DeclareMathAccent{\widecheck}{0}{mathx}{"71}

\newcommand{\catname}[1]{\mathcal{#1}}
\newcommand{\namedcat}[1]{\mathsf{#1}}
\newcommand{\ps}[1]{\mathscr{#1}}

\newcommand{\TWO}{\adjustbox{scale=.5}{\begin{tikzcd}[baseline=-0.5ex,cramped,sep=.2in,ampersand replacement=\&]\bullet\ar[r,shift left]\ar[r,shift right] \& \bullet
\end{tikzcd}}}

\newcommand{\Comod}{\namedcat{Comod}}
\newcommand{\Mod}{\namedcat{Mod}}
\newcommand{\Cat}{\namedcat{Cat}}

\newcommand{\Mon}{\namedcat{Mon}}
\newcommand{\Comon}{\namedcat{Comon}}

\newcommand{\ob}{\mathrm{ob}}

\newcommand{\cod}{\mathrm{cod}}
\newcommand{\opl}{\mathrm{op\ell}}
\newcommand{\lax}{\ell}
\newcommand{\pse}{\mathrm{ps}}
\newcommand{\Cart}{\ensuremath{\mathrm{Cart}}}
\newcommand{\Cocart}{\ensuremath{\mathrm{Cocart}}}

\newcommand{\op}{\mathrm{op}}

\newcommand{\sbul}{\scriptstyle\bullet}
\newcommand{\tick}{\object@{|}}

\newcommand{\Set}{\namedcat{Set}}
\newcommand{\matr}[3]{\SelectTips{eu}{10}\xymatrix@C=.2in{#1\colon #2\ar[r]|-{\object@{|}} & #3}}
\newcommand{\proar}[3]{\SelectTips{eu}{10}\xymatrix@C=.2in{#1\colon #2\ar[r]|-{\sbul} & #3}}
\newcommand{\simrightarrow}{\xrightarrow{\raisebox{-4pt}[0pt][0pt]{\ensuremath{\sim}}}}
\newcommand{\Fib}{\namedcat{Fib}}
\newcommand{\OpFib}{\namedcat{OpFib}}
\newcommand{\MonFib}{\namedcat{MonFib}}
\newcommand{\MonOpFib}{\namedcat{MonOpFib}}
\newcommand{\BrMonFib}{\namedcat{BrMonFib}}
\newcommand{\SymMonFib}{\namedcat{SymMonFib}}
\newcommand{\SymMonCat}{\namedcat{SymMonCat}}
\newcommand{\BrMonOpFib}{\namedcat{BrMonOpFib}}
\newcommand{\SymMonOpFib}{\namedcat{SymMonOpFib}}
\newcommand{\MonCat}{\namedcat{MonCat}}
\newcommand{\cMonCat}{\namedcat{cMonCat}}
\newcommand{\cocMonCat}{\namedcat{cocMonCat}}
\newcommand{\cMonFib}{\namedcat{cMonFib}}
\newcommand{\cocMonOpFib}{\namedcat{cocMonOpFib}}
\newcommand{\MonICat}{\namedcat{MonICat}}

\newcommand{\OpICat}{\namedcat{OpICat}}
\newcommand{\MonOpICat}{\namedcat{MonOpICat}}
\newcommand{\TCat}{\namedcat{2}\namedcat{Cat}} 
\newcommand{\MonTCat}{\Mon\TCat}
\newcommand{\BrMonTCat}{\Br\MonTCat}
\newcommand{\SymMonTCat}{\Sym\MonTCat}
\newcommand{\SSpan}{\mathbb{S}\namedcat{pan}}

\newcommand{\NetMod}{\namedcat{NetMod}}

\newcommand{\maps}{\colon}
\newcommand{\ICat}{\namedcat{ICat}}
\newcommand{\spl}{\mathrm{s}}
\newcommand{\Fam}{\namedcat{Fam}}
\newcommand{\Maf}{\namedcat{Maf}}

\newcommand\braid{\beta}
\newcommand{\mlt}{m} 
\newcommand{\uni}{j} 
\def\ot{\otimes}
\newcommand{\PsMon}{{\sf{PsMon}}}
\newcommand{\BrPsMon}{{\sf{BrPsMon}}}
\newcommand{\SymPsMon}{{\sf{SymPsMon}}}

\definecolor{purple(x11)}{rgb}{0.5, 0.0, 0.5}

\newcommand{\inta}{\textstyle\int\hspace{-.07in}}

\newcommand{\Br}{\namedcat{Br}}
\newcommand{\Sym}{\namedcat{Sym}}

\definecolor{joecolor(x11)}{rgb}{0.0, 0.5, 0.5}

\newcommand{\define}[1]{{\bf \boldmath{#1}}}
\newcommand{\A}{\catname A}
\newcommand{\B}{\catname B}
\newcommand{\C}{\catname C}
\newcommand{\U}{\catname A}
\newcommand{\V}{\catname X} 
\newcommand{\W}{\catname B}
\newcommand{\Z}{\catname Y} 
\newcommand{\T}{P} 
\newcommand{\ES}{Q} 
\newcommand{\J}{\catname J}
\newcommand{\K}{\catname K}
\renewcommand{\L}{\catname L}
\newcommand{\M}{\ps M}
\newcommand{\N}{\ps N}
\newcommand{\F}{\ps F}
\newcommand{\G}{\ps G}
\newcommand{\X}{\catname X}
\newcommand{\Y}{\catname Y}

\newcommand{\FinSet}{\namedcat{FinSet}}
\renewcommand{\1}{\mathbf{1}}
\usepackage{tikz, tikz-cd}
\usetikzlibrary{positioning}
\definecolor {processblue}{cmyk}{0.9,0.5,0,0}

\tikzstyle{simple}=[-,line width=2.000]
\tikzstyle{arrow}=[-,postaction={decorate},decoration={markings,mark=at position .5 with {\arrow{>}}},line width=1.100]
\pgfdeclarelayer{edgelayer}
\pgfdeclarelayer{nodelayer}
\pgfdeclarelayer{bg}
\pgfsetlayers{bg,edgelayer,nodelayer,main}
\tikzstyle{none}=[inner sep=-1pt]
\tikzstyle{species}=[circle,fill=none,draw=black,scale=1.0]
\tikzstyle{transition}=[rectangle,fill=none,draw=black,scale=1.15]
\tikzstyle{empty}=[circle,fill=none, draw=none]
\tikzstyle{inputdot}=[circle,fill=black,draw=black, scale=.5]
\tikzstyle{dot}=[circle,fill=black,draw=black]
\tikzstyle{bounding}=[circle,dashed, fill=none,draw=black, scale=9.00]
\tikzstyle{simple}=[-,draw=black,line width=1.000]
\tikzstyle{inarrow}=[-,draw=black,postaction={decorate},decoration={markings,mark=at position .5 with {\arrow{>}}},line width=1.000]
\tikzstyle{tick}=[-,draw=black,postaction={decorate},decoration={markings,mark=at position .5 with {\draw (0,-0.1) -- (0,0.1);}},line width=1.000]
\tikzstyle{inputarrow}=[->,draw=black, shorten >=.05cm]

\tikzset{main node/.style={circle,fill=blue!20,draw,minimum size=1cm,inner sep=0pt},}

\newcommand{\Grph}{\namedcat{Grph}}
\newcommand{\Cospan}{\mathsf{Cospan}}
\renewcommand{\2}{\mathbf{2}}

\tikzset{
   oriented WD/.style={
      every to/.style={out=0,in=180,draw},
      label/.style={
         font=\everymath\expandafter{\the\everymath\scriptstyle},
         inner sep=0pt,
         node distance=2pt and -2pt},
      semithick,
      node distance=1 and 1,
      decoration={markings, mark=at position .5 with {\arrow{stealth};}},
      ar/.style={postaction={decorate}},
      execute at begin picture={\tikzset{
         x=\bbx, y=\bby,
         every fit/.style={inner xsep=\bbx, inner ysep=\bby}}}
      },
   bbx/.store in=\bbx,
   bbx = 1.5cm,
   bby/.store in=\bby,
   bby = 1.75ex,
   bb port sep/.store in=\bbportsep,
   bb port sep=2,
   bb port length/.store in=\bbportlen,
   bb port length=4pt,
   bb min width/.store in=\bbminwidth,
   bb min width=1cm,
   bb rounded corners/.store in=\bbcorners,
   bb rounded corners=2pt,
   bb small/.style={bb port sep=1, bb port length=2.5pt, bbx=.4cm, bb min width=.4cm, bby=.7ex},
   bb Small/.style={bb port sep=1, bb port length=2.5pt, bbx=.5cm, bb min width=.5cm, bby=1ex},
   bb/.code 2 args={
      \pgfmathsetlengthmacro{\bbheight}{\bbportsep * (max(#1,#2)+1) * \bby}
      \pgfkeysalso{draw,minimum height=\bbheight,minimum width=\bbminwidth,outer sep=0pt,
         rounded corners=\bbcorners,thick,
         prefix after command={\pgfextra{\let\fixname\tikzlastnode}},
         append after command={\pgfextra{\draw
            \ifnum #1=0{} \else foreach \i in {1,...,#1} {
               ($(\fixname.north west)!{\i/(#1+1)}!(\fixname.south west)$) +(-\bbportlen,0) coordinate
               (\fixname_in\i) -- +(\bbportlen,0) coordinate (\fixname_in\i')}\fi 
            \ifnum #2=0{} \else foreach \i in {1,...,#2} {
               ($(\fixname.north east)!{\i/(#2+1)}!(\fixname.south east)$) +(-\bbportlen,0) coordinate
               (\fixname_out\i') -- +(\bbportlen,0) coordinate (\fixname_out\i)}\fi;
            }}
        }
    },
    bb name/.style={append after command={\pgfextra{\node[anchor=north] at (\fixname.north) {#1};}}}
}

\newcommand{\inp}[1]{{#1}^\mathrm{in}}
\newcommand{\out}[1]{{#1}^\mathrm{out}}

\title{Monoidal Grothendieck Construction}

\author{Joe Moeller} 
\address{Department of Mathematics, University of California, Riverside, USA}
\email{moeller@math.ucr.edu}

\author{Christina Vasilakopoulou}
\address{Department of Mathematics, University of Patras, Greece}
\email{cvasilak@math.upatras.gr}

\begin{document}

\begin{abstract}
    We lift the standard equivalence between fibrations and indexed categories to an equivalence between monoidal fibrations and monoidal indexed categories, namely lax monoidal pseudofunctors to the 2-category of categories. Furthermore, we investigate the relation between this `global' monoidal version where the total category is monoidal and the fibration strictly preserves the structure, and a `fibrewise' one where the fibres are monoidal and the reindexing functors strongly preserve the structure, first hinted by Shulman. In particular, when the domain is cocartesian monoidal, we show how lax monoidal structures on a pseudofunctor to $\Cat$ bijectively correspond to lifts of the pseudofunctor to $\MonCat$. Finally, we give some examples where this correspondence appears, spanning from the fundamental and family fibrations to network models and systems.
\end{abstract}

\maketitle
\setcounter{tocdepth}{1}
\tableofcontents

\section{Introduction}

The Grothendieck construction \cite{Grothendieckcategoriesfibrees} exhibits one of the most fundamental relations in category theory, namely the equivalence between contravariant pseudofunctors into $\Cat$ and fibrations. 
This equivalence allows us to freely move between the worlds of indexed categories and fibred categories, providing access to tools and results from both. 
Due to its importance, it is only natural that one would be interested in possible extra structure these objects may have, and how the correspondence extends. 

The goal of this paper is to establish the appropriate correspondence in the monoidal setting. As the first benchmark, \cref{thm:mainthm} accomplishes this by lifting the standard equivalence $\ICat \simeq \Fib$ induced by the Grothendieck construction to an equivalence between the pseudomonoids in each 2-category. 
Using 2-categorical machinery, we obtain a canonical correspondence between \emph{monoidal fibrations} (fibrations which are strict monoidal functors with a cartesian lifting condition on the domain tensor product functor) and \emph{monoidal indexed categories} (lax monoidal pseudofunctors into $\Cat$). 
The monoidal Grothendieck construction in this sense employs the monoidal structure of the pseudofunctor to equip the corresponding total category with a monoidal product, which is strictly preserved by the fibration.

On a different but highly related note, Shulman introduced monoidal fibrations in \cite{FramedBicats} where he also explicitly constructed an equivalence between monoidal fibrations over a cartesian monoidal base and ordinary pseudofunctors into $\MonCat$; the latter were already called \emph{indexed (strong) monoidal categories} in \cite{DescentForMonads}. 
For this result, the existence of finite products was instrumental, making it impossible to extend it to arbitrary monoidal products. 
Moreover, the involved monoidal fibrations have monoidal fibre categories and strong monoidal reindexing functors between them, which is certainly not always the case for an arbitrary monoidal fibration. 

This striking dissimilarity between Shulman's equivalence and the one established here motivated an investigation regarding a `fibrewise' monoidal structure of a fibration as opposed to a `global' one. We  show that from a high level perspective, these structures are encompassed as pseudomonoids in different monoidal 2-categories: fixed-base fibrations $\Fib(\X)$ and arbitrary-base fibrations $\Fib$, see \cref{clarifyingdiag}. This crucial observation at that stage implies that these two distinct versions only meet when the base category has a (co)cartesian monoidal structure. Our key result and its 2-categorical proof, \cref{thm:fibrewise=global}, positions the general monoidal versus the special (co)cartesian case in their proper setting, and expresses an unforeseen bijection between ordinary pseudofunctors into $\MonCat$ and lax monoidal pseudofunctors into $(\Cat, \times, \1)$, \cref{lem:helplemma}.

This stimulating subtlety concerning the transfer of monoidality from the target category to the very structure of the functor and vice versa could potentially bring new perspective into further variations of the Grothendieck construction. As an example, in \cite{BeardsleyWong} the authors work towards a `fibrewise' enriched correspondence: under certain assumptions on a monoidal category $\catname{V}$, they establish a bijection between ordinary functors into $\catname{V}$-$\Cat$ and $\catname{V}$-functors over the free $\catname{V}$-category on the base. Future work could address the `global' enriched Grothendieck construction, namely one including an enriched (rather than ordinary) functor into $\Cat$ --- similarly to what is therein called \emph{fully} enriched correspondence --- and there is evidence that the monoidal correspondence of our current framework in fact underlies it.

Finally, the fact that the monoidal Grothendieck construction naturally arises in diverse settings is what motivated its theoretical clarification, thoroughly presented in this work. We gather a few examples in the last section of the paper so as to exhibit the various constructions concretely, and we are convinced that many more exist and would benefit from such a viewpoint. The examples include standard (op)fibrations like (co)domain and families classified in their monoidal contexts, as well as certain special algebraic cases of interest such as monoid-(co)algebras as objects in monoidal Grothendieck categories. For the special case of graphs, the monoidal Grothendieck correspondence is vigorously used to explicitly relate two distinct categorical frameworks for network theory, namely \emph{decorated cospans} and \emph{network models} via \cref{thm:decoratorsvsnetworkmodels}. Moreover, global categories of (co)modules for (co)monoids in any monoidal category, as well as (co)modules for (co)monads in monoidal double categories also naturally fit in this context. Finally, certain categorical approaches to systems theory employ algebras for monoidal categories, namely monoidal indexed categories, as their basic compositional tool for nesting of systems; clearly these also fall into place, giving rise to total monoidal categories of systems with new potential to be explored.

\subsection*{Outline of the paper}

In \cref{sec:preliminaries}, we review the basic theory of fibrations and indexed categories, as well as that of monoidal 2-categories and pseudomonoids. \cref{sec:mongroth} contains the eponymous construction in the form of 2-equivalences between the respective 2-categories of monoidal objects: \cref{sec:monfib,sec:monicat} contains elementary descriptions of (braided/symmetric) monoidal variations of fibrations and indexed categories, whereas \cref{sec:monequiv} details the relevant correspondences. In \cref{sec:fibrewisemonoidal}, we investigate the relation between the `global' and `fibrewise' monoidal Grothendieck construction for cartesian bases.
Finally, \cref{sec:applications} highlights some examples of this construction as it arises in various contexts, and \cref{sec:appendix} presents some of the earlier structures in greater detail.

\subsection*{Acknowledgements}

We would like to thank John Baez for invaluable guidance, as well as Mitchell Buckley, Nick Gurski, Claudio Hermida, Tom Leinster, Ignacio L\'opez Franco, Jade Master, Mike Shulman, Ross Street, David Spivak,  and Christian Williams for various helpful conversations. We also thank the careful reviewer for their extremely helpful comments; in particular, some non-trivial fixes on the split version of the correspondence are due to them. 
{CV} would like to thank the General Secretariat for Research and Technology (GSRT) and the Hellenic Foundation for Research and Innovation (HFRI).

\section{Preliminaries}\label{sec:preliminaries}

We assume familiarity with the basics of monoidal categories, see e.g.\ \cite{BraidedTensorCats}, as well as 2-category theory, see e.g.\ \cite{Review,2-catcompanion}.
We denote by $\TCat$ the paradigmatic example of a 3-category \cite{Gurskitricats} which consists of 2-categories, 2-functors, 2-natural transformations and modifications between them. If we take pseudofunctors $\ps{F} \colon \K \to \L$ between 2-categories, i.e.\ assignments that preserve the composition and identities up to coherent isomorphism, along with pseudonatural transformations $\ps{F} \Rightarrow \ps{G}$ between them, i.e.\ with components for which the usual naturality squares commute only up to coherent isomorphism, we obtain a tricategory denoted by $\TCat_{ps}$.

\subsection{Fibrations and Indexed Categories}\label{sec:fibrations}

We recall some basic facts and constructions from the theory of fibrations and indexed categories, as well as the equivalence between them via the Grothendieck construction. A few indicative references for the general theory are \cite{Grayfibredandcofibred, FibredAdjunctions, Handbook2, Jacobs, Elephant1}.

Consider a functor $P \maps \A \to \X$. A morphism $\phi \maps a \to b$ in $\A$ over a morphism $f = P(\phi) \maps x \to y$ in $\X$ is called \define{cartesian} if and only if, for all $g \maps x' \to x$ in $\X$ and $\theta \maps a'\to b$ in $\A$ with $P \theta = f \circ g$, there exists a unique arrow $\psi \maps a'\to a$ such that $P \psi = g$ and $\theta = \phi \circ \psi$:
\begin{displaymath}
    \xymatrix @R=.1in @C=.6in
    {a'\ar [drr]^-{\theta}\ar @{-->}[dr]_-{\exists!\psi} 
    \ar @{.>}@/_/[dd] &&& \\
    & a\ar[r]_-{\phi} \ar @{.>}@/_/[dd] & 
    b \ar @{.>}@/_/[dd] & \textrm{in }\A \\
    x'\ar [drr]^-{f\circ g=P\theta}\ar[dr]_-g &&&\\
    & x\ar[r]_-{f=P\phi} & y & \textrm{in }\X}
\end{displaymath}
For $x \in \ob\X$, the \define{fibre of $P$ over $x$} written $\A_x$, is the subcategory of $\A$ which consists of objects $a$ such that $P(a) = x$ and morphisms $\phi$ with $P(\phi) = 1_x$, called \define{vertical} morphisms. The functor $P \maps \A \to \X$ is called a \define{fibration} if and only if, for all $f \maps x \to y$ in $\X$ and $b\in\A_y$, there is a cartesian morphism $\phi$ with codomain $b$ above $f$; it is called a \define{cartesian lifting} of $b$ along $f$. The category $\X$ is then called the \define{base} of the fibration, and $\A $ its \define{total category}.

Dually, the functor $U \maps \C \to \X$ is an \define{opfibration} if $U^\mathrm{op}$ is a fibration, i.e.\ for every $c \in \C _x$ and $h \maps x \to y$ in $\X$, there is a cocartesian morphism with domain $c$ above $h$, the \define{cocartesian lifting} of $c$ along $h$ with the dual universal property:
\begin{displaymath}
\xymatrix @R=.1in @C=.6in
{&& d'\ar @{.>}@/_/[dd] &&\\
c\ar[r]_-{\beta} \ar @{.>}@/_/[dd]
\ar[urr]^-{\gamma} & 
d \ar @{.>}@/_/[dd] \ar @{-->}[ur]_-{\exists! \delta}
&& \textrm{in }\C\\
&& y' &&\\
x\ar[r]_-{h=U\beta} \ar[urr]^-{k\circ h=U\gamma}
 & y \ar[ur]_-k && \textrm{in }\X}
\end{displaymath}
A \define{bifibration} is a functor which is both a fibration and opfibration.

If $P\maps \A \to\X$ is a fibration, assuming the axiom of choice we may select a cartesian arrow over each $f\maps x\to y$ in $\X$ and $b\in\A _y$, denoted by $\Cart(f,b)\maps f^*(b)\to b$. Such a choice of cartesian liftings is called a \define{cleavage} for $P$, which is then called a \define{cloven fibration}; any fibration is henceforth assumed to be cloven. Dually, if $U$ is an opfibration, for any $c\in\C _x$ and $h \maps x \to y$ in $\X$ we can choose a cocartesian lifting of $c$ along $h$, $\Cocart(h,c)\maps c\longrightarrow h_!(c)$. 
The choice of (co)cartesian liftings in an (op)fibration induces a so-called \define{reindexing functor} between the fibre categories
\begin{equation}\label{reindexing}
    f^*\maps \A _y\to\A _x\quad\textrm{ and }\quad h_!\colon\C _x\to\C _y
\end{equation}
respectively, for each morphism $f\maps x\to y$ and $h\colon x\to y$ in the base category.
It can be verified by the (co)cartesian universal property that $1_{\A _x}\cong(1_x)^*$ 
and that for composable morphism in the base category, $g^*\circ f^*\cong(g\circ f)^* $, as well as $(1_x)_!\cong1_{\C_x}$ and $(k\circ h)_!\cong k_!\circ h_!$. If these isomorphisms are equalities, we have the notion of a \define{split} (op)fibration.

A \define{fibred 1-cell} $(H,F) \maps P \to Q$ between fibrations $P \maps \A \to \X$ and $Q \maps \B \to \Y$ is given by a commutative square of functors and categories \begin{equation}\label{commutativefibredcell}
    \xymatrix @C=.4in @R=.4in
    {\A \ar[r]^-H \ar[d]_-P &
    \B \ar[d]^-Q \\
    \X\ar[r]_-F &
    \Y}
\end{equation}
where the top $H$ preserves cartesian liftings, meaning that if $\phi$ is $P$-cartesian, then $H\phi$ is $Q$-cartesian. In particular, when $P$ and $Q$ are fibrations over the same base category, we may consider fibred 1-cells of the form $(H,1_{\X})$ displayed by
\begin{equation}\label{eq:fibredfunctor}
    \xymatrix @C=.2in
    {\A \ar[rr]^-H \ar[dr]_-P
    && \B \ar[dl]^-Q\\
    & \X &}
\end{equation}
and $H$ is then called a \define{fibred functor}. Dually, we have the notion of an \define{opfibred 1-cell} and \define{opfibred functor}. 
Notice that any such (op)fibred 1-cell induces functors between the fibres, by commutativity of \cref{commutativefibredcell}:
\begin{equation}\label{eq:functorbetweenfibres}
  H_{x}\colon\A_x\longrightarrow\B_{Fx}
\end{equation}

A \define{fibred 2-cell} between fibred 1-cells $(H,F)$ and $(K,G)$ is a pair of natural transformations ($\beta\maps H\Rightarrow K,\alpha\maps F\Rightarrow G$) with $\beta$ above $\alpha$, i.e.\ $Q(\beta_a)=\alpha_{Pa}$ for all $a\in\A $, displayed as
\begin{equation}\label{eq:fibred2cell}
    \xymatrix @C=.8in @R=.5in
    {\A \rtwocell^H_K{\beta}\ar[d]_-P
    & \B \ar[d]^-Q \\
    \X\rtwocell^F_G{\alpha} & \Y.}
\end{equation}
A \define{fibred natural transformation} is of the form $(\beta,1_{1_{\X}})\maps(H,1_{\X})\Rightarrow(K,1_\X)$
\begin{equation}\label{eq:fibrednaturaltrans}
    \xymatrix @C=.3in @R=.4in
    {\A \rrtwocell^H_K{\beta}\ar[dr]_-P
    && \B \ar[dl]^-Q\\
    & \X &}
\end{equation}
Dually, we have the notion of an \define{opfibred 2-cell} and \define{opfibred natural transformation} between opfibred 1-cells and functors respectively.

We thus obtain a 2-category $\Fib$ of fibrations over arbitrary base categories, fibred 1-cells and fibred 2-cells. There is also a 2-category $\Fib(\X)$ of fibrations over a fixed base category $\X$, fibred functors and fibred natural transformations. Dually, we have the 2-categories $\OpFib$ and $\OpFib(\X)$. Moreover, we also have 2-categories $\Fib_\spl$ and $\OpFib_\spl$ of split (op)fibrations, and (op)fibred 1-cells that preserve the cartesian liftings `on the nose'.

\begin{rmk}\label{rem:Fibisfibred}
Notice that $\Fib$ and $\OpFib$ are both sub-2-categories of $\Cat^\2 = [\2, \Cat]$, the arrow 2-category of $\Cat$. Similarly, $\Fib(\X)$ and $\OpFib(\X)$ are sub-2-categories of $\Cat/\X$, the slice 2-category of functors into $\X$. In fact,
both these categories form fibrations themselves \cite{hermida1999some}, see also \cref{prop:FibICatglobalfibrewise} and \cref{sec:fundamentalfib}; explicitly, $\cod\colon\Fib\to\Cat$ maps a fibration to its base. The \emph{2-fibration} structure is also explained in \cite[2.3.8]{2Fibs}.
\end{rmk}

We now turn to the world of indexed categories. Given an ordinary category $\X$, an $\X$-\define{indexed category} is a pseudofunctor \[\M \maps \X^\op \to \Cat\] where $\X$ is viewed as a 2-category with trivial 2-cells; it comes with natural isomorphisms $\delta_{g,f} \colon (\M g) \circ (\M f) \simrightarrow \M(g \circ f)$ and $\gamma_x \colon  1_{\M x} \simrightarrow \M (1_x)$ for every $x\in\X$ and composable morphisms $f$ and $g$, satisfying coherence axioms.
Dually, an $\X$-\define{opindexed category} is an $\X^\op$-indexed category, i.e.\ a pseudofunctor $\X \to \Cat$.
If an (op)indexed category strictly preserves composition, i.e.\ is a (2-)functor, then it is called \define{strict}.

An \define{indexed $1$-cell} $(F, \tau) \maps \M \to \N$ between indexed categories $\M \maps \X^\op \to \Cat$ and $\N \maps \Y^\op \to \Cat$ consists of an ordinary functor $F \maps \X \to \Y$ along with a pseudonatural transformation $\tau \maps \M \Rightarrow \N \circ F^\op$
\begin{equation}\label{eq:indexed1cell}
\begin{tikzcd}[column sep=.7in,row sep=.2in]
\X^\op\ar[dr, "\M"]\ar[dd, "F^\op"'] \\ \ar[r,phantom,"\Downarrow{\scriptstyle\tau}"description] & \Cat \\
\Y^\op \ar[ur,"\N"']
\end{tikzcd}
\end{equation}
with components functors $\tau_x\colon\M x\to\N Fx$, equipped with coherent natural isomorphisms $\tau_f\colon(\N Ff)\circ\tau_x\simrightarrow\tau_y\circ(\M f)$ for any $f\colon x\to y$ in $\X$.
For indexed categories with the same base, we may consider indexed 1-cells of the form $(1_\X, \tau)$
\begin{equation}\label{eq:ifun}
\begin{tikzcd}[column sep=.7in]
    \X^\op 
    \arrow[r, bend left, "\M"]
    \arrow[r, bend right, swap, "\N"]
    \arrow[r, phantom, "\Downarrow \scriptstyle \tau"]
    &
    \Cat
\end{tikzcd}
\end{equation}
which are called \define{indexed functors}.
Dually, we have the notion of an \define{opindexed 1-cell} and \define{opindexed functor}.

An \define{indexed 2-cell} $(\alpha,m)$ between indexed 1-cells $(F, \tau)$ and $(G, \sigma)$, pictured as 
\begin{displaymath}
 \begin{tikzcd}[column sep=.6in,row sep=.4in]
    \X^\op\ar[drr, "\M"]
    \ar[drr, ""{name = M}, swap, pos = 0.5]
     \ar[dd,bend right=40,"F^\op" description]
     \ar[dd, "G^\op"description, bend left=40]
    \arrow[dd,phantom,"{\scriptstyle\stackrel{\alpha^\op}{\Leftarrow}}"description] &&
    \\
& & \Cat
    \\
    \Y^\op
    \arrow[urr,"\N",swap]
    \arrow[urr,""{name = H}, pos = 0.4]
    \arrow[from = M, to = H, Rightarrow, "\sigma", pos = 0.4, bend left=45]
    \arrow[from = M, to = H, Rightarrow, "\tau", pos = 0.56, bend right=45,swap]
    \arrow[from = M, to = H, phantom, "\scriptscriptstyle\stackrel{m}{\Rrightarrow}"] &&
 \end{tikzcd} 
\end{displaymath}
consists of an ordinary natural transformation $\alpha \maps F \Rightarrow G$ and a modification $m$
\begin{equation}\label{eq:indexed2cell}
\begin{tikzcd}[column sep=.5in,row sep=.2in]
    \X^\op
    \ar[rr,bend left,"\M"]
    \ar[dr,bend right=10,"F^\op"']
    \ar[rr,phantom,"\Downarrow{\scriptstyle \tau}"description]  
    && 
    \Cat \ar[r,phantom,"\stackrel{m}{\Rrightarrow}"description] 
    & 
    \X^\op
    \ar[rr,bend left=30,"\M"]
    \ar[dr,bend left=35,"G^\op"]
    \ar[dr,bend right=35,"F^\op"']
    \ar[rr, phantom, near end, "\Downarrow{\scriptstyle \sigma}"description]
    \ar[dr,phantom, "\Downarrow{\scriptstyle \alpha^\op}"description]
    && 
    \Cat 
    \\
    & 
    \Y^\op
    \ar[ur,bend right=10,"\N"'] 
    &&& 
    \Y^\op
    \ar[ur,bend right,"\N"'] &
\end{tikzcd}
\end{equation}
given by a family of natural transformations $m_x \colon\tau_x \Rightarrow\N \alpha_x \circ\sigma_x$. Notice that taking opposites is a 2-functor $(-)^\op\colon\Cat\to\Cat^{co}$, on which the above diagrams rely.
An \define{indexed natural transformation} between two indexed functors is an indexed 2-cell of the form $(1_{1_\X},m)$.
Dually, we have the notion of an \define{opindexed 2-cell} and \define{opindexed natural transformation} between opindexed 1-cells and functors respectively.

Notice that an indexed 2-cell $(\alpha,m)$ is invertible if and only if both $\alpha$ is a natural isomorphism and the modification $m$ is invertible, due to the way vertical composition is formed. 

We obtain a 2-category $\ICat$ of indexed categories over arbitrary bases, indexed 1-cells and indexed 2-cells. In particular, there is a 2-category $\ICat(\X)$ of indexed categories with fixed domain $\X$, indexed functors and indexed natural transformations, which coincides with the functor 2-category $\TCat_\pse(\X^\op,\Cat)$.

Dually, we have the 2-categories $\OpICat$ and $\OpICat(\X)=\TCat_\pse(\X,\Cat)$. Notice that due to the absence of opposites in the world of opindexed categories, opindexed 2-cells have a different form than \cref{eq:indexed2cell}, namely
\begin{displaymath}
\begin{tikzcd}[column sep=.5in,row sep=.15in]
    \X
    \ar[rr, bend left=30, "\M"]
    \ar[dr, bend left=30, "F"]
    \ar[dr, bend right=30, "G"']
    \ar[rr, phantom, near end, "\Downarrow{\scriptstyle\tau}"description]
    \ar[dr, phantom, "\Downarrow{\scriptstyle\alpha}"description]  
    && 
    \Cat 
    \ar[r, phantom, "\stackrel{m}{\Rrightarrow}"description] 
    & 
    \X
    \ar[rr, bend left, "\M"]
    \ar[dr, bend right=10, "G"']
    \ar[rr, phantom, "\Downarrow{\scriptstyle \sigma}"description]
    && 
    \Cat 
    \\& 
    \Y
    \ar[ur,bend right,"\N"'] 
    &&& 
    \Y
    \ar[ur,bend right=10,"\N"'] &
\end{tikzcd}
\end{displaymath}
Moreover, we have 2-categories of strict (op)indexed categories and (op)indexed 1-cells that consist of strict natural transformations $\tau$ \cref{eq:indexed1cell}, i.e.\ $\ICat_\spl(\X)=[\X^\op,\Cat]$ and $\OpICat_\spl(\X)=[\X,\Cat]$ the usual functor 2-categories.

\begin{rmk}\label{rem:ICatisfibred}
    Similarly to \cref{rem:Fibisfibred}, notice that these (1-)categories also form fibrations over $\Cat$, this time essentially using the family fibration also seen in \cref{sec:familyfib}. The functor $\ICat\to \Cat$ is a split fibration that sends an indexed category to its domain and an indexed 1-cell to its first component as in \cref{prop:FibICatglobalfibrewise}. In fact, it is also a 2-fibration as explained in \cite[2.3.2]{2Fibs}.
\end{rmk}

In the first volume of the \emph{S\'eminaire de G\'eom\'etrie Alg\'ebrique du Bois Marie}  \cite{Grothendieckcategoriesfibrees}, Grothendieck introduced a construction for a fibration $P_\M \maps \inta \M \to \X$ from a given indexed category $\M \maps \X^\op \to \Cat$ as follows. If $\delta$ and $\gamma$ are the structure pseudonatural transformations of the pseudofunctor $\M$,
the total category $\inta \M$ has
\begin{itemize}
    \item objects $(x,a)$ with $x \in \X$ and $a \in \M x$;
    \item morphisms $(f,k) \maps (x,a) \to (y,b)$ with $f \maps x \to y$ a morphism in $\X$, and $k \maps a \to (\M f)(b)$ a morphism in $\M x$;
    \item composition $(g, \ell) \circ (f, k)\colon (x,a)\to(y,b)\to(z,c)$ is given by $g \circ f\colon a\to b\to c$ in $\X$ and 
    \begin{equation}\label{eq:comp_intM}
     a\xrightarrow{k}(\M f)(b)\xrightarrow{(\M g)(\ell)}(\M g\circ\M f)(c)\xrightarrow{(\delta_{f,g})_c}\M(g\circ f)(c)\quad\textrm{in }\M x;
    \end{equation}
    \item unit $1_{(x,a)}\colon (x,a)\to(x,a)$ is given by $1_x\colon x\to x$ in $\X$ and \[a=1_{\M x}a\xrightarrow{(\gamma_x)_a}(\M 1_x)(a)\quad\textrm{in }\M x.\]
\end{itemize}

The fibration $P_\M \maps \inta \M \to \X$ is given by $(x,a) \mapsto x$ on objects and $(f,k) \mapsto f$ on morphisms, and the cartesian lifting of 
any $(y,b)$ in $\inta \M$ along $f\colon x\to y$ in $\X$ is precisely $(f,1_{(\M f)b})$. Its fibres are precisely $\M x$ and the reindexing functors between them are $\M f$.

In the other direction, given a (cloven) fibration $P \maps \A \to \X$, we can define an indexed category $\M_P \maps \X^\op \to \Cat$ that sends each object $x$ of $\X$ to its fibre category $\A_x$, and each morphism $f \maps x \to y$ to the corresponding reindexing functor $f^* \maps \A_y \to \A_x$ as in \cref{reindexing}. The isomorphisms of cartesian liftings $f^* \circ g^* \cong (g \circ f)^*$ and $1_{\A_x} \cong 1_x^*$  render this assignment pseudofunctorial.

Details of the above, as well as the correspondence between 1-cells and 2-cells can be found in the provided references. Briefly, given a pseudonatural transformation $\tau\colon\M\to\N\circ F^\op$ \cref{eq:indexed1cell} with components $\tau_x\colon\M x\to\N Fx$, define a functor $P_\tau\colon\inta\M\to\inta\N$ 
mapping $(x\in\X,a\in\M x)$ to the pair $(Fx\in\Y,\tau_x(a)\in\N Fx)$ and accordingly for arrows. This makes the square
\begin{equation}\label{eq:inducedfibred1cell}
\begin{tikzcd}
\inta\M\ar[r,"P_\tau"]\ar[d,"P_\M"'] & \inta\N\ar[d,"P_\N"] \\
\X\ar[r,"F"'] & \Y
\end{tikzcd}
\end{equation}
commute, and moreover $P_\tau$ preserves cartesian liftings due to pseudonaturality of $\tau$. Moreover, given an indexed 2-cell $(\alpha,m)\colon(F,\tau)\Rightarrow(G,\sigma)$ as in \cref{eq:indexed2cell}, we can form a fibred 2-cell
\begin{equation}\label{eq:inducedfibred2cell}
\begin{tikzcd}[column sep=.8in,row sep=.6in]
\inta\M\ar[r,bend left,"P_\tau"]\ar[r,bend right,"P_\sigma"']\ar[r,phantom,"\Downarrow{\scriptstyle P_m}"description]\ar[d,"P_\M"'] & \inta\N\ar[d,"P_\N"] \\
\X\ar[r,bend left,"F"]\ar[r,bend right,"G"']\ar[r,phantom,"\Downarrow{\scriptstyle\alpha}"description] & \Y
\end{tikzcd}
\end{equation}
where $\alpha\colon F\Rightarrow G$ is piece of the given structure, whereas $P_m$ is given by components 
\[(P_m)_{(x,a)}\colon P_\tau(x,a)=(Fx,\tau_xa)\to P_\sigma(x,a)=(Gx,\sigma_xa)\quad\textrm{in } \inta \N\]
explicitly formed by $\alpha_x\colon Fx\to Gx$ in $\Y$ and $(m_x)_a\colon\tau_xa\to(\N\alpha_x)\sigma_xa$ in $\N Fx$.

The following theorem summarizes these standard results.

\begin{thm}\label{thm:Grothendieck}
    \leavevmode
    \begin{enumerate}
        \item Every fibration $P \maps \A \to \X$ gives rise to a pseudofunctor $\M_P \maps \X^\op \to\Cat$.
        \item Every indexed category $\M \maps \X^\op \to \Cat$ gives rise to  a fibration $P_\M \colon \inta\M\to\X$.
        \item The above correspondences yield an equivalence of 2-categories 
        \begin{displaymath}
            \ICat(\X) \simeq \Fib(\X)
        \end{displaymath}
        so that $\M_{P_\M} \cong \M$ and $P_{\M_P} \cong P$.
        \item The above 2-equivalence extends to one between 2-categories of arbitrary-base fibrations and arbitrary-domain indexed categories
        \begin{displaymath}
        \ICat\simeq\Fib    
        \end{displaymath}
    \end{enumerate}
\end{thm}

If we combine the above with \cref{rem:Fibisfibred} and \cref{rem:ICatisfibred} which point out that the 2-categories $\Fib$ and $\ICat$ are fibred over $\Cat$ with fibres $\Fib(\X)$ and $\ICat(\X)$ respectively, we obtain the following $\Cat$-fibred equivalence
\begin{displaymath}
\begin{tikzcd}
    \ICat
    \arrow[rr, "\simeq"]
    \arrow[dr]
    &&
    \Fib
    \arrow[dl]
    \\&
    \Cat
\end{tikzcd}
\end{displaymath}
There is an analogous story for opindexed categories and opfibrations that results
into a 2-equivalences $\OpICat(\X)\simeq\OpFib(\X)$ and $\OpICat\simeq\OpFib$, as well as for the split versions of (op)indexed and (op)fibred categories.

\subsection{Monoidal 2-categories and pseudomonoids}\label{sec:Monoidal2cats}

Below we sketch some basic definitions and constructions relative to monoidal 2-categories, necessary for what follows; relevant references where explicit axioms can be found are~\cite{Carmody, CoherenceTricats, Monoidalbicatshopfalgebroids, Mccruddencoalgebroids,Gurskitricats}.

A \define{monoidal bicategory} $\K$ comes equipped with a pseudofunctor
$\otimes \colon \K \times \K \to \K$ and a unit object $I \colon \mathbf{1} \to \K$ which are associative and unital up to coherent equivalence; a \define{monoidal 2-category} is one whose underlying bicategory is really a 2-category. In fact any monoidal bicategory is monoidally biequivalent to a special monoidal 2-category, namely a Gray monoid, by the one-object case of coherence for tricategories. A \define{2-monoidal 2-category} is a monoidal 2-category whose tensor product is a 2-functor rather than a pseudofunctor. Although our main examples in this paper are of the latter kind, namely the cartesian monoidal 2-categories of fibrations and indexed categories, the weaker structures are also required as they arise in the monoidal variants of the Grothendieck construction.

A \define{(weakly) lax monoidal pseudofunctor}
$\ps{F} \colon \K \to \L$ (called \emph{weak monoidal} in \cite{Monoidalbicatshopfalgebroids}) between monoidal 2-categories is a pseudofunctor equipped with pseudonatural transformations
\begin{equation}\label{eq:weakmonpseudo}
    \begin{tikzcd}[row sep=.2in]
        \K \times \K \ar[dr,phantom,"\Downarrow\scriptstyle\mu"]\ar[r,"\ps{F} \times \ps{F}"] \ar[d,"\ot_\K"']
        & |[alias=doma]| \L \times \L \ar[d,"\ot_\L"] \\
        |[alias=coda]| \K \ar[r,"\ps{F}"'] & \L
    \end{tikzcd}
    \quad\quad
    \begin{tikzcd}[row sep=.2in]
        \1\ar[d,"I_{\K}"']\ar[dr,phantom,bend right=2,"\Downarrow\scriptstyle\mu_0"]\ar[dr,bend left,"I_{\L}"{name=doma}] \\
        |[alias=coda]| \K\ar[r,"\ps{F}"'] & \L
    \end{tikzcd} 
\end{equation}
with components $\mu_{a,b} \colon \ps{F} a \otimes \ps{F} b \to \ps{F} (a \otimes b)$, $\mu_0 \colon I \to \ps{F} I$, along with invertible modifications
\begin{displaymath}
\scalebox{0.85}{
\begin{tikzcd}[column sep = 15,ampersand replacement=\&]
\& \L{\times}\L{\times}\L \arrow[d, phantom, "\Downarrow{\scriptstyle\mu{\times}1}"]
\arrow[r, "\otimes_\L{\times}1"] \&
\L{\times}\L
    \arrow[dd, phantom, "\Downarrow{\scriptstyle\mu}"]
    \arrow[dr, "\otimes_\L"]
    \&\&\&
\L{\times}\L{\times}\L
\arrow[dd, phantom, "\Downarrow{\scriptstyle 1{\times}\mu}"]
\arrow[r, "\otimes_\L{\times}1"]
\arrow[dr, "1{\times}\otimes_\L"description]
\& \L{\times}\L
\arrow[d, phantom, "{\scriptstyle\simeq}"]
\arrow[dr, "\otimes_\L"]
    \\
\K{\times}\K{\times}\K
    \arrow[ur, "\F{\times}\F{\times}\F"]
    \arrow[r, "\otimes_\K{\times}1"]
    \arrow[dr, "1{\times}\otimes_\K", swap]
    \&
    \K{\times}\K
    \arrow[ur, "\F{\times}\F"description]
    \arrow[dr, "\otimes_\K"description]
    \&\&
    \L
    \arrow[r, phantom, "\stackrel{\omega}{\Rrightarrow}"]
    \&
    \K{\times}\K{\times}\K
    \arrow[ur, "\F{\times}\F{\times}\F"]
    \arrow[dr, "1{\times}\otimes_\K", swap]
    \&\&
    \L{\times}\L
    \arrow[d, phantom, "\Downarrow{\scriptstyle\mu}"]
    \arrow[r, "\otimes_\L"]
    \&
    \L
    \\ \&
    \K{\times}\K
    \arrow[u, phantom, "{\scriptstyle\simeq}"]
    \arrow[r, "\otimes_\K", swap]
    \&
\K \arrow[ur, "\F", swap]
    \&\&\&
    \K{\times}\K
    \arrow[ur, "\F{\times}\F"description]
    \arrow[r, "\otimes_\K", swap]
    \&
    \K
    \arrow[ur, "\F", swap]
\end{tikzcd}}
\end{displaymath}
\begin{displaymath}
\scalebox{0.85}{\begin{tikzcd}[column sep=.25in,ampersand replacement=\&]
    \K
    \ar[dr, "1\times I"']
    \ar[rr, "\F\times I"{name=doma}]
    \ar[rrd, bend right=70, "1"', "\simeq"]
    \ar[rrr, bend left=30, "\F"]
    \ar[rrr, phantom, bend left=15, "{\scriptstyle\simeq}"description] 
    \&\& 
    \L\times\L
    \ar[r, "\otimes_\L"'] 
    \& 
    \L
    \arrow[Rightarrow, from=doma, to=coda, "1\times\mu_0\;\;"', shorten <=.5em, shorten >=.5em]
    \\ \& 
    |[alias=coda]|\K\times\K
    \ar[r, "\otimes_\K"']
    \ar[ur, "\F\times\F"description]
    \ar[urr, phantom, "\Downarrow{\scriptstyle \mu}"description] 
    \&
    \K
    \ar[ur, "\F"'] 
    \&
\end{tikzcd}}
\stackrel{\zeta}{\Rrightarrow}
\scalebox{0.85}{\begin{tikzcd}[column sep=.25in,ampersand replacement=\&]
    \hole \\
    \K
    \ar[rr,bend left,"\F"]
    \ar[rr,bend right,"\F"']
    \ar[rr,phantom,"\Downarrow{\scriptstyle1}"description] 
    \&\& 
    \L 
    \\
    \hole
\end{tikzcd}}
\stackrel{\xi}{\Rrightarrow}
\scalebox{0.85}{\begin{tikzcd}[column sep=.25in,ampersand replacement=\&]
    \K
    \ar[dr,"I\times1"']
    \ar[rr,"I\times\F"{name=doma}]
    \ar[rrd, bend right=70, "1"',"\simeq"]
    \ar[rrr,bend left=30,"\F"]
    \ar[rrr,phantom,bend left=15,"{\scriptstyle\simeq}"description] 
    \&\& 
    \L\times\L
    \ar[r,"\otimes_\L"'] 
    \& 
    \L
    \arrow[Rightarrow, from=doma, to=coda,"\mu_0\times1\;\;"',shorten <=.5em, shorten >=.5em]
    \\ \& 
    |[alias=coda]|\K\times\K
    \ar[r,"\otimes_\K"']
    \ar[ur,"\F\times\F"description]
    \ar[urr,phantom,"\Downarrow{\scriptstyle \mu}"description] 
    \&
    \K
    \ar[ur,"\F"'] 
    \&
\end{tikzcd}}
\end{displaymath}
with components
\begin{equation}\label{eq:omega}
\begin{tikzcd}[column sep=.8in, row sep=.15in]
(\F a\ot \F b)\ot\F c\ar[ddr,phantom,"{\stackrel{\omega_{a,b,c}}{\cong}}"]\ar[r,"\mu_{a,b}\ot1"]\ar[d,"\simeq"'] & \F (a\ot b)\ot\F c\ar[d,"\mu_{a\ot b,c}"] \\
\F a\ot( \F b\ot \F c)\ar[d,"1\ot\mu_{b,c}"'] & 
\F ((a\ot b)\ot c)\ar[d,"\simeq"] \\
\F a\ot\F(b\ot c)\ar[r,"\mu_{a,b\ot c}"'] &
\F (a\ot(b\ot c))
\end{tikzcd}
\end{equation}
\begin{displaymath}
\begin{tikzcd}[column sep=.3in,row sep=.15in]
\F a\ar[r,"\simeq"]\ar[drr,bend right=5,"\simeq"'] & \F a\otimes I\ar[r,"1\otimes\mu_0"]\ar[dr,phantom,"\stackrel{\zeta_a}{\cong}"] & \F a\otimes \F I\ar[d,"\mu_{a,I}"] \\
&& \F(a\ot I)
\end{tikzcd}\quad
\begin{tikzcd}[column sep=.3in,row sep=.15in]
\F a\ar[r,"\simeq"]\ar[drr,bend right=5,"\simeq"'] & I\otimes\F a \ar[r,"\mu_0\otimes1"]\ar[dr,phantom,"\stackrel{\xi_a}{\cong}"] & \F I\otimes \F a\ar[d,"\mu_{I,a}"] \\
&& \F(I\ot a)
\end{tikzcd}
\end{displaymath}
subject to coherence conditions. 

A \define{(weakly) monoidal pseudonatural transformation} $\tau \colon \ps{F} \Rightarrow \ps{G}$ between two lax monoidal pseudofunctors $(\F,\mu,\mu_0)$ and $(\G,\nu,\nu_0)$ is a pseudonatural transformation equipped with two invertible modifications
\begin{equation}\label{eq:monpseudonat}
\begin{tikzcd}[column sep = .15in, row sep = .1in]
   \K \times \K 
   \ar[rr, bend left, "\ps{F} \times \ps{F}"] 
   \ar[rr, bend right, "\ps{G} \times \ps{G}"'] 
   \ar[dd, "\otimes"'] 
   \ar[ddrr, phantom, bend right = 15, "{\scriptstyle\Downarrow \nu}"'] 
   \ar[rr, phantom, description, "{\scriptstyle\Downarrow \tau \times \tau}"] 
   && 
   \L \times \L
   \ar[dd, "\otimes"] \ar[ddr,phantom,"\stackrel{u}{\Rrightarrow}"]
   &
   \K \times \K 
   \ar[rr, "\F \times \F"]
   \ar[dd, "\otimes"']
   \ar[ddrr, phantom, bend left=15, "{\scriptstyle\Downarrow \mu}"] 
   && 
   \L \times \L 
   \ar[dd, "\otimes"] 
   \\ \hole
   \\
   \K
   \ar[rr, "\ps{G}"'] 
   && 
   \L 
   &
   \K
   \ar[rr, bend left, "\ps{F}"]
   \ar[rr, bend right, "\ps{G}"']
   \ar[rr, phantom, description, "{\scriptstyle\Downarrow\tau}"] 
   && 
   \L
\end{tikzcd}\qquad
\begin{tikzcd}[row sep = .1in, column sep = .2in]
   \1
    \ar[ddrr, phantom, "{\scriptstyle\Downarrow \nu_0}" description]
    \ar[rr, "I_\L"]
    \ar[dd, "I_\K"'] 
    && 
    \L\ar[ddr,phantom,bend right=5,"\stackrel{u_0}{\Rrightarrow}"]
    &
    \1
    \ar[ddrr, phantom,near start,bend left=10, "{\scriptstyle\Downarrow \mu_0}" description]
    \ar[rr, "I_\L"]
    \ar[dd, "I_\K"']
    && 
    \L 
    \\ \hole \\
   \K
    \ar[uurr, bend right = 40,
    "\G"']  
    &&\phantom{A}& 
     \K 
    \ar[uurr, "\F" description]
    \ar[uurr, bend right=60, "\G"']
    \ar[uurr, phantom, bend right, "{\scriptstyle \Downarrow\tau}" description]
    && 
    \phantom{A}
\end{tikzcd}
\end{equation}
that consist of natural isomorphisms with components
\begin{equation}\label{eq:monpseudocomponents}
u_{a,b}\colon\nu_{a,b} \circ (\tau_a \otimes \tau_b) \xrightarrow{\sim} \tau_{a \otimes b} \circ \mu_{a,b},\quad
u_0\colon\nu_0 \xrightarrow{\sim}\tau_I \circ \mu_0 
\end{equation}
satisfying coherence conditions.

A \define{monoidal modification}
between monoidal pseudonatural transformations $(\tau,u,u_0)$ and $(\sigma,v,v_0)$ is a modification
\begin{displaymath}
\begin{tikzcd}
\K\ar[rr,bend left=40,"\F",""'{name = F}]\ar[rr,bend right=40,"\G"',""{name = G}] \ar[rr,phantom,"\stackrel{m}{\Rrightarrow}"description] && \L
 \arrow[from = F, to = G, Rightarrow, "\tau"',bend right=50]
  \arrow[from = F, to = G, Rightarrow, "\sigma",bend left=50]
\end{tikzcd}
\end{displaymath}
which consists of pseudonatural transformations $m_a\colon\tau_a\Rightarrow\sigma_a$ compatible with the monoid\-al structures, in the sense that
\begin{equation}\label{eq:monoidalmodaxioms}
\scalebox{0.9}{\begin{tikzcd}[column sep=.27in,ampersand replacement=\&]
\& \G a\ot \G b\ar[dr,bend left,"\nu_{a,b}"] \&\&\&\G a\ot \G b \ar[dr,bend left,"\nu_{a,b}"] 
\\
\F a\ot \F b\ar[ur,bend left,"\tau_a\ot \tau_b"]\ar[dr,bend right,"\mu_{a,b}"']\ar[rr,phantom,near start,"\Downarrow{\scriptstyle u_{a,b}}"description] \&\& \G (x\ot y) \ar[r,phantom,"="description] \& \F a\ot \F b\ar[dr,bend right,"\mu_{a,b}"']\ar[ur,bend right,"\sigma_a\ot \sigma_b"']\ar[rr,phantom,near end,"\Downarrow{\scriptstyle v_{a,b}}"description]\ar[ur,bend left,"\tau_a\ot\tau_b"]\ar[ur,phantom,"\Downarrow{\scriptstyle m_a\ot m_b}"description] \&\& \G (a\ot b) 
\\
\& \F(a\ot b)\ar[ur,bend left,"\tau_{a\ot b}"]\ar[ur,bend right,"\sigma_{a\ot b}"']\ar[ur,phantom,"\Downarrow{\scriptstyle m_{a\ot b}}"description] \&\&\& \F(a\ot b)\ar[ur,bend right,"\sigma_{a\ot b}"']
\end{tikzcd}}
\end{equation}
\begin{displaymath}
\scalebox{0.9}{\begin{tikzcd}[row sep=.2in,ampersand replacement=\&]
I\ar[dr,bend right,"\mu_0"']\ar[rr,bend left,"\nu_0"]\ar[rr,phantom,near start,"\Downarrow{\scriptstyle v_0}"description] \&\& \G(I)\ar[r,phantom,"="description] \& I\ar[rr,phantom,"\Downarrow{\scriptstyle u_0}"description]\ar[dr,bend right,"\mu_0"']\ar[rr,bend left,"\nu_0"] \&\& \G(I) \\
\& \F(I)\ar[ur,bend left,"\tau_I"]\ar[ur,bend right,"\sigma_I"']\ar[ur,phantom,"\Downarrow{\scriptstyle m_I}"description] \&\&\& \F(I)\ar[ur,bend right,"\tau_I"'] \&
\end{tikzcd}}
\end{displaymath}

For any monoidal 2-categories $\K,\L$, these form a 2-category $\MonTCat_\pse(\K,\L)$; if $\K$ and $\L$ were monoidal bicategories, this bicategory is denoted by $\namedcat{WMonHom}(\K,\L)$ in \cite{Monoidalbicatshopfalgebroids} or $\mathsf{MonBicat}(\K,\L)$ in \cite{PeriodicII}. If we take weakly lax monoidal 2-functors, i.e. whose underlying pseudofunctors are 2-functors \cite{Mccruddencoalgebroids}, and weakly monoidal 2-transfor\-ma\-tions, the corresponding sub-2-category is denoted by $\MonTCat(\K,\L)$. This should not be confused with the stricter notions of lax monoidal 2-functors and monoidal 2-natural transformations, for which \cref{eq:weakmonpseudo} are 2-natural and \cref{eq:omega,eq:monpseudocomponents} are identities, corresponding to the monoidal $\Cat$-enriched definitions. Notice that in the later \cref{sec:monicat}, in the strict context we will come across an intermediate notion of a weakly lax monoidal 2-functor whose structure maps \cref{eq:weakmonpseudo} are strictly natural but \cref{eq:omega} are still isomorphisms.

A \define{pseudomonoid} (or \define{monoidale}) in a monoidal 2-category $(\K, \otimes, I)$ is an object $a$ equipped with multiplication $\mlt \maps a\ot a\to a$,  unit $\uni \maps I \to a$, and invertible 2-cells
\begin{equation}\label{alphalambdarho}
\begin{tikzcd}
    a \otimes a \otimes a
    \ar[r, "1 \otimes \mlt"]
    \ar[d, "\mlt \otimes 1"']
    \ar[dr, phantom, "\scriptstyle \stackrel{\alpha}{\cong}"description]
    & 
    a \otimes a \ar[d, "\mlt"] \ar[d, "\mlt"]
    & a \ar[r, "1 \otimes \uni"] \ar[dr,"1"'] 
    & a \otimes a \ar[d, "\mlt", "\stackrel{\lambda}{\cong}\quad"'near start, "\quad\;\stackrel{\rho}{\cong}"near start] 
    & a \ar[l, "\uni \otimes 1"']
    \ar[dl,"1"]  \\
    a \otimes a\ar[r, "\mlt"'] 
    & a && a &
\end{tikzcd}
\end{equation}
expressing assiociativity and unitality up to isomorphism, that satisfy appropriate coherence conditions; we omitted the associativity and unit constraints of $\K$.  A \define{lax morphism} between pseudomonoids $a, b$ is a 1-cell $f\colon a\to b$ equipped with 2-cells
    \begin{equation}\label{eq:laxmorphism}
    \begin{tikzcd}
        a \otimes a\ar[d,"\mlt"']\ar[dr,phantom,"\Downarrow\scriptstyle\phi"]\ar[r,"f\otimes f"] & |[alias=doma]| b\otimes b\ar[d,"\mlt"] \\
        |[alias=coda]| a\ar[r,"f"'] & b
    \end{tikzcd}
    \qquad
    \begin{tikzcd}
        I\ar[d,"\uni"']\ar[dr,bend left,"\uni"{name=doma}]\ar[dr,phantom,bend right=2,"\Downarrow\scriptstyle\phi_0"] \\
        |[alias=coda]| a\ar[r,"f"'] & b
    \end{tikzcd}
    \end{equation}
    such that the following conditions hold:
    \begin{equation}\label{eq:axiomslaxmorphism}
    \adjustbox{scale=.85,center}{
        \begin{tikzcd}[column sep=.3in]
            & b\ot b\ot b\ar[r,"\mlt\ot1"]\ar[d,phantom,"\Downarrow{\scriptstyle\phi\otimes1_f}"description] & b\ot b\ot a\ar[dr,"\mlt"] & \\
            a\otimes a\otimes a\ar[r,"\mlt\otimes1"]\ar[dr,"1\otimes\mlt"']\ar[ur,"f\otimes f\otimes f"] &
            a\otimes a\ar[dr,"\mlt"]\ar[ur,"f\otimes f"']\ar[rr,phantom,"\Downarrow{\scriptstyle\phi}"description]
            \ar[d,phantom,"\stackrel{\alpha}{\cong}"description] && b \\
            & a\ot a\ar[r,"\mlt"'] & a\ar[ur,"f"']
        \end{tikzcd} 
        =
        \begin{tikzcd}[column sep=.3in]
            & b\ot b\ot b\ar[r,"\mlt\ot1"]\ar[dr,"1\ot\mlt"']\ar[dd,phantom,"\Downarrow{\scriptstyle1_f\otimes\phi}"description] &
            b\ot b\ar[dr,"\mlt"]\ar[d,phantom,"\stackrel{\alpha}{\cong}"description] & \\
            a\otimes a\otimes a\ar[ur,"f\otimes f\otimes f"]\ar[dr,"1\otimes\mlt"'] & & b\otimes b\ar[r,"\mlt"]
            \ar[d,phantom,"\Downarrow{\scriptstyle\phi}"description] & b \\
            & a\ot a\ar[ur,"f\ot f"]\ar[r,"\mlt"'] & b\otimes b\ar[ur,"f"']
        \end{tikzcd}
    }
    \end{equation}
    
    \begin{displaymath}
    \adjustbox{scale=.85,center}{
        \begin{tikzcd}[column sep=.25in]
            {\phantom{abcd}}a\ar[dr,"1\otimes\uni"']\ar[rr,"f\otimes\uni"{name=doma}]\ar[rrd,bend right=60,"1_a"',"\lambda\cong"]
            \ar[rrr,bend left=30,"f"]\ar[rrr,phantom,bend left=15,"{\scriptstyle 1_f\ot\lambda\cong}"description] && b\otimes b\ar[r,"\mlt"'] & b
            \arrow[Rightarrow, from=doma, to=coda,"1_f\otimes\phi_0\;\;"',shorten <=.5em, shorten >=.5em]\\
            & |[alias=coda]|a\otimes a\ar[r,"\mlt"']\ar[ur,"f\otimes f"description]\ar[urr,phantom,"\Downarrow{\scriptstyle \phi}"description] &
            a\ar[ur,"f"'] &
        \end{tikzcd}
        =
        \begin{tikzcd}[column sep=.25in]
            \hole \\
            a\ar[rr,bend left,"f"]\ar[rr,bend right,"f"']\ar[rr,phantom,"\Downarrow{\scriptstyle1_f}"description] && b \\
            \hole
        \end{tikzcd}
        =
        \begin{tikzcd}[column sep=.25in]
            {\phantom{abcd}}a\ar[dr,"\uni\ot1"']\ar[rr,"\uni\otimes1"{name=doma}]\ar[rrd,bend right=60,"1_a"',"\rho\cong"]
            \ar[rrr,bend left=30,"f"]\ar[rrr,phantom,bend left=15,"{\scriptstyle \rho\ot1_f\cong}"description] && b\otimes b\ar[r,"\mlt"'] & b
            \arrow[Rightarrow, from=doma, to=coda,"\phi_0\otimes1_f\;\;"',shorten <=.5em, shorten >=.5em]\\
            & |[alias=coda]|a\otimes a\ar[r,"\mlt"']\ar[ur,"f\otimes f"description]\ar[urr,phantom,"\Downarrow{\scriptstyle \phi}"description] &
        a\ar[ur,"f"'] &
        \end{tikzcd}
        }
    \end{displaymath}

    If $(f,\phi,\phi_0)$ and $(g,\psi,\psi_0)$ are two lax morphisms between pseudomonoids $a$ and $b$, a \define{2-cell} between them $\sigma \colon f \Rightarrow g$ in $\K$ which is compatible with multiplications and units, in the sense that
    \begin{equation}\label{monoidal2cell}
        \begin{tikzcd}[row sep=.15in]
            & b\ot b\ar[dr,bend left=20,"\mlt"] & \\
            a\ot a\ar[ur,bend left,"f\ot f"]\ar[ur,bend right,"g\ot g"']\ar[dr,bend right=20,"\mlt"']
            \ar[ur,phantom,"\Downarrow{\scriptstyle\sigma\ot\sigma}"description]
            \ar[rr,phantom,"{\scriptstyle\psi}\Downarrow"{description,near end}] && b \\
            & a\ar[ur,bend right=20,"g"'] &
        \end{tikzcd}
        \quad=\quad
        \begin{tikzcd}[row sep=.15in]
            & b\ot b\ar[rd,bend left=20,"\mlt"] & \\
            a\ot a\ar[ur,bend left=20,"f\ot f"]\ar[dr,bend right=20,"\mlt"']
            \ar[rr,phantom,"{\scriptstyle\phi}\Downarrow"{description,near start}] && b \\
            & a \ar[ur,phantom,"\Downarrow{\scriptstyle\sigma}"description]
            \ar[ur,bend left,"f"]\ar[ur,bend right,"g"'] \\
        \end{tikzcd}
    \end{equation}
    \begin{displaymath}
        \begin{tikzcd}[row sep=.15in,column sep=.6in]
            I\ar[rr,bend left,"\uni"]\ar[dr,bend right=20,"\uni"']\ar[rr,phantom,"{\scriptstyle\phi_0}\Downarrow"{description,near start}] && b \\
            & a\ar[ur,bend left,"f"]\ar[ur,bend right,"g"']\ar[ur,phantom,"\Downarrow{\scriptstyle \sigma}"description] &
        \end{tikzcd}
        \quad=\quad
        \begin{tikzcd}[row sep=.15in,column sep=.6in]
            I\ar[dr,bend right=20,"\uni"']\ar[rr,bend left,"\uni"]\ar[rr,phantom,"\Downarrow{\scriptstyle\psi_0}"description] && b \\
            & a\ar[ur,bend right=20,"g"'] &
        \end{tikzcd}
    \end{displaymath} 

We obtain a 2-category $\PsMon_\lax(\K)$ for any monoidal 2-category $\K$, which is sometimes
denoted by $\Mon(\K)$ \cite{Hopfmonoidalcomonads}. By changing the direction of the 2-cells in \cref{eq:laxmorphism} and the rest of the axioms appropriately, or asking for them to be invertible, we have 2-categories $\PsMon_\opl(\K)$ and $\PsMon(\K)$ of pseudomonoids with \define{oplax} or \define{(strong) morphisms} between them.

\begin{examples}
    The prototypical example is that of the (2-)monoidal 2-category $\K = (\Cat, \times, \1)$ of categories, functors, and natural transformations with the cartesian product of categories and the unit category with a unique object and arrow. A pseudomonoid in $(\Cat,\times,\1)$ is a monoidal category, a lax (resp. oplax, strong) morphism between two of these is precisely a lax (resp. oplax, strong) monoidal functor, and a 2-cell is a monoidal natural transformation. Therefore we obtain the well-known 2-categories $\MonCat_\lax$, $\MonCat_\opl$ and $\MonCat$.
\end{examples}

\begin{rmk}\label{rem:pseudomonas2functors}
    There is an evident similarity between the structures defined above, e.g \cref{eq:weakmonpseudo} and \cref{eq:laxmorphism}, or \cref{eq:monpseudonat} and \cref{monoidal2cell}. This is due to the fact that monoidal 2-categories, lax monoidal pseudofunctors and monoidal pseudonatural transformations are themselves appropriate pseudomonoid-related notions in a higher level; we do not get into such details, as they are not pertinent to the present work.

For our purposes, we are interested in a different observation, see \cite{Monoidalbicatshopfalgebroids}: any pseudomonoid $a$ in a monoidal 2-category $\K$ can in fact be expressed as a (weakly) lax monoidal 2-functor $A \colon \1 \to \K$ with $A(*) = a$. Moreover, a monoidal 2-natural transformation $\tau\colon A\Rightarrow B\colon\1\to\K$ bijectively corresponds to a \emph{strong morphism} between the pseudomonoids $a$ and $b$, and similarly for monoidal modifications and 2-cells. Therefore the 2-category of pseudomonoids $\PsMon(\K)$ can be equivalently viewed as $\MonTCat(\1,\K)$, the 2-category of weakly lax monoidal 2-functors $\1\to\K$, weakly monoidal 2-natural transformations and monoidal modifications.
\end{rmk}

In what follows, we want to make precise the natural idea that if two monoidal 2-categories are equivalent, then their 2-categories of pseudomonoids are also equivalent.
As was already shown in~\cite[Prop.~5]{Monoidalbicatshopfalgebroids}, any lax monoidal 2-functor $\ps{F}\colon\K\to\L$ takes pseudomonoids to pseudomonoids, and in fact
\cite{Mccruddencoalgebroids} there is a functor $\PsMon(\ps{F})$ that commutes with the respective forgetful functors
\begin{displaymath}
    \begin{tikzcd}[column sep=.6in]
        \PsMon(\K)\ar[r,"\PsMon(\ps{F})"]\ar[d] & \PsMon(\L)\ar[d] \\
        \K\ar[r,"\ps{F}"'] & \L.
    \end{tikzcd}
\end{displaymath}
Based on the above \cref{rem:pseudomonas2functors}, we can define a hom-2-functor that clarifies these assignments. Its domain $\mathsf{Mon2Cat}$ is the 2-category of monoidal 2-categories, (weakly) lax monoidal 2-functors and (weakly) monoidal 2-natural transformations, which lives inside the more general tricategory $\mathsf{MonBicat}$ established in \cite{PeriodicII}. Although it may look surprising at first, $\mathsf{Mon2Cat}$ can be explicitly verified to be a 2-category rather than a more relaxed structure, using the ambient axioms and strictness of the underlying structures. 

\begin{prop}\label{prop:PsMon2functor}
There is a 2-functor
\begin{equation}\label{eq:PsMon}
\PsMon(-)\simeq\MonTCat(\1,-)\colon\MonTCat\to\TCat
\end{equation}
which maps a monoidal 2-category to its 2-category of pseudomonoids, strong morphisms and 2-cells between them.
\end{prop}

The theory in \cite{Monoidalbicatshopfalgebroids,Mccruddencoalgebroids} extends the above definitions to the case of \emph{braided} and \emph{symmetric} pseudomonoids in \emph{braided} and \emph{symmetric monoidal 2-categories}. Briefly recall that a \define{braiding} for $(\K,\ot,I)$ is a pseudonatural equivalence with components $\braid_{a,b}\colon a\ot b\to b\ot a$ and invertible modifications satisfying axioms, whereas a \define{syllepsis} is an invertible modification 
\[a\ot b\xrightarrow{1}a\ot b\Rrightarrow a\ot b\xrightarrow{\braid_{a,b}}b\ot a\xrightarrow{\braid_{b,a}}a\ot b\] satisfying axioms,
which is called \define{symmetry} if one extra axiom holds. With the appropriate notions of \emph{braided} and \emph{symmetric} lax monoidal pseudofunctors and monoidal pseudonatural transformations (and usual monoidal modifications), we have tricategories $\BrMonTCat_{\pse}$ and $\SymMonTCat_{\pse}$. 
For example, a braided lax monoidal pseudofunctor comes equipped an invertible modification with components
\begin{equation}\label{eq:brweakmonpseudo}
\begin{tikzcd}
\F a\otimes \F b\ar[r,"\mu_{a,b}"]\ar[d,"\braid_{\F a,\F b}"']\ar[dr,phantom,"\Downarrow{\scriptstyle v_{a,b}}"description] & \F(a\ot b)\ar[d,"\F(\braid_{a,b})"] \\
\F b\otimes\F a\ar[r,"\mu_{b,a}"'] & \F(b\ot a)
\end{tikzcd}
\end{equation}
As earlier, there exist 2-categories of braided and symmetric pseudomonoids with strong morphisms $\BrPsMon(\K)=\BrMonTCat(\1,\K)$ and $\SymPsMon(\K)=\SymMonTCat(\1,\K)$.

\begin{prop}\label{prop:BrPsMon}
    There are 2-functors
    \begin{displaymath}
    \BrPsMon\colon\BrMonTCat\to \TCat,\quad
    \SymPsMon\colon\SymMonTCat\to\TCat 
    \end{displaymath}
    which map a braided or symmetric monoidal 2-category to its 2-category of braided or symmetric pseudomonoids.
\end{prop}

Finally, there is a notion of a monoidal 2-equivalence arising as the equivalence internal to the 2-category $\MonTCat$.

\begin{defi}
    A \define{monoidal 2-equivalence} is a 2-equivalence $\ps{F}\colon\K\simeq\L:\ps{G}$ where both 2-functors
    are lax monoidal, and the 2-natural isomorphisms $1_\K \cong \ps{F} \ps{G}$, $\ps{G} \ps{F} \cong 1_\L$ are monoidal. Similarly for \define{braided} and \define{symmetric} monoidal 2-equivalences.
\end{defi}

As is the case for any 2-functor between 2-categories, $\PsMon$ as well as $\BrPsMon$ and $\SymPsMon$ map equivalences
to equivalences.

\begin{prop}\label{prop:2equivpseudomon}
    Any monoidal 2-equivalence $\K\simeq\L$ induces a 2-equivalence between the respective 2-categories
    of pseudomonoids $\PsMon(\K)\simeq\PsMon(\L)$. Similarly, any braided or symmetric monoidal 2-equivalence induces a 2-equivalence $\BrPsMon(\K)\simeq\BrPsMon(\L)$ or $\SymPsMon(\K)\simeq\SymPsMon(\L)$.
\end{prop}

It would be reasonable to have a corresponding statement for biequivalent Gray monoids and their pseudomonoids, by appropriately adjusting \cref{prop:PsMon2functor}. Due to our case of interest  -- fibrations and indexed categories -- being equivalent in this stronger sense, we do not need to work in that level of generality.

\section{The Monoidal Grothendieck Construction} \label{sec:mongroth}

In this section, we give explicit descriptions of the 2-categories of pseudomonoids in the cartesian monoidal 2-categories of fibrations and indexed categories, $\Fib$ and $\ICat$, and we exhibit their equivalence induced by the \emph{monoidal} Grothendieck construction. We also consider the fixed-base case, namely pseudomononoids in $\Fib(\X)$ and $\ICat(\X)$ and their corresponding equivalence. These two cases are in general distinct, and can be summarized in
\begin{equation}\label{clarifyingdiag}
\begin{tikzcd}[column sep=.01in,row sep=.3in]
    &
    \Fib{\simeq}\ICat
    \arrow[dl,"\PsMon(-)"']
    \arrow[dr,"\textrm{fix }\X"]
    \\
    \MonFib{\simeq}\MonICat
    \arrow[d,"\textrm{fix }\X"']
    &&
    \Fib(\X){\simeq}\TCat_\pse(\X^\op,\Cat)
    \arrow[d,"\PsMon(-)"]
    \\
    \MonFib(\X){\simeq}\MonTCat_\pse(\X^\op,\Cat) 
    &&
    \PsMon(\Fib(\X)){\simeq}\TCat_\pse(\X^\op, \MonCat)
\end{tikzcd}
\end{equation}
The two feet turn out to coincide only under certain hypotheses, which reveal some interesting subtleties on the potential monoidal structures on fibrations and pseudofunctors, discussed in detail in \cref{sec:fibrewisemonoidal}. 

\subsection{Monoidal Fibrations}\label{sec:monfib}

The 2-categories $\Fib$ and $\OpFib$ of (op)fibrations over arbitrary bases, recalled in \cref{sec:fibrations}, have a natural cartesian monoidal structure inherited from $\Cat^\2$. For two fibrations $P$ and $Q$, their (2-)product 
\begin{equation}\label{Fib_cart}
P \times Q \colon \A \times \B \to \X \times \Y
\end{equation}
is also an fibration, where a cartesian lifting is a pair of a $P$-lifting and a $Q$-lifting; similarly for opfibrations. The monoidal unit is the trivial (op)fibration $1_\1 \colon \1 \to \1$. Since the monoidal structure is cartesian, they are both symmetric monoidal 2-categories.

A pseudomonoid in $(\Fib,\times,1_\1)$ is called a \define{monoidal fibration}. A detailed argument of how the following proposition captures the required structure can be found in \cite{EnrichedFibration}, and also that was the original description of this notion in \cite{FramedBicats}. See also \cite{hermida2004fibrations} for an alternative approach via fibrations of multicategories.

\begin{prop}\label{def:monoidal_fibration}
    A monoidal fibration $\T \colon \U \to \V$ is a fibration for which both the total $\U$ and base category $\V$ are monoidal, $\T$ is a strict monoidal functor and the tensor product $\otimes_\U$ of $\U$ preserves cartesian liftings.
\end{prop}

Explicitly, the multiplication and unit are fibred 1-cells $\mlt = (\otimes_\U, \otimes_\V) \colon \T \times \T \to \T$ and $\uni = (I_\U, I_\V) \colon \1 \to \T$
displayed as
\begin{equation} \label{multunitmonoidalfibr}
\xymatrix @C=.6in{
    \U \times \U \ar[r]^-{\otimes_\U}
    \ar[d]_-{\T \times \T} 
    & 
    \U
    \ar[d]^-\T \\
    \V \times \V 
    \ar[r]_-{\otimes_{\V}} 
    & 
    \V
} 
\qquad \mathrm{and} \qquad
\xymatrix @C=.6in 
{\1 \ar[r]^-{I_\U} 
    \ar[d]_-{1_\1} 
    &  \U \ar[d]^-{\T}     \\
    \1 \ar[r]_-{I_{\V}} 
    &  \V
}
\end{equation}
and the fact that $\ot_\A$ is cartesian is expressed by $f^*x\ot_\A g^*y\cong (f\ot_\X g)^*(x\ot_\A y)$.

A \define{monoidal fibred 1-cell} between two monoidal fibrations $\T\colon\U\to\V$ and $\ES\colon\W\to\Z$ is a (strong) morphism of pseudomonoids between them, as defined in \cref{sec:Monoidal2cats}. It amounts to a fibred 1-cell, i.e.\ a commutative
\begin{equation}\label{eq:HF}
\begin{tikzcd}
    \U    \ar[r,"H"]    \ar[d,"\T"'] 
    & \W    \ar[d,"\ES"]    \\
    \V  \ar[r,"F"'] 
    &  \Z  
\end{tikzcd}
\end{equation}
where $H$ preserves cartesian liftings,
along with invertible 2-cells \cref{eq:laxmorphism} in $\Fib$ satisfying axioms \cref{eq:axiomslaxmorphism}. By \cref{eq:fibred2cell}, these are fibred 2-cells
\begin{displaymath}
    \begin{tikzcd}[row sep=.05in]
        & 
        \W\times\W
        \ar[dr,bend left=1ex,"\otimes_W"] 
        \\
        \U \times \U
        \ar[dddd, "\T\times\T"'] 
        \ar[dr,bend right = 1ex, "\otimes_\U"']
        \ar[ur, bend left = 1ex, "H \times H"]
        \ar[rr, phantom, "\Downarrow{\scriptstyle\phi}"] 
        && 
        \W \ar[dddd,"\ES"] 
        \\& \U
        \ar[ur,bend right = 1ex, "H"']
        & \\ \hole \\ 
        & 
        \Z\times\Z \ar[dr,bend left=1ex,"\otimes_\Z"] & \\
        \V \times \V 
        \ar[ur,bend left = 1ex, "F \times F"]
        \ar[dr, bend right = 1ex, "\otimes_\V"']
        \ar[rr, phantom, "\Downarrow{\scriptstyle\psi}"] 
        &&  \Z  \\& \V
        \ar[ur, bend right = 1ex, "F"'] &
    \end{tikzcd}
    \qquad\qquad
\begin{tikzcd}[row sep=.08in]
& & & \\
\1 \ar[dddd, equal] 
\ar[dr, bend right = 1ex, "I_\U"']
\ar[rr, bend left, "I_\W"]
\ar[rr,phantom,"\Downarrow{\scriptstyle\phi_0}"] && \W \ar[dddd,"\ES"] \\ & \U \ar[ur, bend right = 1ex, "H"']  &\\    \hole \\ &&& \\
\1 \ar[dr, bend right = 1ex, "I_\V"']
\ar[rr, bend left, "I_\Z"]
\ar[rr, phantom, "\Downarrow{\scriptstyle\psi_0}"] 
&& \Z \\
&\V \ar[ur, bend right = 1ex, "F"'] &
\end{tikzcd} 
\end{displaymath}
where  $\phi$ and $\psi$ are natural isomorphisms with components
\begin{displaymath}
    \phi_{a,b} \colon Ha \ot Hb \xrightarrow{\sim} H(a \ot b), \quad \psi_{x,y} \colon  Fx \ot Fy \xrightarrow{\sim} F(x \ot y)
\end{displaymath}
such that $\phi$ is above $\psi$, i.e.\ the following diagram commutes:
\begin{displaymath}
\begin{tikzcd}[column sep=.6in]
 \ES(Ha \ot Hb) \ar[r, "{\ES\phi_{a, b}}"]
 \ar[d, equal, "{\cref{multunitmonoidalfibr}}"'] 
 & \ES H(a \ot b) \ar[d, equal, "{\cref{eq:HF}}"] \\
 \ES H a \ot \ES Hb \ar[d, equal,"{\cref{eq:HF}}"' ] 
& F\T(a \ot b) \ar[d, equal,"{\cref{multunitmonoidalfibr}}"] \\
F\T a \ot F\T b \ar[r, "{\psi_{\T a, \T b}}"']
& F(\T a \ot \T b)
\end{tikzcd}
\end{displaymath}
Similarly, $\phi_0$ and $\psi_0$ have single components $\phi_0 \colon I_\W \xrightarrow{\sim} H(I_\U)$ and $ \psi_0 \colon I_\Z \xrightarrow{\sim} F(I_\V)$ such that $\ES(\phi_0) = \psi_0$.
These two conditions in fact say that the identity transformation, a.k.a. commutative square \cref{eq:HF} is a monoidal one, as expressed in \cite[12.5]{FramedBicats}.
The relevant axioms dictate that $(\phi,\phi_0)$ and $(\psi,\psi_0)$ give $H$ and $F$ the structure of strong monoidal functors, thus we obtain the following characterization.

\begin{prop}\label{prop:monoidalfibred1cell}
A monoidal fibred 1-cell between two monoidal fibrations $\T$ and $\ES$ is a fibred 1-cell $(H,F)$ where both functors are monoidal, $(H,\phi,\phi_0)$ and $(F,\psi,\psi_0)$, such that $\ES(\phi_{a,b})=\psi_{\T a,\T b}$ and $\ES\phi_0=\psi_0$.
\end{prop}

For lax or oplax morphisms of pseudomonoids in $\Fib$, we obtain appropriate notions of monoidal fibred 1-cells, where the top and bottom functors of \cref{eq:HF} are lax or oplax monoidal respectively.

Finally, a \define{monoidal fibred 2-cell} is a 2-cell between morphisms $(H,F)$ and $(K,G)$ of pseudomonoids $\T$, $\ES$ in $\Fib$. Explicitly, it is a fibred 2-cell as described in \cref{sec:fibrations}
\begin{displaymath}
\begin{tikzcd}[row sep=.45in,column sep=.7in]
    \U
    \ar[d, "\T"']
    \ar[r, bend left = 20, "H"]
    \ar[r, bend right = 20, "K"']
    \ar[r, phantom, "\Downarrow {\scriptstyle\beta}" description] 
    & \W \ar[d, "\ES"] \\
    \V    \ar[r, bend left = 20, "F"]
    \ar[r, bend right = 20, "G"']
    \ar[r, phantom, "\Downarrow{\scriptstyle\alpha}" description] &  \Z
\end{tikzcd}
\end{displaymath}
satisfying the axioms \cref{monoidal2cell}; these come down to the fact that both $\beta$ and $\alpha$ are monoidal natural transformations
between the respective lax monoidal functors.

\begin{prop}\label{prop:monfib2cell}
A monoidal fibred 2-cell between two monoidal fibred 1-cells is an ordinary fibred 2-cell $(\alpha,\beta)$ where both natural transformations are monoidal.
\end{prop}

We denote by $\PsMon(\Fib)= \MonFib$ the 2-category of monoidal fibrations, monoid\-al fibred 1-cells and monoidal fibred 2-cells.
By changing the notion of morphisms between pseudomonoids to lax or oplax, we obtain 2-categories $\MonFib_\lax$ and $\MonFib_\opl$.
There are also 2-categories $\BrMonFib$ and $\SymMonFib$ of \define{braided} (resp. \define{symmetric}) \define{monoidal fibrations}, \define{braided} (resp. \define{symmetric}) \define{monoidal fibred 1-cells} and monoidal fibred 2-cells, defined to be $\BrPsMon(\Fib)$ and $\SymPsMon(\Fib)$ respectively; see \cref{prop:BrPsMon}.

Dually, we have appropriate 2-categories of \define{monoidal opfibrations}, \define{monoidal opfibred 1-cells} and \define{monoidal opfibred 2-cells} and their braided and symmetric variations, $\MonOpFib$, $\BrMonOpFib$ and $\SymMonOpFib$. All the structures are constructed dually, where a monoidal opfibration, namely
a pseudomonoid in the cartesian monoidal $(\OpFib,\times,1_\1)$,
is a strict monoidal functor such that the tensor product of the total category preserves cocartesian liftings.

All the above 2-categories have sub-2-categories 
of monoidal (op)fibrations over a fixed monoidal base $(\V,\otimes,I)$, e.g.\ $\MonFib(\V)$ and $\MonOpFib(\V)$. The morphisms are \define{monoid\-al (op)fibred functors}, i.e.\ fibred 1-cells of the form $(H,1_\V)$ with $H$ monoidal, and the 2-cells are \define{monoidal (op)fibred natural transformations}, i.e.\ fibred 2-cells of the form $(\beta,1_{1_\V})$ with $\beta$ monoidal. These 2-categories constitute the lower left foot of the diagram \cref{clarifyingdiag} on the side of fibrations, and correspond to the `global' monoidal part of the story.

Moreover, the above constructions can be adjusted accordingly to the context of split fibrations. Explicitly, the 2-category $\PsMon(\Fib_\spl)=\MonFib_\spl$ has as objects \define{monoidal split fibrations}, namely split fibrations $P\colon\A\to\X$ between monoidal categories which are strict monoidal functors and $\otimes_\A$ strictly preserves cartesian liftings (compare to \cref{def:monoidal_fibration}). Furthermore, the hom-categories $\MonFib_\spl(P,Q)$ between monoidal split fibrations are full subcategories of $\MonFib(P,Q)$ spanned by the monoidal fibred 1-cells which are split as fibred 1-cells, namely $(H,F)$ as in \cref{prop:monoidalfibred1cell} where $H$ strictly preserves cartesian liftings.

We end this section by considering a different monoidal object in the context of (op)fibra\-tions, starting over from the usual 2-categories of (op)fibrations over a fixed base $\X$, (op)fibred functor and (op)fibred natural transformations $\Fib(\X)$ and $\OpFib(\X)$. Notice that contrary to the earlier devopment, there is no monoidal structure on $\X$. Both these 2-categories are also cartesian monoidal, but in a different manner than $\Fib$ and $\OpFib$, due to the cartesian monoidal structure of $\Cat/\X$; see for example \cite[1.7.4]{Jacobs}. Explicitly, for fibrations $P\colon\A\to\X$ and $Q \colon \B \to \X$, their tensor product $P \boxtimes Q$ is given by any of the two equal functors to $\X$ from the following pullback
\begin{equation}\label{Fib_X_cart}
\begin{tikzcd}[column sep=.5in,row sep=.5in]   \A \times_\X \B
        \ar[r]
        \ar[d]
        \ar[dr, phantom, very near start, "\lrcorner"]\ar[dr,dashed,bend left,"P\boxtimes Q"description]
        & 
        \A
        \ar[d,"P"] 
        \\
        \B
        \ar[r,"Q"'] 
        & 
        \X
    \end{tikzcd}
\end{equation}
since fibrations are closed under pullbacks and composition. The monoidal unit is $1_\X \colon \X \to \X$.

A pseudomonoid in $(\Fib(\X),\boxtimes,1_\X)$ is an ordinary fibration $P\colon\A\to\X$ equip\-ped with two fibred functors $(\mlt,1_\X)\colon P\boxtimes P\to P$ and $(\uni,1_\X)\colon 1_\X\to P$ displayed as
\begin{equation}\label{eq:fibrewisetensor}
    \begin{tikzcd}
        \A \times_\X \A 
        \ar[dr,"P\boxtimes P"']
        \ar[rr,"\mlt"] 
        && 
        \A
        \ar[dl,"P"] 
        \\
        & 
        \X 
        &
    \end{tikzcd}\qquad
    \begin{tikzcd}
        \X
        \ar[rr,"\uni"]
        \ar[dr,"1_X"'] 
        && 
        \A
        \ar[dl,"P"] 
        \\
        & 
        \X 
        &
        \end{tikzcd}
    \end{equation}
along with invertible fibred 2-cells satisfying the usual axioms. 
In more detail, the pullback $\A \times_\X \A$  consists of pairs of objects of $\A$ which are in the same fibre of $P$, and $P\boxtimes P$ sends such a pair to their underlying object defining their fibre. 
The functor $\mlt$ maps any $(a,b)\in\A_x$ to some $m(a,b):=a\otimes_x b\in\A_x$ and the map $\uni$ sends an object $x \in \X$ to a chosen one, $I_x$, in its fibre. The invertible 2-cells and the axioms guarantee that these maps define a monoidal structure on each fibre $\A_x$, providing the associativity, left and right unitors. The fact that $\mlt$ and $\uni$ preserve cartesian liftings translate into a strong monoidal structure on the reindexing functors: for any $f\colon x\to y$ and $a,b\in\A_y$, $f^*a\otimes_x f^*b\cong f^*(a\otimes_y b)$ and $I_y\cong f^*(I_x)$.

A (lax) morphism between two such fibrations is a fibred functor \cref{eq:fibredfunctor} such that the induced functors $H_x \colon \A_x \to \B_x$ between the fibres as in \cref{eq:functorbetweenfibres} are (lax) monoidal, whereas a 2-cell between them is a fibred natural transformation $\beta\colon H\Rightarrow K$ \cref{eq:fibrednaturaltrans} which is monoidal when restricted to the fibers, $\beta_x|_{\A_x}\colon H_x\Rightarrow K_x$. In this way, we obtain the 2-category $\PsMon(\Fib(\X))$ and dually $\PsMon(\OpFib(\X))$. These 2-categories constitute the lower right foot of the diagram \cref{clarifyingdiag} on the side of fibrations, and correspond to the `fibrewise' monoidal part of the story.

Finally, taking pseudomonoids in the 2-category of split fibrations over a fixed base, we obtain the 2-category $\PsMon(\Fib_\spl(\X))$ with objects split fibrations equipped with a fibrewise tensor product and unit as above, but now the reindexing functors strictly preserve that monoidal structure since the top functors of \cref{eq:fibrewisetensor} strictly preserve cartesian liftings: $f^*a \ot_x f^*b = f^* (a \ot_y b)$ and $I_y = f^* (I_x)$. Moreover, $\PsMon (\Fib_s(\X)) (P, Q)$ is the full subcategory of $\PsMon(\Fib(\X))(P,Q)$ spanned by split fibred functors, namely $H\colon\A\to\B$ which strictly preserve cartesian liftings but still $H_x$ are monoidal functors between the monoidal fibres as before.

\begin{rmk}\label{rmk:fixedbasefib}
    As is evident from the above descriptions, the 2-categories $\MonFib(\X)$ and $\PsMon(\Fib(\X))$ are  different in general. A monoidal fibration over $\X$ is a strict monoidal functor, whereas a pseudomonoid in fixed-base fibrations is a fibration with monoidal fibres in a coherent way: none of the base or the total category need to be monoidal. This corresponds to the two distinct legs of \cref{clarifyingdiag} concerning fibrations.
\end{rmk}

\subsection{Monoidal Indexed Categories}\label{sec:monicat}

The 2-categories of indexed and opindexed categories $\ICat$ and $\OpICat$, recalled in \cref{sec:fibrations}, are both monoidal.
Explicitly, given two indexed categories $\M \maps \X^\op \to \Cat$ and $\N \maps \Y^\op \to \Cat$, their tensor product $\M \otimes \N \maps (\X \times \Y)^\op \to \Cat$ is the composite 
\begin{equation}\label{ICat_cart}
    (\X \times \Y)^\op \cong \X^\op \times \Y^\op \xrightarrow{\M \times \N} \Cat \times \Cat \xrightarrow{\times} \Cat
\end{equation}
i.e.\ $(\M \ot \N)(x, y) = \M(x) \times \N(y)$ using the cartesian monoidal structure of $\Cat$. The monoidal unit is the indexed category $\Delta \1 \colon \1^\op \to \Cat$ that picks out the terminal category $\1$ in $\Cat$, and similarly for opindexed categories. Notice that this monoidal 2-structure, formed pointwise in $\Cat$, is also cartesian.

We call a pseudomonoid in $(\ICat, \otimes, \Delta \1)$ a \define{monoidal indexed category}. Explicitly, it is an indexed category $\M \maps \X^\op \to \Cat$ equipped with multiplication and unit indexed 1-cells $(\otimes_\X, \mu) \maps \M \otimes \M \to \M$, $(\eta, \mu_0) \maps \Delta \mathbf 1 \to \M$ which by \cref{eq:indexed1cell} look like
\[
\begin{tikzcd}[column sep=.7in,row sep=.2in]
\X^\op \times \X^\op\ar[dr, "\M \otimes \M"]\ar[dd, "\otimes^\op"'] \\ \ar[r,phantom,"\Downarrow{\scriptstyle\mu}"description] & \Cat \\
\X^\op \ar[ur,"\M"']
\end{tikzcd}\qquad 
\begin{tikzcd}[column sep=.7in,row sep=.2in]
\1^\op\ar[dr, "\Delta\1"]\ar[dd, "I^\op"'] \\ \ar[r,phantom,"\Downarrow{\scriptstyle\mu_0}"description] & \Cat \\
\X^\op \ar[ur,"\M"']
\end{tikzcd}
\]
These come equipped with invertible indexed 2-cells 
as in \cref{alphalambdarho}; the axioms this data is required to satisfy, on the one hand, render $\X$ a monoidal category with $\otimes\maps \X \times \X \to \X$ its tensor product functor and $I\colon\1\to\X$ its unit. On the other hand,  
the resulting axioms for the components
\begin{equation}\label{eq:laxatorunitor}
\mu_{x,y} \maps \M x \times \M y \to \M(x\otimes y),\qquad
\mu_0\colon\1\to\M(I)
\end{equation}
of the above pseudonatural transformations precisely give $\M$ the structure of a (weakly) lax monoidal pseudofunctor, recalled in \cref{sec:Monoidal2cats}.

\begin{prop}\label{def:monoidal_indexedcat}
    A monoidal indexed category $(\M, \mu, \mu_0) \maps (\X^\op, \otimes^\op, I) \to (\Cat, \times, \1)$ is a lax monoidal pseudofunctor, where $(\X, \otimes, I)$ is an (ordinary) monoidal category.
\end{prop}

We then define a \define{monoidal indexed 1-cell} to be a (strong) morphism between pseudomonoids in $(\ICat,\otimes,\Delta\1)$. In detail, it is an indexed 1-cell $(F, \tau) \maps \M \to \N$
\begin{displaymath}
\begin{tikzcd}[column sep=.7in,row sep=.2in]
\X^\op\ar[dr, "\M"]\ar[dd, "F^\op"'] \\ \ar[r,phantom,"\Downarrow{\scriptstyle\tau}"description] & \Cat \\
\Y^\op \ar[ur,"\N"']
\end{tikzcd}
\end{displaymath}
between two monoidal indexed categories
$(\M, \mu, \mu_0)$ and $(\N, \nu, \nu_0)$ equipped with two invertible indexed 2-cells $(\psi, m)$ and $(\psi_0, m_0)$ as in \cref{eq:laxmorphism}, which explicitly consist of natural isomorphisms $\psi$, $\psi_0$ and invertible modifications
\[
\begin{tikzcd}[row sep=.3in, column sep=.25in]
    \X^\op \times \X^\op\ar[d,"\otimes^\op"']
    \ar[dr, "F^\op \times F^\op" description]
    \ar[drrr, "\M \otimes \M", bend left=20]
    \ar[drrr, phantom,bend left=5,
    "\Downarrow{\scriptstyle\tau\times\tau}"]
    &&&&
    \X^\op \times \X^\op\arrow[drrr, "\M \otimes \M", bend left]\ar[d,"\otimes^\op"']\ar[drrr, phantom, "\Downarrow{\scriptstyle\mu}"description]
    \\
    \X^\op\ar[d, "F^\op"']\ar[r,phantom,"\Downarrow{\scriptstyle\psi}"description,bend right]
    &
    \Y^\op \times \Y^\op\ar[rr, "\N \otimes \N"description]\ar[dr, "\otimes^\op" description]\ar[drr,phantom,"\Downarrow{\scriptstyle\nu}"description]
    &&
    \Cat\ar[r, phantom, "\stackrel{m}{\Rrightarrow}"]
    &
    \X^\op\ar[rrr, "\M"description]\ar[d,"F^\op"']\ar[drrr, phantom, "\Downarrow{\scriptstyle\tau}"description]
    &&&
    \Cat
    \\
    \Y^\op\ar[rr,"\mathrm{id}"']
    &&
    \Y^\op\ar[ur, "\N"',bend right]
    & \phantom{A} &
    \Y^\op\ar[urrr, "\N"',bend right]
    &&& \phantom{A}
\end{tikzcd}
\]

\[
\begin{tikzcd}[row sep=.25in, column sep=.3in]
    \1^\op\ar[d,"I^\op"']\ar[drrr, "\Delta\1", bend left=20]
    \ar[drrr, phantom,"\Downarrow{\scriptstyle\nu_0}"]\ar[ddrr,"I^\op"description]
    &&&&&
    \1^\op\ar[drrr, "\Delta\1", bend left]\ar[d,"I^\op"']\ar[drrr, phantom, "\Downarrow{\scriptstyle\mu_0}"description]
    \\
    \X^\op\ar[d, "F^\op"']\ar[r,phantom,"\Downarrow{\scriptstyle\psi_0}"description,bend right]
    & \phantom{A} &&
    \Cat\ar[rr, phantom, "\stackrel{m_0}{\Rrightarrow}"]
    &&
    \X^\op\ar[rrr, "\M"description]\ar[d,"F^\op"']\ar[drrr, phantom, "\Downarrow{\scriptstyle\tau}"description]
    &&&
    \Cat
    \\
    \Y^\op\ar[rr,"\mathrm{id}"']
    &&
    \Y^\op\ar[ur, "\N"',bend right]
    & \phantom{A} &&
    \Y^\op\ar[urrr, "\N"',bend right]
    &&& \phantom{A}
\end{tikzcd}
\]
as dictated by the general form \cref{eq:indexed2cell} of indexed 2-cells. 
The natural isomorphisms $\psi$ and $\psi_0$ have components
\begin{displaymath}
\psi_{x,z}\colon Fx\otimes Fy\xrightarrow{\sim} F(x\otimes y),\quad
\psi_0\colon I\xrightarrow{\sim} F(I)\quad \textrm{ in $\Y^\op$}
\end{displaymath}
whereas the modifications $m$ and $m_0$ are given by families of invertible natural transformations 
\begin{displaymath}
\begin{tikzcd}[column sep=.3in,row sep=.2in]
& \N Fx{\times}\N Fy\ar[r,"\nu_{Fx,Fy}"]\ar[dr,dashed] & \N (Fx\otimes Fy)\ar[d,"\N\psi_{x,y}"] \\
\M x{\times}\M y
\ar[ur,"\tau_x\times\tau_y"]\ar[dr,"\mu_{x,y}"']\ar[rr,phantom,"\Downarrow{\scriptstyle m_{x,y}}"description] && \N F(x\otimes y) \\
& \M (x\otimes y)\ar[ur,"\tau_{x\otimes y}"'] &
\end{tikzcd}\quad
\begin{tikzcd}[column sep=.2in,row sep=.2in]
& \N(I)\ar[dr,"\N\psi_0"] & \\
\1\ar[ur,"\nu_0"]\ar[dr,"\mu_0"']\ar[rr,phantom,"\Downarrow{\scriptstyle m_0}"description] && \N(FI) \\
& \M(I)\ar[ur,"\tau_I"'] &
\end{tikzcd}
\end{displaymath}
The appropriate coherence axioms ensure that the functor $F \colon \X \to \Y$ has a strong monoidal structure $(F, \psi, \psi_0)$, and that the pseudonatural transformation $\tau \colon \M \Rightarrow \N \circ F^\op$ is monoidal with $m_{x,y}$, $m_0$ as in \cref{eq:monpseudocomponents}. Notice that $F^\op$ being monoidal makes $F$ monoidal with inverse structure isomorphisms. 

\begin{prop}\label{prop:moni1cell}
    A monoidal indexed 1-cell between two monoidal indexed categories $\M$ and $\N$ is an indexed 1-cell $(F, \tau)$, where the functor $F$ is (strong) monoidal and the pseudonatural transformation $\tau$ is monoidal. 
\end{prop}

Finally, a \define{monoidal indexed 2-cell} is a 2-cell between morphisms of pseudomonoids in $(\ICat,\otimes,\Delta\1)$. Following the definition of \cref{sec:Monoidal2cats}, it turns out that an indexed 2-cell $(a,m)\colon(F,\tau)\Rightarrow(G,\sigma)\colon\M\to\N$ as in \cref{eq:indexed2cell}, which consists of a natural transformation $\alpha\colon F\Rightarrow G$ and a modification $m$ with components
\begin{displaymath}
\begin{tikzcd}[column sep=.5in,row sep=.15in]
\M x\ar[dr,bend right=10,"\sigma_x"']\ar[rr,bend left,"\tau_x"] \ar[rr,phantom,"\Downarrow{\scriptstyle m_x}"description] && \N Fx \\
& \N Gx\ar[ur,bend right=10,"\N\alpha_x"'] & 
\end{tikzcd}
\end{displaymath}
is monoidal, exactly when $\alpha\colon F\Rightarrow G$ is compatible with the strong monoidal structures of $F$ and $G$, and the modification $m \colon \tau \Rrightarrow \N \alpha^\op \circ \sigma$ satisfies \cref{eq:monoidalmodaxioms} for the induced monoidal structures on its domain and target pseudonatural transformations.

\begin{prop}\label{prop:moni2cell}
A monoidal indexed 2-cell between two monoidal indexed 1-cells $(F, \tau)$ and $(G, \sigma)$ is an indexed 2-cell  $(\alpha, m)$ such that $\alpha$ is an ordinary monoidal natural transformation and $m$ is a monoidal modification.
\end{prop} 

We write $\PsMon(\ICat) = \Mon\ICat$, the 2-category of monoidal indexed categories, monoidal indexed 1-cells and monoidal indexed 2-cells. Moreover, their braided and symmetric counterparts form $\Br\Mon\ICat$ and $\Sym\Mon\ICat$ respectively, as the 2-categories of braided and symmetric pseudomonoids in $(\ICat,\otimes,\Delta\1)$ formally discussed in \cref{sec:Monoidal2cats}.
Similarly, we have 2-categories of \define{(braided} or \define{ symmetric) monoidal opindexed} categories, 1-cells and 2-cells $\Mon\OpICat$, $\Br\Mon\OpICat$ and $\Sym\Mon\OpICat$.

All these 2-categories have sub-2-categories of monoidal (op)indexed categories with a fixed monoidal domain $(\X,\otimes, I)$, and specifically
\begin{gather}\label{eq:monicatX}
\MonICat(\X)=\MonTCat_\pse(\X^\op,\Cat) \\ \Mon\OpICat(\X)=\MonTCat_\pse(\X,\Cat)\nonumber
\end{gather}
the functor 2-categories of lax monoidal pseudofunctors, monoidal pseudonatural transformations and monoidal modifications; these belong to the lower left foot of the diagram \cref{clarifyingdiag}, now on the side of indexed categories.

Moreover, we can consider pseudomonoids in the strict context. Explicitly, the 2-category $\PsMon(\ICat_\spl)=\MonICat_\spl$ has as objects \define{monoidal strict indexed categories} namely (2-)functors $\ps{M}\colon\X^\op\to\Cat$ from an ordinary monoidal category $\X$ which are weakly lax monoidal as before, but the laxator and unitor \cref{eq:laxatorunitor} are strictly natural rather than pseudonatural transformations as in the original \cref{eq:weakmonpseudo}. The hom-categories $\PsMon(\ICat_\spl)(\ps{M},\ps{N})$ between monoidal strict indexed categories are full subcategories of $\MonICat(\ps{M},\ps{N})$ spanned by strict natural transformations -- which are however still weakly monoidal, i.e. equipped with isomorphisms \cref{eq:monpseudocomponents}.

Similarly to the previous \cref{sec:monfib} on fibrations, we end this section with the study of pseudomonoids in a different but related monoidal 2-category, namely $\ICat(\X)=\TCat_\pse(\X^\op,\Cat)$ of indexed categories with a fixed domain $\X$. Working in this 2-category, or in $\OpICat(
\X)$, there is no assumed monoidal structure on $\X$. Their monoidal structure is again cartesian: for two $\X$-indexed categories $\M, \N \colon \X^\op \to \Cat$, their product is
\begin{equation}\label{eq:icatxprod}
    \M\boxtimes\N\colon\X^\op \xrightarrow{\Delta} \X^\op \times \X^\op \xrightarrow{\M \times \N} \Cat \times \Cat \xrightarrow{\times} \Cat
\end{equation}
with pointwise components $(\M \boxtimes \N) (x) = \M (x) \times \N (x)$ in $\Cat$. The monoidal unit is just $\X^\op \xrightarrow{!} \1 \xrightarrow{\Delta \1} \Cat$, which we will also call $\Delta \1$.

A pseudomonoid in $(\ICat(\X),\boxtimes,\Delta\1)$ is a pseudofunctor $\M \colon \X^\op \to \Cat$ equipped with indexed functors \cref{eq:ifun} $\mlt\colon\M\boxtimes\M\to\M$ and $\uni\colon\Delta\1\to\M$ namely
\begin{displaymath}
\begin{tikzcd}[row sep=.1in]
& \X^\op\times\X^\op\ar[r,"\M\times\M"] & \Cat\times\Cat\ar[dr,"\times"] & & & \1\ar[dr,"\Delta\1"] & \\
\X^\op\ar[ur,"\Delta"]\ar[rrr, phantom, "\Downarrow{\scriptstyle\mlt}"description]\ar[rrr,bend right=20,"\M"'] 
&&&\Cat & \X^\op\ar[ur,"!"]\ar[rr,bend right,"\M"']\ar[rr,phantom,"\Downarrow{\scriptstyle\uni}"description] && \Cat
\end{tikzcd}
\end{displaymath}
with components $\mlt_x \colon \M x \times \M x \to \M x$ and $\uni_x \colon \1 \to \M x$ which are pseudonatural via
\begin{equation}\label{eq:pseudonaturalmult}
\begin{tikzcd}[column sep=.6in]
\M x\times \M x\ar[d,"\mlt_x"']\ar[r,"\M f\times\M f"]\ar[dr,phantom,"\cong"description] & \M y\times\M y\ar[d,"\mlt_y"] \\
\M x\ar[r,"\M f"'] & \M y
\end{tikzcd}\quad
\begin{tikzcd}
\1\ar[r,"="]\ar[d,"\uni_x"']\ar[dr,phantom,"\cong"description] & \1\ar[d,"\uni_y"] \\
\M x\ar[r,"\M f"'] & \M y
\end{tikzcd}
\end{equation}
If we denote $\mlt_x=\ot_x$ and $\uni_x=I_x$, the pseudomonoid invertible 2-cells \cref{alphalambdarho} and the axioms these data satisfy make each $\M x$ into a monoidal category $(\M x,\ot_x,I_x)$, and each $\M f$ into a strong monoidal functor: the above isomorphisms have components $\M f(a)\otimes_y\M f(b)\cong\M f(a\otimes_x b)$ and $I_y\cong\M f(I_x)$ for any $a,b\in\M x$. 

Such a structure, namely a pseudofunctor $\M\colon\X^\op\to\MonCat$ into the 2-category of monoidal categories, strong monoidal functors and monoidal natural transformations, was directly defined as an \emph{indexed strong monoidal category} in \cite{DescentForMonads}, and as \emph{indexed monoidal category} in \cite{PontoShulman}. We will avoid this terminology in order to not create confusion with the term \emph{monoidal indexed categories}.

A strong morphism of pseudomonoids \cref{eq:laxmorphism} in $(\ICat(\X),\boxtimes,\Delta\1)$ ends up being a pseudonatural trasformation $\tau\colon\M\Rightarrow\N\colon\X^\op\to\Cat$ (indexed functor) whose components $\tau_x\colon\M x\to\N x$ are strong monoidal functors, whereas a 2-cell between strong morphisms of pseudomonoids is an ordinary modification
\begin{displaymath}
\begin{tikzcd}
\X^\op\ar[rr,bend left=40,"\M",""'{name = F}]\ar[rr,bend right=40,"\N"',""{name = G}] \ar[rr,phantom,"\stackrel{m}{\Rrightarrow}"description] && \Cat
 \arrow[from = F, to = G, Rightarrow, "\tau"',bend right=50]
  \arrow[from = F, to = G, Rightarrow, "\sigma",bend left=50]
\end{tikzcd}
\end{displaymath}
whose components $m_x\colon\tau_x\Rightarrow\sigma_x$ are monoidal natural transformations.

We thus obtain the 2-categories $\PsMon (\ICat(\X))$ as well as $\PsMon(\OpICat(\X))$; from the above descriptions, it is clear that  
\begin{gather}\label{eq:imoncats}
\PsMon(\ICat(\X))=\TCat_\pse(\X^\op, \MonCat) \\
\PsMon(\OpICat(\X))=\TCat_\pse(\X, \MonCat)\nonumber
\end{gather}
which will also be rediscovered by \cref{prop:imoncat_formally}. These 2-categories correspond to the right foot of \cref{clarifyingdiag} on the side of indexed categories.

Finally, taking pseudomonoids in strict $\X$-indexed categories $\ICat_\spl(\X)=[\X^\op,\Cat]$ produces the 2-category $\PsMon(\ICat_\spl(\X))$ with objects functors $\M\colon\X^\op\to\MonCat_\mathrm{st}$ into monoidal categories with strict monoidal functors: the isomorphisms \cref{eq:pseudonaturalmult} are now equalities due to strict naturality of the multiplication and unit. Then the hom-categories $\PsMon(\ICat_\spl(\X))(\ps{M},\ps{N})$ are full subcategories of $\PsMon(\ICat(\X))(\ps{M},\ps{N})$ spanned by strictly natural transformations $\tau\colon\ps{M}\Rightarrow\ps{N}$, still with strong monoidal components $\tau_x$. For example, it would not be correct to write $\PsMon(\ICat_\spl(\X))=[\X^\op,\MonCat_{(\mathrm{st})}]$.

\begin{rmk}\label{rmk:fixedbaseicat}
    Similarly to what was noted in \cref{rmk:fixedbasefib}, it is evident that $\MonICat(\X)$ and $\PsMon(\ICat(\X))$ are in principle different. A monoidal indexed category with base $\X$ is a lax monoidal pseudofunctor into $\Cat$ (and $\X$ is required to be monoidal already), whereas a a pseudomonoid in $\X$-indexed categories is a pseudofunctor from an ordinary category $\X$ into $\MonCat$. This is also highlighted by the indexed category legs of \cref{clarifyingdiag}.
\end{rmk}

\subsection{The equivalence \texorpdfstring{$\MonFib \simeq \MonICat$}{M}} \label{sec:monequiv}

In \cref{sec:fibrations}, we recalled 
the standard equivalence between fibrations and indexed categories via the Grothendieck construction. We will now lift this correspondence to their monoidal versions studied in Section \ref{sec:monfib} and \ref{sec:monicat}, using general results about pseudomonoids in arbitrary monoidal 2-categories described in \cref{sec:Monoidal2cats}.

Since both $\Fib$ and $\ICat$ are cartesian monoidal 2-categories, via \cref{Fib_cart} and \cref{ICat_cart} respectively, our first task is to ensure that they are \emph{monoidally} equivalent.

\begin{lem}\label{lem:monoidalGrfunctor}
    The 2-equivalence $\Fib\simeq\ICat$ between the cartesian monoidal 2-categories of fibrations and indexed categories is (symmetric) monoidal.
\end{lem}

\begin{proof}
    Since they form an equivalence, both 2-functors from \cref{thm:Grothendieck} preserve limits, therefore are monoidal 2-functors. Moreover, it can be verified that the natural isomorphisms with components $\ps{F}\cong\ps{F}_{P_\ps{F}}$ and $P\cong P_{\ps{F}_P}$ are monoidal with respect to the cartesian structure, due to universal properties of products.
\end{proof}

As a result, and since $\MonFib=\PsMon(\Fib)$ and $\MonICat=\PsMon(\ICat)$, we obtain the following equivalence as a special case of \cref{prop:2equivpseudomon}; also for $\OpFib\simeq\OpICat$.

\begin{thm}\label{thm:mainthm}
    There are 2-equivalences 
    \begin{gather*}
        \MonFib\simeq \MonICat \\
        \BrMonFib\simeq\Br\Mon\ICat \\
        \SymMonFib\simeq\Sym\Mon\ICat
\end{gather*}
between the 2-categories of monoidal fibrations and monoidal indexed categories, as well as their braided and symmetric versions.

Dually, there is a 2-equivalence $\MonOpFib\simeq\MonOpICat$ between the 2-categories of monoidal opfibrations and monoidal opindexed categories, as well as their braided and symmetric versions.
\end{thm}

\begin{cor}\label{cor:fixedbasemonoidalGr}
The above 2-equivalences restrict to the sub-2-categories of fixed bases or domains, which by \cref{eq:monicatX} are
\begin{gather*}
\MonFib(\X) \simeq\MonTCat_\pse(\X^\op,\Cat) \\
\MonOpFib(\X) \simeq\MonTCat_\pse(\X^\op,\Cat)
\end{gather*}
\end{cor}

These results are summarized by the equivalences on the left foot of \cref{clarifyingdiag}, and correspond to the \emph{global} monoidal structure of fibrations and indexed categories.
Even though they were directly derived via abstract reasoning, for exposition purposes we briefly describe this equivalence on the level of objects; some relevant details can also be found in \cite[\S 6]{NetworkModels}. Independently and much earlier, in his thesis \cite{ShulmanPhD} Shulman explores such a fixed-base equivalence on the level of double categories (of monoidal fibrations and monoidal pseudofunctors over the same base).

Suppose that $(\M, \mu, \mu_0) \maps (\X^\op, \otimes, I) \to (\Cat, \times, \1)$
is a monoidal indexed category, i.e.\ a lax monoidal pseudofunctor with structure maps \cref{eq:laxatorunitor}. The induced monoidal product $\otimes_\mu \maps \inta \M \times \inta \M \to \inta \M$ on the Grothendieck category is defined on objects by 
\begin{equation}\label{eq:globalmonstr}
    (x,a) \otimes_\mu (y,b) = (x \otimes y, \mu_{x,y}(a,b))
\end{equation}
and $I_\mu=(I, \mu_0(*))$ is the unit object. 
Clearly, the induced fibration $\inta \ps{M} \to \X$ which maps each pair to the underlying $\X$-object strictly preserves the monoidal structure. Moreover, pseudonaturality of $\mu$ implies that $\otimes_\mu$ preserves cartesian liftings, so all clauses of \cref{def:monoidal_fibration} are satisfied.
For a more detailed exposition of the structure, as well as the braided and symmetric version, we refer the reader to the \cref{sec:monoidal}.

We can also restrict to the context of split fibrations and strict indexed categories. Again by applying $\PsMon(\textrm{-})$ to the 2-equivalence $\ICat_\spl\simeq\Fib_\spl$, we obtain equivalences between the respective structures discussed in \cref{sec:monfib,sec:monicat}, as the strict counterparts of \cref{thm:mainthm} and \cref{cor:fixedbasemonoidalGr}. Recall that a monoidal strict indexed category is a weakly lax monoidal 2-functor $\X^\op\to\Cat$ whose structure maps $(\phi,\phi_0)$ are strictly natural transformations, and corresponds to a split fibration which is monoidal like before, only the tensor product of the total category strictly preserves cartesian liftings.

\begin{thm}\label{thm:mainthmsplit}
    There are 2-equivalences
    \begin{gather*}
        \MonFib_\spl \simeq \MonICat_\spl \\
        \MonOpFib_\spl \simeq \MonOpICat_\spl
    \end{gather*}
    between monoidal split (op)fibrations and monoidal strict (op)indexed categories, as well as for the fixed-base case.
\end{thm}

\begin{rmk}
    Based on an observation made by Mike Shulman in private correspondence with the authors, this monoidal version of the Gro\-the\-ndieck construction may in fact be further generalized to the context of double categories. More specifically, there is evidence of a correspondence between discrete fibrations of double categories, and lax double functors into the double category $\SSpan$ of sets and spans.  If such a result was also true for arbitrary fibrations of double categories, \cref{thm:mainthm} would be a special case for double categories with one object and one vertical arrow, namely monoidal categories.
\end{rmk}

We close this section in a similar manner to Sections \ref{sec:monfib} and \ref{sec:monicat}, namely by working in the cartesian monoidal 2-categories $(\Fib(\X),\boxtimes,1_\X)$ and $(\ICat(\X),\boxtimes,\Delta\1)$ of fibrations and indexed categories with fixed bases and domains, to begin with. 
Since $\Fib(\X)\simeq\ICat(\X)$ is also a monoidal 2-equivalence, \cref{prop:2equivpseudomon} applies once more -- recall \cref{eq:imoncats}.

\begin{thm}\label{thm:fibrmonGr}
    There are 2-equivalences between (op)fibrations with monoidal fibres and strong monoidal reindexing functors, and pseudofunctors into $\MonCat$
    \begin{align*}
        \PsMon(\Fib(\X)) &\simeq \TCat_\pse(\X^\op,\MonCat)\quad 
        \\
        \PsMon(\OpFib(\X)) &\simeq \TCat_\pse(\X^\op,\MonCat) 
    \end{align*}
    Moreover, these restrict to 2-equivalences between split (op)fibrations with monoidal fibres and strict monoidal reindexing functors, and ordinary functors into $\MonCat_\mathrm{st}$.
\end{thm}
These equivalences establish the right leg of \cref{clarifyingdiag}, and correspond to the \emph{fibrewise} monoidal structure on fibrations and indexed categories.
In more detail, a pseudofunctor $\M\colon\X^\op\to\MonCat$ maps every object $x$ to a monoidal category $\M x$ and every morphism $f\colon x\to y$ to a strong monoidal functor $\M f\colon\M y\to\M x$; under the usual Grothendieck construction, these are precisely the fibre categories and the reindexing functors between them for the induced fibration, as described at the end of \cref{sec:monfib}. 
Notice how, in particular, $\X$ is \emph{not} a monoidal category, as was the case in \cref{cor:fixedbasemonoidalGr}.

\begin{rmk}
A very similar, relaxed version of the fibrewise monoidal correspondence seems to connect the concepts of an \emph{indexed monoidal category}, defined in \cite{DescentForMonads} as a pseudofunctor $\M\colon\X^\op\to\MonCat_\lax$, and that of of a \emph{lax monoidal fibration}, defined in \cite{LaxMonFibs}. Notice that these terms are misleading with respect to ours: an indexed monoidal category is \emph{not} a monoidal indexed category, and also a lax monoidal fibration is \emph{not} a functor with a lax monoidal stucture.

Briefly, there is a full sub-2-category $\Fib_\opl(\X)\subseteq\Cat/\X$ of fibrations, namely fibred 1-cells \cref{commutativefibredcell} which are not required to have a cartesian functor on top. As discussed in \cite[Prop.3.6]{FramedBicats}, this is 2-equivalent to $\TCat_{ps,opl}(\X^\op,\Cat)$, the 2-category of pseudofunctors, oplax natural transformations and modifications. Describing pseudomonoids therein appears to give rise to a fibration with monoidal fibres and \emph{lax} monoidal reindexing functors between them, or equivalently a pseudofunctor into $\MonCat_\lax$. We omit the details so as to not digress from our main development.
\end{rmk}

\section{(Co)cartesian case: fibrewise and global monoidal structures}\label{sec:fibrewisemonoidal}

In the previous section, we obtain two different equivalences between fixed-base fibrations and fixed-domain indexed categories of monoidal flavor: \cref{cor:fixedbasemonoidalGr} where both total and base categories are monoidal, and \cref{thm:fibrmonGr} where only the fibres are monoidal, namely the two different legs of \cref{clarifyingdiag}. 
Clearly, neither of these two cases implies the other in general. The global monoidal structure as defined in \cref{eq:globalmonstr} sends two objects in arbitrary fibres to a new object lying in the fibre of the tensor of their underlying objects in the base, whereas having a fibrewise tensor products does not give a way of multiplying objects in different fibres of the total category.

In \cite{FramedBicats}, Shulman introduces monoidal fibrations (\cref{def:monoidal_fibration}) as a building block for fibrant double categories.
Due to the nature of the examples, the results restrict to the case where the base of the monoidal fibration $\T \colon \U \to \V$ is equipped with specifically a cartesian or cocartesian monoidal structure; 
the main idea is that these fibrations form a ``parameterized family of monoidal categories''. Formally, a central result therein lifts the Grothendieck construction to the monoidal setting, by showing an equivalence between monoidal fibrations over a fixed (co)cartesian base and ordinary pseudofunctors into $\MonCat$.

\begin{thm}\cite[Thm. 12.7]{FramedBicats}\label{thm:Shulman}
    If $\X$ is cartesian monoidal,
    \begin{equation}\label{eq:Shulmanequiv}
        \MonFib(\X)\simeq\TCat_\pse(\X^\op,\MonCat)
    \end{equation}
    Dually, if $\X$ is cocartesian monoidal, $\MonOpFib(\X)\simeq\TCat_\pse(\X,\MonCat)$.
\end{thm}

Evidently, this result provides an equivalence between the two separate feet of \cref{clarifyingdiag}. Bringing all these structures together, we obtain the following.

\begin{thm}\label{thm:fibrewise=global}
    If $\X$ is a cartesian monoidal category, 
    \begin{displaymath}
    \begin{tikzcd}[ampersand replacement=\&,sep=.25in]
    \MonFib(\X)\ar[r,"\simeq"]\ar[d,"\simeq"'{anchor=south, rotate=90, inner sep=.5mm}]
    \&
    \MonTCat_\pse(\X^\op,\Cat)\ar[d,"\simeq"{anchor=south, rotate=270, inner sep=.5mm}] \\
    \PsMon(\Fib(\X))\ar[r,"\simeq"] \& \TCat_\pse(\X^\op,\MonCat)
    \end{tikzcd}
    \end{displaymath}
    Dually, if $\X$ is a cocartesian monoidal category,
    \begin{displaymath}
    \begin{tikzcd}[ampersand replacement=\&,sep=.25in]
    \MonOpFib(\X)\ar[r,"\simeq"]\ar[d,"\simeq"'{anchor=south, rotate=90, inner sep=.5mm}]
    \&
    \MonTCat_\pse(\X,\Cat)\ar[d,"\simeq"{anchor=south, rotate=270, inner sep=.5mm}] \\
    \PsMon(\OpFib(\X))\ar[r,"\simeq"] \& \TCat_\pse(\X,\MonCat)
    \end{tikzcd}
    \end{displaymath}
In the strict context, the restricted equivalences give a correspondence between monoid\-al split (op)fibrations over $\X$ and functors $\X^{(\op)}\to\MonCat_\mathrm{st}$.
\end{thm}

The original proof of \cref{thm:Shulman} is an explicit, piece-by-piece construction of an equivalence, and employs the reindexing functors $\Delta^*$ and $\pi^*$ induced by the diagonal and projections in order to move between the appropriate fibres and build the required structures. The global monoidal structure is therein called \emph{external} and the fibrewise \emph{internal}.

Here we present a different argument that does not focus on the fibrations side. The equivalence between lax monoidal pseudofunctors $\X^\op\to\Cat$ and ordinary pseudofunctors $\X^\op\to\MonCat$, which essentially provides a way of transferring the monoidal structure from the target category to the functor itself and vice versa, brings a new perspective on the behavior of such objects.

\begin{lem}\label{lem:helplemma}
For any two monoidal 2-categories $\K$ and $\L$, the following are true.
\begin{enumerate}
    \item For an arbitrary 2-category $\A$,
    \begin{equation}\label{eq:equiv1}
\TCat_\pse(\A,\MonTCat_\pse(\K,\L))\simeq\MonTCat_\pse(\K,\TCat_\pse(\A,\L))
\end{equation}
    \item For a cocartesian 2-category $\A$,
    \begin{equation} \label{eq:equiv2}
\TCat_\pse(\A,\MonTCat_\pse(\K,\L))\simeq\MonTCat_\pse(\A\times\K,\L)
    \end{equation}
\end{enumerate}
\end{lem}

\begin{proof}
First of all, recall \cite[1.34]{FibrationsinBicats} that there are equivalences
\begin{displaymath}
\TCat_\pse(\A,\TCat_\pse(\K,\L))\simeq
\TCat_\pse(\A\times\K,\L)\simeq
\TCat_\pse(\K,\TCat_\pse(\A,\L))
\end{displaymath}
which underlie \cref{eq:equiv1,eq:equiv2} for the respective pseudofunctors; so the only part needed is the correspondence between the respective monoidal structures. Notice that $\A\times\K$ is a monoidal 2-category since both $\A$ and $\K$ are, and also $\TCat_\pse(\A,\L)$ is monoidal since $\L$ is: define $\otimes_{[]}$ and $I_{[]}$ by $(\ps{F}\otimes_{[]}\ps{G})(a)=\ps{F}a\otimes_\L\ps{G}a$ (similarly to \cref{eq:icatxprod}) and $I_{[]}\colon\A\xrightarrow{!}\1\xrightarrow{I_\L}\L$.

$(1)$ Take a pseudofunctor $\ps{F}\colon\A\to\MonTCat_\pse(\K,\L)$. For every $a\in\A$, its image pseudofunctor $\ps{F}a$ is lax monoidal, i.e.\ comes equipped with morphisms in $\L$
\begin{equation}\label{eq:Faweakmon}
    \phi_{x,y}^a \colon (\ps{F} a) (x) \otimes_\L (\ps{F} a) (y) \to (\ps{F} a) (x \otimes_\K y), \quad \phi_0^a \colon I_\L \to (\ps{F} a) I_\K
\end{equation}
for every $x, y \in \K$, satisfying coherence axioms.

Now define the pseudofunctor $\bar{\ps{F}} \colon \K \to \TCat_\pse(\A, \L)$, with $(\bar{\ps{F}} x) (a) := (\ps{F}a) (x)$. It has a lax monoidal structure, given by pseudonatural transformations
\begin{displaymath}
    \bar{\ps{F}} x \otimes_{[]} \bar{\ps{F}} y \Rightarrow \bar{\ps{F}} (x \otimes_\K y), \quad I_{[]} \Rightarrow\bar{\ps{F}}(I_\K) 
\end{displaymath}
whose components evaluated on some $a\in\A$ are defined to be \cref{eq:Faweakmon}. Pseudonaturality and lax monoidal axioms follow, and in a similar way we can establish the opposite direction and verify the equivalence.

$(2)$ If $\A$ is a cocartesian monoidal 2-category, a lax monoidal pseudofunctor $\ps{F}\colon\A\to\MonTCat_\pse(\K,\L)$ induces a pseudofunctor $\tilde{\ps{F}}\colon\A\times\K\to\L$ by $\tilde{\ps{F}}(a,x):=(\ps{F}a)(x)$. Its lax monoidal structure is given by the composite
\begin{displaymath}
\begin{tikzcd}[row sep=.1in,column sep=.2in]
    \tilde{\ps{F}}(a,x) \otimes_\L \tilde{\ps{F}}(b,y)
    \ar[d,equal]
    \ar[rr,dashed,"\psi_{(a,x),(b,y)}"] 
    && 
    \tilde{\ps{F}}(a + b, x \otimes_\K y)
    \ar[d,equal] 
    \\
    (\ps{F}a)(x) \otimes_\L (\ps{F}b)(y)
    \ar[rdd, "{(\ps{F} \iota_a)_x \otimes (\ps{F}\iota_b)_y}"'] 
    && 
    (\ps{F}(a+b))(x \otimes_\K y) 
    \\ \hole \\& 
    (\ps{F}(a+b))(x) \otimes_\L(\ps{F}(a+b))(y)
    \ar[uur,"\phi^{a+b}_{x,y}"'] 
    & 
\end{tikzcd}
\end{displaymath}
where $a\xrightarrow{\iota_a}a+b\xleftarrow{\iota_b}b$ are the inclusions, and $\psi_0\colon I_\L\xrightarrow{\phi_0^0}\tilde{\ps{F}}(0,I_\K)$; the respective axioms follow.

In the opposite direction, starting with some pseudofunctor $\ps{G}\colon\A\times\K\to\L$ equipped with a lax monoidal structure $\psi_{(a,x),(b,y)}$ and $\psi_0$, we can build $\hat{\ps{G}}\colon\A\to\MonTCat_\pse(\K,\L)$ for which every $\hat{\ps{G}}a$ is a lax monoidal pseudofunctor, via
\begin{gather*}
\begin{tikzcd}[row sep=.1in,column sep=.5in,ampersand replacement=\&]
    (\hat{\ps{G}}a)(x)\otimes_\L(\hat{\ps{G}}b)(y)
    \ar[d,equal]\ar[rr,dashed,"\phi^a_{(x,y)}"] \&\& (\hat{\ps{G}}a)(x\otimes_\K y)\ar[d, equal] \\
    \ps{G}(a, x) \otimes_\L \ps{G}(a, y)
    \ar[r,"{\psi_{(a, x), (a, y)}}"'] \&\ps{G}(a + a, x \otimes_\K y)
    \ar[r,"{G(\nabla,1)}"'] \& \ps{G}(a,x\otimes_\K y)
\end{tikzcd}\\
\phi_0^a \colon I_\L \xrightarrow{\psi_0} G(0,I_\K) \xrightarrow{G(!,1)} G(a, I_\K)
\end{gather*}
The equivalence follows, using the universal properties of coproducts and initial object.
\end{proof} 

\begin{proof}[Proof of \cref{thm:fibrewise=global}]
The top and bottom right 2-categories of the first square are equivalent as follows, where $\X^\op$ is cocartesian.
\begin{align*}
    \TCat_\pse(\X^\op,\MonCat)
    & \simeq \TCat_\pse (\X^\op, \PsMon(\Cat)) & \cref{eq:PsMon}\\
    & \simeq \TCat_\pse (\X^\op, \MonTCat_\pse(\1, \Cat)) & \cref{eq:equiv2} \\
    & \simeq \MonTCat_\pse (\X^\op \times \1, \Cat)  \\
    & \simeq \MonTCat_\pse(\X^\op, \Cat)
\end{align*}
The strict context equivalence can be explicitly verified as a special case of the above, where the corresponding 1-cells and 2-cells are as described in \cref{sec:monfib,sec:monicat}.
\end{proof}
The decisive step in the above proof is
the much broader \cref{lem:helplemma}; for a grounded explanation of the correspondence of the relevant structures, see \cref{monicat=imoncat}.
In simpler words, a lax monoidal structure of a pseudofunctor $F\colon(\A,+,0)\to(\Cat,\times,\1)$ gives a pseudofunctor $F\colon\A\to\MonCat$ and vice versa: in a sense, `monoidality' can move between the functor and its target.

As another corollary of \cref{lem:helplemma}, we can formally deduce that pseudomonoids in $(\ICat(\X),\boxtimes,\Delta\1)$ are functors into $\MonCat$, as described at the end of \cref{sec:monicat}.

\begin{prop}\label{prop:imoncat_formally}
    For any $\X$, $\PsMon(\ICat(\X))\simeq\TCat_\pse(\X^\op,\MonCat)$.
\end{prop}

\begin{proof} There are equivalences
\begin{align*}
    \PsMon(\ICat(\X))
    & = \PsMon(\TCat_\pse(\X^\op,\Cat)) \\
    & \simeq \MonTCat_\pse (\1,\TCat_\pse(\X^\op,\Cat)) & \cref{eq:equiv1}\\
    & \simeq \TCat_\pse(\X^\op,\MonTCat_\pse(\1,\Cat)) & \cref{eq:PsMon} \\
    & \simeq
    \TCat_\pse(\X^\op,\PsMon(\Cat))\\
    & \simeq
    \TCat_\pse(\X^\op,\MonCat) \qedhere
\end{align*}
\end{proof}

As a first and meaningful example of \cref{thm:fibrewise=global}, recall by Remarks \ref{rem:Fibisfibred} and \ref{rem:ICatisfibred} that the categories $\Fib$ and $\ICat$ are themselves fibred over $\Cat$, with fibres $\Fib(\X)$ and $\ICat(\X)$ respectively. The base category in both cases is the cartesian monoidal category $(\Cat,\times,1)$, therefore \cref{thm:fibrewise=global} applies. The following proposition shows that the monoidal structures of $\Fib$, $\ICat$ and $\Fib(\X)$, $\ICat(\X)$, instrumental for the study of global and fibrewise monoidal structures, follow the very same abstract pattern.  

\begin{prop}\label{prop:FibICatglobalfibrewise}
    The fibrations $\Fib\to\Cat$ and $\ICat\to\Cat$ are monoidal, and moreover their fibres $\Fib(\X)$ and $\ICat(\X)$ are monoidal and the reindexing functors are strong monoidal.
\end{prop}

\begin{proof}
The pseudofunctors inducing $\Fib\to\Cat$ and $\ICat\to\Cat$ are
\begin{displaymath}
\begin{tikzcd}[row sep=.05in]
\Cat^\op\ar[rr] && \mathsf{CAT} && \Cat^\op\ar[rr] && \mathsf{CAT} \\
\X\ar[mapsto,rr]\ar[dd,"F"'] && \Fib(\X) && \X\ar[mapsto,rr]\ar[dd,"F"'] && \ICat(\X) \\
\hole \\
\Y\ar[mapsto,rr] && \Fib(\Y)\ar[uu,"F^*"'] &&
\Y\ar[mapsto,rr] && \ICat(\Y)\ar[uu,"-\circ F^\op"']
\end{tikzcd}
\end{displaymath}
where $\mathsf{CAT}$ is the 2-category of possibly large categories, $F^*$ takes pullbacks along $F$ and $-\circ F^\op$ precomposes with the opposite of $F$. These are both lax monoidal, with the respective structures essentially being \cref{Fib_cart} and \cref{ICat_cart} giving the global monoidal structure on the fibrations.

Since the base of both monoidal fibrations is cartesian, the global monoidal structure is equivalent to a fibrewise monoidal structure, as per the theme of this whole section. The induced monoidal structure on each $\Fib(\X)$ is given by \cref{Fib_X_cart} and on each $\ICat(\X)$ by \cref{eq:icatxprod}, and $F^*$, $-\circ F^\op$ are strong monoidal functors accordingly. 
\end{proof}

The above essentially lifts the global and fibrewise monoidal structure development one level up, exhibiting fibrations and indexed categories as examples of the monoidal Grothendieck construction themselves.

Concluding this investigation on monoidal structures of fibrations and indexed categories, we consider the (co)cartesian monoidal (op)fibration case; for example, a monoidal fibration $P\colon(\A,\times,1)\to(\X,\times,1)$ as in \cref{def:monoidal_fibration} where $P$ preserves products (or coproducts for opfibrations) on the nose. As remarked in \cite[12.9]{FramedBicats}, the equivalence \cref{eq:Shulmanequiv} restricts to one between pseudofunctors which land in cartesian monoidal categories, and monoidal fibrations where the total category is cartesian monoidal. With the appropriate 1-cells and 2-cells that preserve the structure, we can write the respective equivalences as
\begin{gather}
    \TCat_\pse(\X^\op, \cMonCat) \simeq \cMonFib(\X) \textrm{ for cartesian $\X$}\label{eq:cocartspecialcase} \\
    \TCat_\pse(\X, \cocMonCat) \simeq \cocMonOpFib(\X) \textrm{ for cocartesian $\X$}\nonumber
\end{gather}
where the prefixes $\mathsf{c}$ and $\mathsf{coc}$ correspond to the respective (co)cartesian structures.
Explicitly, in order for the total category to specifically be endowed with (co)carte\-sian monoidal structure, it is required not only that the base category is but also the fibres are and the reindexing functors preserve finite (co)products.

\begin{rmk}
    This special case of the monoidal Grothendieck construction that connects the existence of (co)products and initial/terminal object in the fibres and in the total category, is reminiscent (and also an example of) the general theory of \emph{fibred limits} originated from \cite{Grayfibredandcofibred}. Explicitly, \cite[Cor.~4.9]{hermida1999some} deduces that if the base of a fibration $P\colon\A\to\X$ has $\J$-limits for any small category $\J$, then the fibres have and the reindexing functors preserve $\J$-limits if and only if $\A$ has $\J$-limits and $P$ strictly preserves them, and dually for opfibrations and colimits. 
    Hence for finite (co)products in (op)fibrations, \cref{eq:cocartspecialcase} re-discovers that result using the monoidal Grothendieck correspondence. 
\end{rmk}

Moreover, since the squares of \cref{thm:fibrewise=global} reduce to their (co)cartesian variants, we would like to identify the conditions that the corresponding lax monoidal pseudofunctor into $\Cat$ needs to satisfy in order to give rise to a (co)cartesian monoidal (op)fibration.
Recall that by a folklore result presented in \cite{HeunenVicary12}, any symmetric monoidal category equipped with suitably well-behaved diagonals and augmentations must in fact be cartesian monoidal. We employ its dual version to tackle the opfibration case:
if, in a symmetric monoidal category $\X$, there exist monoidal natural transformations with components
\begin{displaymath}
\nabla_x \maps x \otimes x \to x,\quad
u_x \maps I \to x
\end{displaymath}
satisfying the commutativity of
\begin{equation}\label{eq:nabla}
\begin{tikzcd}
I\ot x\ar[r,"u_x\ot1"]\ar[dr,"\sim"{rotate=-30},"\ell_x"'] & x\ot x\ar[d,"\nabla_x"] & x\ot I\ar[dr,"\sim"{rotate=-30},"r_x"']\ar[r,"1\ot u_x"] & x\ot x\ar[d,"\nabla_x"] \\
& x&& x
\end{tikzcd}
\end{equation}
then $\X$ is cocartesian monoidal. In fact, it is the case that a symmetric monoidal category is cocartesian if and only if $\Mon(\X)\cong\X$.

Suppose $(\M,\mu,\mu_0) \maps \X\to\Cat$ is a (symmetric) lax monoidal pseudofunctor, such that the corresponding Grothendieck category $(\inta \M, \otimes_\mu,I_\mu)$ described in \cref{sec:monequiv} is cocartesian monoidal. This means there are monoidal natural transformations with components
\begin{equation*}
    \nabla_{(x, a)} 
    \maps (x, a) \otimes_{\mu} (x, a) 
    \to (x, a)
    \quad\textrm{and}\quad  u_{(x, a)} 
    \maps (I,\mu_0(*)) 
    \to (x, a)
\end{equation*}
making the diagrams \cref{eq:nabla} commute.
Explicitly, by \cref{eq:globalmonstr}, $\nabla_{(x, a)}$ consists of morphisms
$f_x \maps x \otimes x \to x$ in $\X$ and
$\kappa_a\maps (\M f_x)(\mu_{x,x}(a,a)) \to a$ in $\M x$,
whereas $u_{(x, a)}$ consists of 
$i_x \maps I \to x$ in $\X$ and
$\lambda_a \maps (\M i_x)\mu_0 \to a$
in $\M x$.

The conditions \cref{eq:nabla} say that the composites
\[
    (I, \mu_0) \otimes_\mu (x, a) \xrightarrow{u_{(x,a)} \otimes_\mu 1_{(x,a)}} (x, a) \otimes_\mu (x, a) \xrightarrow{\nabla_{(x, a)}} (x, a)
\]
\[
    (x, a) \otimes_\mu (I, \mu_0) \xrightarrow{1_{(x,a)} \otimes_\mu u_{(x,a)}} (x, a) \otimes_\mu (x, a) \xrightarrow{\nabla_{(x, a)}} (x, a)
\]
are equal to the left and right unitor on $x$, where all respective structures are detailed in \cref{sec:monoidal}.
Using the composition inside $\inta \M$ analogously to \cref{eq:comp_intM}, these conditions translate, on the one hand, to the base being cocartesian monoidal $(\X,+,0)$ with $f_x=\nabla_x$ and $i_x=u_x$. On the other hand, $\kappa_a$ and $\lambda_a$ form natural transformations 
\begin{equation}\label{kappalambda}
\begin{tikzcd}[row sep=.1in,column sep=.2in]
& \M x \times \M x \ar[r, "\mu_{x,x}"] & 
\M (x+x)\ar[dr, "\M(\nabla_x)"] & \\
\M x \ar[ur, "\Delta"] \ar[rrr, bend right=20, "1"'] \ar[rrr, phantom, "\Downarrow {\scriptstyle \kappa^x}"] &&& \M x
\end{tikzcd}\quad
\begin{tikzcd}[row sep=.1in,column sep=.2in]
& \1\ar[r,"\mu_0"] & \M(0)\ar[dr,"\M(u_x)"] & \\ 
\M x\ar[ur,"!"]\ar[rrr,bend right=20,"1"']\ar[rrr,phantom,"\Downarrow{\scriptstyle \lambda^x}"] &&& \M x
\end{tikzcd}
\end{equation}
satisfying the commutativity of
\begin{displaymath}
\begin{tikzcd}[column sep=.2in,row sep=.2in]
\M(\nabla_x\circ(u_x+1))(\mu_{0,x}(\mu_0(*),a))\ar[ddd,"\mathrm{id}"']\ar[rr,"\sim"',"\delta"] && (\M(\nabla_x)\circ\M(u_x+1))((\mu_{0,x}(\mu_0(*),a))\ar[d,"\sim"{anchor=south, rotate=90, inner sep=.5mm},"{\M(\nabla_x)(\mu_{u_x,1})}"] \\
& & \M(\nabla_x)(\mu_{x,x}(\M(u_x)(\mu_0(*),a)))\ar[d,"{\M(\nabla_x)\left(\mu_{x,x}(\lambda^x_a,\gamma)\right)}"] \\
&& \M(\nabla_x)(\mu_{x,x}(a,a))\ar[d,"{\kappa^x_a}"] \\
\M(\ell_x)(\mu_{0,x}(\mu_0(*),a))\ar[rr,"\xi","\sim"'] && a
\end{tikzcd}
\end{displaymath}
and a similar one with $\mu_0$ on second arguments. The above greatly simplifies if $\M$ is just a lax monoidal functor:
the first condition becomes $1_a \cong \kappa^x_a \circ \M(\nabla_x) (\mu_{x, x} (\lambda_a^x, 1))$, and the second one $1_a \cong \kappa^x_a \circ \M(\nabla_x) (\mu_{x, x} (1_a, \lambda_a^x))$.

\begin{cor}\label{cor:kappalambda}
    A lax monoidal pseudofunctor $\M\colon(\X,+,0)\to(\Cat,\times,\1)$ equip\-ped with natural transformations $\kappa$ and $\lambda$ as in \cref{kappalambda} corresponds to an ordinary pseudofunctor $\M\colon\X\to\cocMonCat$, or equivalently \cref{eq:cocartspecialcase} to a cocartesian monoidal opfibration.
\end{cor}

\section{Examples}\label{sec:applications}

In this section, we explore certain settings where the equivalence between monoidal fibrations and monoidal indexed categories naturally arises. Instead of going into details that would result in a much longer text, we mostly sketch the appropriate example cases up to the point of exhibition of the monoidal Grothendieck correspondence, providing indications of further work and references for the interested reader.

\subsection{Fundamental (bi)fibration}\label{sec:fundamentalfib}

For any category $\X$, the \emph{codomain} or \emph{fundamental} opfibration is the usual functor from its arrow category
\[\cod \colon \X^2 \longrightarrow \X\] mapping every morphism to its codomain and every commutative square to its right-hand side leg. It uniquely corresponds to the strict opindexed category, i.e.\ mere functor
\begin{equation}\label{eq:fundamentalindexedcat}
\begin{tikzcd}[row sep=.05in]
    \X\ar[r] & \Cat \\
    x\ar[r,mapsto]\ar[dd,"f"'] &     \X/x\ar[dd,"f_!"] \\
    \hole \\
    y\ar[mapsto,r] & \X/y
\end{tikzcd}
\end{equation}
that maps an object to the slice category over it and a morphism to the post-composition functor $f_!=f\circ-$ induced by it.

If the category has a monoidal structure $(\X,\otimes,I)$, this (2-)functor naturally becomes weakly lax monoidal with structure maps
\begin{equation}\label{eq:slicelaxator}
\X/x\times\X/y\xrightarrow{\otimes}\X/(x\otimes y), \quad \1\xrightarrow{1_I}\X/I.
\end{equation}
These components form strictly natural transformations, and for example the invertible modification $\omega$ \cref{eq:omega} has components the evident isomorphisms, for $(f,g,h)\in\X/x\times\X/y\times\X/z$, between 
\begin{align}
a \ot (b \ot c)& \xrightarrow{f \ot (g \ot h)} x \ot (y \ot z) \cong (x \ot y) \ot z \label{eq:pseudoassociativity}\\
(a \ot b) \ot c & \xrightarrow{(f \ot g) \ot h} (x \ot y) \ot z\nonumber
\end{align}
By \cref{thm:mainthmsplit}, this monoidal strict opindexed category correspondes to a monoidal split fibration, i.e.
$(\X^\2, \otimes, 1_I)$ is monoidal and $\cod$ strict monoidal, where $\ot_{\X^\2}$ strictly preserves cartesian liftings via $f_!k \ot g_! \ell = (f \ot g)_! (k \ot \ell)$ -- which can of course be independently verified. However in general, the slice categories $\X/x$ do not inherit the monoidal structure: there is no way to restrict the global monoidal structure to a fibrewise one.

According to \cref{thm:fibrewise=global}, there is an induced monoidal structure on the categories $\X/x$ and a strict monoidal structure on all $f_!$ only when the monoidal structure on $\X$ is given by binary coproducts and an initial object (i.e.\ cocartesian).
In that case, for each $k\colon a\to x$ and $\ell\colon b\to x$ in the same fibre $\X/x$, their tensor product in $\X/x$ is given by
\begin{displaymath}
a+b\xrightarrow{\;k+\ell\;}x+x\xrightarrow{\nabla_x}x
\end{displaymath}
as a simple example of \cref{eq:explicitstructure1}.
In fact, this is precisely the coproduct of two objects in $\X/x$, and $0\xrightarrow{!}x$ the initial object, due to the way colimits in the slice categories are constructed. Therefore this falls under the cocartesian-fibres special case \cref{eq:cocartspecialcase}, bijectively corresponding to the cocartesian structure on $\X^\2$ inherited from $\X$.

Now suppose an ordinary category $\X$ has pullbacks. This endows the codomain functor also with a fibration structure, corresponding to the indexed category
\begin{displaymath}
\begin{tikzcd}[row sep=.05in]
    \X^\op\ar[r] & \Cat \\
    x\ar[r,mapsto]\ar[dd,"f"'] & \X/x\\
    \hole \\
    y\ar[mapsto,r] & \X/y\ar[uu,"f^*"']
\end{tikzcd}
\end{displaymath}
with the same mapping on objects as \cref{eq:fundamentalindexedcat} but by taking pullbacks rather than post-composing along morphisms, a pseudofunctorial assignment. This gives $\cod\colon\X^2\to\X$ a bifibration structure, also by that classic fact that $f_!\dashv f^*$.

In this case, if $\X$ has a general monoidal structure, there is no naturally induced lax monoidal structure of that pseudofunctor as before: there is no reason for the pullback of a tensor to be isomorphic to the tensor of two pullbacks.  
However, if $\X$ is cartesian monoidal (hence has all finite limits), the components
\begin{displaymath}
\X/x\times\X/y\xrightarrow{\times}\X/(x\times y),\qquad \1\xrightarrow{\Delta_!}\X/1
\end{displaymath}
are pseudonatural since pullbacks commute with products. Moreover, this bijectively corresponds to monoidal fibres and strong monoidal reindexing functors, in fact also cartesian ones: for morphisms $k\colon a\to x$ and $\ell\colon b\to x$ in $\X/x$, their induced product is given by
\begin{displaymath}
\begin{tikzcd}[sep=.5in]
\bullet\ar[r]\ar[d,"\delta^*(k\times\ell)"']\ar[dr, phantom, very near start, "\lrcorner"] & a\times b\ar[d,"k\times\ell"] \\
x\ar[r,"\delta"] & x\times x
\end{tikzcd}
\end{displaymath}
and $1_x\colon x\to x$ is the unit of each slice $\X/x$, this indexed monoidal category also described in \cite[3.3(1)]{DescentForMonads}.
The monoidal fibration structure on $\cod\colon(\X^2,\times,1_1)\to(X,\times,1)$ is the evident one, so it again falls in the special case \cref{eq:cocartspecialcase} now for cartesian fibres, by construction of products in slice categories.

As a final remark, analogous constructions hold for the domain functor which is again a bifibration: its fibration structure comes from pre-composing along morphisms, whereas its opfibration structure comes from taking pushouts along morphisms. In fact, \cref{rem:Fibisfibred} as well as \cref{prop:FibICatglobalfibrewise} can be thought as special cases of this more general setting, for the categories of fibrations themselves. 

\subsection{Modeling graphs and networks}

Denote by $\Grph = \Set^{\TWO}$ the usual category of (directed, multi) graphs, and consider the functor 
\begin{equation}\label{eq:vertexfunctor}
    V \maps \Grph \to \Set
\end{equation}
which sends a graph to its set of vertices. It is well-known that this functor is a split opfibration, which can also be obtained as the  Grothendieck construction on the strict opindexed category (i.e.\ functor) $\Grph_{(-)} \maps \Set \to \Cat$ described as follows. A set $X$ is mapped to the category $\Grph_X$ of graphs with vertex set $X$ and homomorphisms which fix the vertices, namely $\Set / X {\times}X$ with morphisms \begin{displaymath}
\begin{tikzcd}
    E 
    \ar[rr, "k"]
    \ar[dr, "{(s, t)}"'] 
    && 
    E' 
    \ar[dl, "{(s', t')}"] 
    \\& 
    X \times X 
    &
\end{tikzcd}\quad\textrm{ or equivalently }\quad
\begin{tikzcd}
    E
    \ar[rr, "k"]
    \ar[dr, shift left=1, "t"]
    \ar[dr, shift right=1, "s"'] 
    &&
    E'
    \ar[dl, shift left=1, "t'"]
    \ar[dl, shift right=1, "s'"'] 
    \\&
    X 
\end{tikzcd}
\end{displaymath}
Moreover, any function $f \colon X \to Y$ gives rise to the post-composition functor \[\Grph_X=\Set / X \times X \xrightarrow{(f \times f) \circ -} \Set/Y \times Y=\Grph_Y\] that maps an $X$-graph
$\xymatrix@C=.15in{(s,t) \colon E \ar@<+.3ex>[r] \ar@<-.3ex>[r]& X}$
to the $Y$-graph  
$\xymatrix@C=.15in{(f\circ s,f\circ t)\colon E \ar@<+.3ex>[r]
\ar@<-.3ex>[r]
& Y}$. 
Clearly, this functor $\Grph_{(-)}=\Set / (-{\times}-)$ relates to the codomain functor \cref{eq:fundamentalindexedcat} described earlier. As explained via \cref{eq:slicelaxator}, considering $\Set$ with its cocartesian monoidal structure induces a (symmetric) weakly lax monoidal structure on the (2-)functor, namely
\begin{equation}\label{eq:graphlaxator}
    (\Grph_{(-)}, \sqcup,1_0) \maps (\Set, +,0) \to (\Cat, \times,\1),
\end{equation}
with natural structure maps
\begin{displaymath}
\sqcup_{X, Y} \colon \Grph_X \times \Grph_Y \to \Grph_{X + Y}, \qquad
1_0\colon\1 \to \Grph_0
\end{displaymath}
where $\sqcup_{X, Y}(\xymatrix@C = .15in{E \ar@<+.3ex>[r]^-{s} \ar@<-.3ex>[r]_-{t} & X}, \xymatrix@C = .15in{F \ar@<+.3ex>[r]^-{s'} \ar@<-.3ex>[r]_-{t'} & Y})=\xymatrix@C = .25in{E + F \ar@<+.3ex>[r]^-{s + s'} \ar@<-.3ex>[r]_-{t+t'} & X + Y}$. 
\begin{rmk}\label{rmk:DecCsp}
Restricting the domain of \cref{eq:graphlaxator} to the monoidal subcategory $\FinSet$ of finite sets and post-composing with the forgetful to $\Set$ which discards morphisms between graphs, we obtain the motivating example in Fong's so-called \emph{theory of decorated cospans} \cite{FongDSpan}.\footnote{As noticed by the reviewer of this manuscript, this does not really form an ordinary lax monoidal structure on $\Grph_{(-)}\colon \FinSet\to\Set$ as assumed in \cite[\S~5.1]{FongDSpan}, due to pseudoassociativity as in \cref{eq:pseudoassociativity}. Here we consider the existing theory of decorated cospans with this subtlety in mind, leaving suitable generalizations to \cite{EquivalenceFrameworks}.}
In more detail, therein a category $F \Cospan$ is constructed from a given symmetric lax monoidal functor \[(F, \phi, \phi_0) \maps (\FinSet, +,0) \to (\Set, \times,1)\] with the broad goal of modelling open networks. The objects of such a category are finite sets, and a morphism is a cospan of finite sets $X \xrightarrow{i} N \xleftarrow{o} Y$ equipping the `network' $N$ with a certain notion of input and output, along with a \define{decoration} of the apex, namely an object $s \in F(N)$. For example, a graph decorated cospan with one input and two output designated nodes looks like
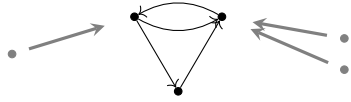
\begin{center}
    \begin{tikzpicture}[auto, scale=1.15]
    \node[circle, draw, inner sep=1pt, fill=gray, color=gray] (x) at (-1.4, -.43) {};
    \node[circle, draw, inner sep=1pt, fill] (A) at (0, 0) {};
    \node[circle, draw, inner sep=1pt, fill] (B) at (1, 0) {};
    \node[circle, draw, inner sep=1pt, fill] (C) at (0.5, -.86) {};
    \node[circle, draw, inner sep=1pt, fill=gray, color=gray] (y1) at (2.4, -.25) {};
    \node[circle, draw, inner sep=1pt, fill=gray, color=gray] (y2) at (2.4, -.61) {};
    \path (B) edge  [bend right, ->] node[above] {} (A);
    \path (A) edge  [bend right, ->] node[below] {} (B);
    \path (A) edge  [->] node[left] {} (C);
    \path (C) edge  [->] node[right] {} (B);
    \path[color=gray, very thick, shorten >=10pt, shorten <=5pt, ->, >=stealth] (x) edge (A);
    \path[color=gray, very thick, shorten >=10pt, shorten <=5pt, ->, >=stealth] (y1) edge (B);
    \path[color=gray, very thick, shorten >=10pt, shorten <=5pt, ->, >=stealth] (y2) edge (B);
\end{tikzpicture}
\end{center}
where the middle graph is the chosen decoration of the three-element set.
In fact, the apex of any $F$-decorated cospan can be viewed as an object of the (discrete) Grothendi\-eck category $\inta F$ on that functor, since it is a finite set that always comes together with an element of the set of all its possible decorations.

We call such a functor $F\in\SymMonCat_\lax((\FinSet, +,0),(\Set, \times,1))$ a \define{decorator}. In fact, the construction of a (symmetric monoidal) category $F\Cospan$ induces a functor from decorators into $\Sym\Mon\Cat$. More details can be found in \cite{EquivalenceFrameworks}, where a correspondence between decorated and \emph{structured} cospans is established, partially due to the monoidal Grothendieck construction.
\end{rmk}
    
Since the domain $\Set$ of the weakly lax monoidal functor $\Grph_{(-)}$ of \cref{eq:graphlaxator} is taken with its cocartesian monoidal structure, the strict version of \cref{thm:fibrewise=global} ensures it bijectively corresponds to a functor $\Set\to\MonCat_\mathrm{st}$, namely the fibres $\Grph_X$ have a monoidal structure which is strictly preserved by the reindexing $-\circ f \colon \Grph_X \to\  \Grph_Y$. It can be verified that the fibres are cocartesian themselves, falling under the equivalence \cref{eq:cocartspecialcase}: for any two graphs $(s,t)\colon E \rightrightarrows X$, $(s',t')\colon E' \rightrightarrows X$ over the set of vertices $X$, \cref{eq:explicitstructure1} gives
\begin{equation}\label{eq:overlay}
\begin{tikzcd}[column sep=.2in]
    E+E'
    \ar[rr, shift left, "s+s'"]
    \ar[rr, shift right, "t+t'"'] 
    && 
    X+X
    \ar[r, "\nabla_X"] 
    & 
    X
\end{tikzcd}
\end{equation}
explicitly constructed by
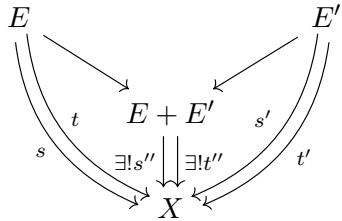
\[
\begin{tikzcd}
    E
    \arrow[dr]
    \arrow[ddr, shift right, bend right, "s", swap]
    \arrow[ddr, shift left, bend right, "t"]
    &&
    E'
    \arrow[dl]
    \arrow[ddl, shift right, bend left, "s'", swap]
    \arrow[ddl, shift left, bend left, "t'"]
    \\&
    E + E'
    \arrow[d, shift right, "\exists!s''", swap]
    \arrow[d, shift left, "\exists!t''"]
    \\&
    X
\end{tikzcd}\]
as is the case of colimits in slice categories. The resulting graph is the \emph{overlay} of the two given graphs, identifying corresponding vertices. This is the same as computing the pushout over the obvious inclusions of the graph with vertex set $X$ and no edges into each of the given graphs. The initial object of its fibre $\Grph_X$ is $(!,!)\colon\emptyset\rightrightarrows X$.

\begin{rmk}\label{rmk:NetMod}
Should we want to view the (symmetric) monoidal categories $\Grph_X$ as commutative monoids of $X$-graphs with overlay \cref{eq:overlay} as the binary operation on the set of objects, we can formally take isomorphism classes of objects and then forget the morphisms in each $\Grph_X$. Putting all this data together, for $\mathsf C\Mon$ the category of commutative monoids, there is an induced ordinary symmetric lax monoidal functor
\[
    (\Grph_{(-)}, \sqcup,1_0) \maps (\Set, +,0) \to (\mathsf C\Mon, \times,\1).
\]
If we furthermore restrict its domain to the symmetric groupoid of finite sets and bijections $\namedcat{S}$, we obtain the so-called \emph{network model} for graphs.
In more detail, in \cite{NetworkModels} an operad $\mathcal O_F$ is constructed from a given symmetric lax monoidal functor 
\[
    (F, \phi, \phi_0) \maps (\namedcat S, +, 0) \to (\Mon, \times,\1)
\] 
and such a functor is called a \define{network model}. The monoids $F(\mathbf n)$ are called the \define{constituent monoids} of $F$, and $\mathcal O_F$ is the \emph{underlying} operad of the induced monoidal category $(\inta F,\otimes_\phi,I_\phi)$ described in \cref{sec:monoidal}. The category of network models is denoted $\NetMod = \Sym\MonCat_\lax(({\namedcat S}, +, 0),(\Mon, \times, \1))$ and the mapping on $F \mapsto \mathcal O_F$ defines a functor $\NetMod \to \namedcat{Opd}$ into the category of operads.

The intuition behind this work is that a large complex network can be built from smaller ones by gluing them together in ways written as combinations of a few basic operations, expressed via monoid multiplications and monoidal functors. 
As an example, let $X$ be a finite set and $F(X)$ be the set of graphs with vertex set $X$; the monoid operation says that two graphs with the same vertex set can be overlaid by identifying the corresponding vertices:
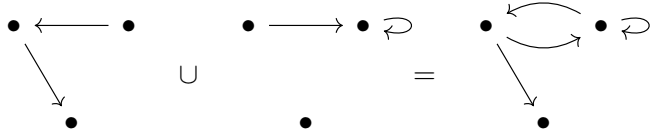
\[
\begin{tikzcd}[column sep = tiny, row sep = tiny]
    \bullet
    \arrow[ddr]
    &&
    \bullet
    \arrow[ll]
    &&
    \bullet
    \arrow[rr]
    &&
    \bullet
    \arrow[loop right]
    &&
    \bullet
    \arrow[rr, bend right]
    \arrow[ddr]
    &&
    \bullet
    \arrow[ll, bend right]
    \arrow[loop right]
    \\&&&
    \cup
    &&&&
    =
    \\&
    \bullet
    &&&&
    \bullet
    &&&&
    \bullet
\end{tikzcd}\]
\end{rmk}

The above considerations exhibit a relation between the theory of decorated cospans (\cref{rmk:DecCsp}) and network models (\cref{rmk:NetMod}) via the machinery of the monoidal Grothendieck correspondence.

\begin{thm}\label{thm:decoratorsvsnetworkmodels}
    There is a faithful, injective-on-objects functor from the category of decorators into the category of network models
\begin{displaymath}
    \SymMonCat_\lax ((\FinSet, +,0), (\Set, \times,\1)) \to \Sym\MonCat_\lax ((\namedcat S, +,0), (\Mon, \times,\1))
\end{displaymath}
\end{thm}

\begin{proof}
Starting with a lax monoidal functor $(F, \phi,\phi_0) \maps (\FinSet, +,0) \to (\Set, \times,\1)$, we can view it as a special case of a monoidal strict opindexed category (namely $\omega,\xi,\zeta$ as in \cref{eq:omega} are identities) which is in fact discrete, post-composing with $\Set\hookrightarrow\Cat$. Under \cref{thm:fibrewise=global}, this bijectively corresponds to an ordinary functor $\FinSet\to\MonCat_{\mathrm{st}}$ but with two particular characteristics. First of all, the fibres end up being strict monoidal, see \cref{rmk:laxmonfun}; combining this with the fact that they are mere sets rather than categories, means that the fibres are monoids, with structure as in \cref{eq:explicitstructure1}. Second, it can be verified that the original laxator and unitor $(\phi,\phi_0)$ of $F$ are in fact morphisms of monoids due to their strict naturality and functoriality of $F$. Therefore the ordinary functor $F$ becomes lax monoidal $(\FinSet,+,0)\to(\Mon,\times,\1)$. Pre-composing with the inclusion from the groupoid of finite sets and bijections $\namedcat S$ concludes this proof.
\end{proof}

As a result, any decorator gives rise to a network model but not vice versa: the constructed network model has always commutative constituent monoids. Noncommutative network models exist and arise in applications, see \cite{NoncommNetMods}.

Concluding this section, the symmetric weakly lax monoidal functor $\Grph_{(-)}$ \cref{eq:graphlaxator} gives rise to a symmetric monoidal opfibration structure on the vertex functor
$V\colon(\Grph,+,0)\to(\Set, +, 0)$ from the very beginning of this section \cref{eq:vertexfunctor}.  This is the well-known fact that the forgetful $V$ (strictly) preserves all coproducts and the initial object, falling under the cocartesian monoidal fibration case of \cref{eq:cocartspecialcase} and  \cref{cor:kappalambda}. 

\subsection{Family Fibration: Zunino and Turaev categories}\label{sec:familyfib}

Recall that for any category $\C$, the standard \emph{family fibration} is induced by the (strict) functor
\begin{equation}\label{eq:functV}
    [-,\C]\colon \Set^\op\to\Cat
\end{equation}
which maps every discrete category $X$ to the functor category $[X,\C]$ and every function $f \colon X \to Y$ to the functor $f^* = [f,1]$, i.e.\ pre-composition with $f$.
The total category of the induced fibration $\Fam(\C)\to\C$ has as objects pairs $(X,M\colon X\to\C)$ essentially given by a family of $X$-indexed objects in $\C$, written $\{M_x\}_{x\in X}$, whereas the morphisms are \begin{displaymath}
\begin{tikzcd}[column sep=.7in,row sep=.2in]
X\ar[dr, "M"]\ar[dd, "f"'] \\ \ar[r,phantom,"\Downarrow{\scriptstyle\alpha}"description] & \C \\
Y \ar[ur,"N"']
\end{tikzcd}
\end{displaymath}
namely a function $f\colon X\to Y$ together with families of morphisms $\alpha_x\colon M_x\to N_{fx}$ in $\C$. Notice the similarity of this description with \cref{eq:indexed1cell}, which for the strict indexed categories case looks like a non-discrete version of the family fibration, for $\C=\Cat$; see also \cref{rem:ICatisfibred}. Moreover, it is a folklore fact that $\Fam(\C)$ is the free coproduct cocompletion on the category $\C$. 

On the other hand, we could consider the opfibration induced by the very same functor \cref{eq:functV}, denoted by $\Maf(\C)\to\Set^\op$. The objects of $\Maf(\C)$ are the same as $\Fam(\C)$, but morphisms $\{M_x\}_{x\in X}\to\{N_y\}_{y\in Y}$ between them are functions $g\colon Y\to X$ (i.e.\ $X\to Y$ in $\Set^\op$) together with families of arrows $\beta_y\colon M_{gy}\to N_y$ in $\C$. Notice that these are now indexed over the set $Y$ rather than $X$ like before, and in fact $\Maf (\X) = \Fam (\X^\op)^\op$.

In the case that the category is monoidal $(\C,\ot,I)$, the (2-)functor $[-,\C]$ has a canonical weakly lax monoidal structure. Explicitly, by taking its domain $\Set^\op$ to be cocartesian by the usual cartesian monoidal structure $(\Set,\times,1)$, the structure maps are
\begin{displaymath}
\phi_{X,Y} \colon [X, \C] \times [Y, \C] \to [X \times Y, \C], \qquad \phi_0 \colon \1 \xrightarrow{I_\C} [\1, \C] \cong \C
\end{displaymath}
where $\phi_{X,Y}$ corresponds, under the tensor-hom adjunction in $\Cat$, to
\begin{displaymath}
    [X, \C] \times [Y, \C] \times X \times Y \xrightarrow{\sim} [X, \C] \times X \times [Y, \C] \times Y \xrightarrow{\textrm{ev}_X \times \textrm{ev}_Y} \C \times \C \xrightarrow{\otimes} \C.
\end{displaymath}
These are again natural components, and for example \cref{eq:omega} has components the natural isomorphisms between the assignments $Mx\ot (Ny\ot Uz)$ and $(Mx\ot Ny)\ot Uz$. 
By \cref{thm:mainthmsplit}, this monoidal strict indexed category endows the corresponding split fibration $\Fam(\X) \to \Set$ with a monoidal structure via $\{M_x\} \otimes \{N_y\} \coloneqq \{M_x \otimes N_y \}_{X \times Y}$. On the other hand, we could use the dual part of the same theorem, and instead consider the induced monoidal split opfibration $\Maf(\X) \to \Set^\op$ corresponding to the same $([-,\C],\phi,\phi_0)$.

Moreover, since $\Set$ is cartesian, \cref{thm:fibrewise=global} also applies in both cases, giving a monoidal structure to the fibres as well: for $M\colon X\to\C$ and $N\colon X\to\C$, their fibrewise tensor product and unit are given by
\begin{displaymath}
X\xrightarrow{\Delta}X\times X\xrightarrow{M\times N}\C\times\C\xrightarrow{\ot}\C, \qquad
X\xrightarrow{!}1\xrightarrow{I}\C
\end{displaymath}
which are precisely constructed as in \cref{eq:explicitstructure1}. Once again, notice the direct similary with \cref{eq:icatxprod}, the fibrewise monoidal structure on $\ICat(\X)$; see also \cref{rem:ICatisfibred} and \cref{prop:FibICatglobalfibrewise}.

As an interesting example, consider $\C=\Mod_R$ for a commutative ring $R$, with its usual tensor product $\ot_R$. In \cite{TuraevZunino}, the authors introduce a category $\mathcal{T}$ of \emph{Turaev} $R$-modules, as well as a category $\mathcal{Z}$ of \emph{Zunino} $R$-modules, which serve as symmetric monoidal categories where group-(co)algebras and Hopf group-(co)algebras, \cite{Turaev}, live as (co)monoids and Hopf monoids respectively. 

In more detail, the objects of both $\mathcal{T}$ and $\mathcal{Z}$ are defined to be pairs $(X,M)$ where $X$ is a set and $\{M_x\}_{x\in X}$ is an $X$-indexed family of $R$-modules, and their morphisms are respectively
\begin{displaymath}
    (\mathcal{T})
    \begin{cases}
        s \colon M_{g(y)} \to N_y \textrm{ in } \Mod_R \\
        g \colon Y \to X \textrm{ in } \Set
    \end{cases}
    \quad
    (\mathcal{Z})
    \begin{cases}
        t\colon M_x\to N_{f(x)}\textrm{ in }\Mod_R \\
        f\colon X\to Y\textrm{ in }\Set
    \end{cases}
\end{displaymath}
There is a symmetric pointwise monoidal structure, $\{M_x\otimes_R N_y\}_{X\times Y}$, and there are strict monoidal forgetful functors $\mathcal{T}\to\Set^\op$, $\mathcal{Z}\to\Set$.
It is therein shown that comonoids in $\mathcal{T}$ are \emph{monoid-coalgebras} and monoids in $\mathcal{Z}$ are \emph{monoid-algebras}, i.e.\ families of $R$-modules indexed over a monoid, together with respective families of linear maps
\begin{gather*}
    (\mathcal{T})\quad C_{g*h}\to C_g\ot C_h \qquad (\mathcal{Z})\quad
    A_g\ot A_h\to A_{g*h} \\
    C_e\to R \qquad \phantom{ZZZZZZZ}R\to A_e
\end{gather*}
satisfying appropriate axioms. Based on the above, it is clear that $\mathcal{T}=\Maf(\Mod_R)$ and $\mathcal{Z}=\Fam(\Mod_R)$, which clarifies the origin of these categories and can be directly used to further generalize the notions of Hopf group-(co)monoids in arbitrary monoidal categories.

\subsection{Global categories of modules and comodules}\label{sec:ModComod}

For any monoidal category $\mathcal{V}$,
there exist \emph{global} categories of modules and comodules, denoted by $\Mod$ and $\Comod$ \cite[6.2]{PhDChristina}.
Their objects are all (co)modules over (co)monoids in $\catname{V}$, whereas a morphism between an $A$-module $M$ and a $B$-module $N$ is given by a monoid map $f\colon A\to B$ together with a morphism $k\colon M\to N$ in $\catname{V}$ satisfying the commutativity of 
\begin{displaymath}
\begin{tikzcd}
    A\otimes M\ar[rr,"\mu"]\ar[d,"1\ot k"'] && M\ar[d,"k"] \\
    A\otimes N\ar[r,"f\ot1"'] & B\ot N\ar[r,"\mu"'] & N
\end{tikzcd}
\end{displaymath}
where $\mu$ denotes the respective action, and dually for comodules.
Both these categories arise as the total categories induced by the Grothendieck construction on the functors 
\begin{equation}\label{eq:functors}
    \xymatrix @R=.05in @C=.4in
    {
        \Mon(\catname{V})^\op
        \ar[r] 
        & 
        \Cat 
        \\
        A
        \ar@{|.>}[r]
        \ar[dd]_-f 
        & 
        \Mod_\catname{V}(A) 
        \\ \hole \\
        B
        \ar@{|.>}[r] 
        & 
        \Mod_\catname{V}(B)
        \ar[uu]_-{f^*}
    } 
    \xymatrix @R=.05in @C=.4in
    {
        \Comon(\catname{V})\ar[r] & 
        \Cat 
        \\
        C\ar@{|.>}[r]\ar[dd]_-g & \Comod_\catname{V}(C)
        \ar[dd]^-{g_!} \\
        \hole \\
        D\ar@{|.>}[r] & \Comod_\catname{V}(D)
    }
\end{equation}
where $f^*$ and $g_!$ are (co)restriction of scalars: if $M$ is a $B$-module, $f^*(M)$ is an $A$-module via the action
\begin{displaymath}
    A \ot M \xrightarrow{f \ot 1} B \ot M \xrightarrow{\mu} M.
\end{displaymath}
The induced split fibration and opfibration, $\Mod \to \Mon(\catname{V})$ and $\Comod \to \Comon(\catname{V})$, map a (co)module to its respective (co)monoid.

Recall that when $(\catname{V},\ot,I,\sigma)$ is braided monoidal, its categories of monoids and como\-noids inherit the monoidal structure: if $A$ and $B$ are monoids, then $A\ot B$ has also a monoid structure via
\begin{displaymath}
A \ot B \ot A \ot B \xrightarrow{1 \ot \sigma \ot 1} A \ot A \ot B \ot B \xrightarrow{m \ot m} A \ot B,\qquad
I \cong I \ot I \xrightarrow{j \ot j} A \ot B
\end{displaymath}
where $m$ and $j$ give the respective monoid structures. 
In that case, the induced split fibration and opfibration are both monoidal. This can be deduced by directly checking the conditions of \cref{def:monoidal_fibration}, as was the case in the relevant references, or in our setting by using \cref{thm:mainthmsplit} since both (2-)functors \cref{eq:functors} are weakly lax monoidal. For example, for any $A,B \in \Mon(\catname{V})$ there are natural maps
\begin{displaymath}
\phi_{A,B} \colon \Mod_\catname{V} (A) \times \Mod_\catname{V} (B) \to \Mod_\catname{V} (A \otimes B)\qquad  \phi_0 \colon \1 \to \Mod_\catname{V} (I)
\end{displaymath}
with $\phi_{A, B} (M, N) = M \ot N$, with the $A \ot B$-module structure being
\begin{displaymath}
    A \ot B \ot M \ot N \xrightarrow{1 \ot \sigma \ot 1} A \ot M \ot B \ot N \xrightarrow{\mu \ot \mu} M \ot N
\end{displaymath}
and $\phi_0(*)=I$, which are pseudoassociative and pseudounital in the sense that e.g. for any $M,N,P\in\Mod_\mathcal{V}(A)\times\Mod_\mathcal{V}(B)\times\Mod_\mathcal{V}(C)$, $M\ot(N\ot P)$ is only isomorphic to $(M\ot N)\ot P$ as $(A\ot B)\ot C$-modules.

Notice that in general, the monoidal bases $\Mon(\catname{V})$ and $\Comon(\catname{V})$ are not (co)\-ca\-rte\-sian, since they have the same tensor as $(\catname{V},\otimes,I,\sigma)$. Therefore this case does not fall under \cref{thm:fibrewise=global}, hence the fibre categories are not monoidal.
For example in $(\mathsf{Vect}_k,\otimes_k,k)$, the $k$-tensor product of two $A$-modules for a $k$-algebra $A$ is not an $A$-module as well.

We remark that the induced monoidal opfibration $\Comod \to \Comon(\catname{V})$ in fact serves as the monoidal base of an \emph{enriched fibration} structure on $\Mod \to \Mon (\catname{V})$ as explained in \cite{EnrichedFibration}, built upon an enrichment between the monoidal bases $\Mon(\catname{V})$ in $\Comon(\catname{V})$ established in \cite{Measuringcomonoid}.
Moreover, analogous monoidal structures are induced on the (op)fibrations of monads and comonads in any fibrant monoidal double category, see \cite[Prop. 3.18]{VCocats}. 

\subsection{Systems as monoidal indexed categories}\label{sec:systemsasmonicats}

In \cite{AbstractMachines} as well as in earlier works e.g.\ \cite{Vagner.Spivak.Lerman:2015a}, the authors investigate a categorical framework for modeling systems of systems using algebras for a monoidal category. In more detail, systems in a broad sense are perceived as lax monoidal pseudofunctors
\begin{displaymath}
\mathcal{W}_\mathcal{C}\to\Cat
\end{displaymath}
where $\mathcal{W}_\C$ is the monoidal category of $\C$-\emph{labeled boxes} and \emph{wiring diagrams} with types in a finite product category $\mathcal{C}$. Briefly, the objects in $\mathcal{W}_\mathcal{C}$ are pairs $X=(X^\mathrm{in},X^\mathrm{out})$ of finite sets equipped with functions to $\ob\mathcal{C}$, thought of as boxes
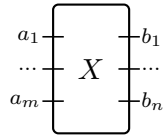
\begin{displaymath}
\begin{tikzpicture}[oriented WD, bbx=.1cm, bby =.1cm, bb port sep=.15cm]
	\node [bb={3}{3}] (X) {$X$};
	\draw[label]
		node[left=.1 of X_in1]  {$a_1$}
		node[left=.1 of X_in2]  {$\dotso$}
		node[left=.1 of X_in3]  {$a_m$}
		node[right=.1 of X_out1] {$b_1$}
		node[right=.1 of X_out2]  {$\dotso$}
		node[right=.1 of X_out3] {$b_n$};
\end{tikzpicture}
\end{displaymath}
where $X^\mathrm{in}=\{a_1,\ldots,a_m\}$ are the input ports, $X^\mathrm{out}=\{b_1,\ldots,b_n\}$
the output ones and all wires are associated to a $\mathcal{C}$-object expressing the type of information that can go through them. A morphism $\phi\colon X\to Y$ in this category consists
of a pair of functions
\begin{displaymath}
\left\{\begin{array}{l}
\inp{\phi}\colon\inp{X}\to\out{X}+\inp{Y} \\
\out{\phi}\colon\out{Y}\to\out{X}\end{array}\right.
\end{displaymath}
that respect the $\mathcal{C}$-types,
which roughly express which port is `fed information' by which.
Graphically, we can picture it as

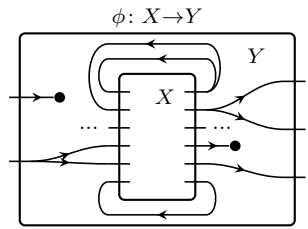
\begin{equation}\label{eq:wiringdiagpic}
\begin{tikzpicture}[oriented WD,baseline=(Y.center), bbx=2em, bby=1.2ex, bb port sep=1.2]
\node[bb={6}{6}] (X) {};
\node[bb={2}{3}, fit={($(X.north east)+(0.7,1.7)$) ($(X.south west)-(.7,.7)$)}] (Y) {};
\node [circle,minimum size=4pt, inner sep=0, fill] (dot1) at ($(Y_in1')+(.5,0)$) {};
\node [circle,minimum size=4pt, inner sep=0, fill] (dot2) at ($(X_out4)+(.5,0)$) {};
\draw[ar] (Y_in1') to (dot1);
\draw[ar] (X_out4) to (dot2);
\draw[ar] (Y_in2') to (X_in5);
\draw[ar] (Y_in2') to (X_in4);
\draw[ar] (X_out5) to (Y_out3');
\draw[ar] (X_out2) to (Y_out1');
\draw[ar] (X_out2) to (Y_out2');
\draw[ar] let \p1=(X.north west), \p2=(X.north east), \n1={\y1+\bby}, \n2=\bbportlen in
	(X_out1) to[in=0] (\x2+\n2,\n1) -- (\x1-\n2,\n1) to[out=180] (X_in1);
\draw[ar] let \p1=(X.north west), \p2=(X.north east), \n1={\y1+2*\bby}, \n2=\bbportlen in
	(X_out1) to[in=0] (\x2+\n2,\n1) -- (\x1-\n2,\n1) to[out=180] (X_in2);
\draw[ar] let \p1=(X.south west), \p2=(X.south east), \n1={\y1-\bby}, \n2=\bbportlen in
	(X_out6) to[in=0] (\x2+\n2,\n1) -- (\x1-\n2,\n1) to[out=180] (X_in6);	
\draw [label] node at ($(Y.north east)-(.5cm,.3cm)$) {$Y$}
              node at ($(X.north east)-(.4cm,.3cm)$) {$X$}
		      node[left=.1 of X_in3]  {$\dotso$}
		      node[right=.1 of X_out3] {$\dotso$}
		      node[above=of Y.north] {$\phi\colon X\to Y$}
	  ;
\end{tikzpicture}
\end{equation}
Composition of morphisms can be thought of a zoomed-in picture of three boxes, and the monoidal structure amounts to parallel placement of boxes as in
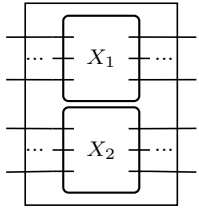
\begin{displaymath}
\begin{tikzpicture}[oriented WD,baseline=(Y.center), bbx=1.3em, bby=1ex, bb port sep=.06cm]
    \node[bb={3}{3}] (X1) {};
    \node[bb={3}{3},below =.5 of X1] (X2) {};
    \node[fit=(X1)(X2),draw] {};
    \draw[label] 
    node at ($(X1.west)+(1,0)$) {$X_1$}
    node at ($(X2.west)+(1,0)$) {$X_2$}
    node[left=.1 of X1_in2]  {$\dotso$}
    node[right=.1 of X1_out2]  {$\dotso$}
    node[left=.1 of X2_in2]  {$\dotso$}
    node[right=.1 of X2_out2]  {$\dotso$};
    \draw (X1_in1) -- (-2.5,1.7);
    \draw (X1_out1) -- (2.5,1.7);
    \draw (X1_in3) -- (-2.5,-1.7);
    \draw (X1_out3) -- (2.5,-1.7);
    \draw (X2_in1) -- (-2.5,-5.6);
    \draw (X2_out1) -- (2.5,-5.6);
    \draw (X2_in3) -- (-2.5,-9.1);
    \draw (X2_out3) -- (2.5,-9.1);
 \end{tikzpicture}
\end{displaymath}
There is a close connection between the definition of $\mathcal{W}_\mathcal{C}$ and that of \emph{Dialectica} categories as well as \emph{lenses}; such considerations are the topic of work in progress \cite{EverythingisDialectica}.

The systems-as-algebras formalism uses lax monoidal pseudofunctors from this category
$\mathcal{W}_\mathcal{C}$ to $\Cat$ that essentially receive a general picture such as 
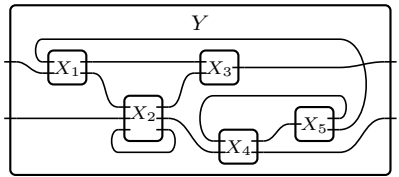
\begin{displaymath}
\begin{tikzpicture}[oriented WD, bb min width =.5cm, bbx=.5cm, bb port sep =1,bb port length=.08cm, bby=.15cm]
\node[bb={2}{2},bb name = {\tiny$X_1$}] (X11) {};
\node[bb={3}{3},below right=of X11,bb name = {\tiny$X_2$}] (X12) {};
\node[bb={2}{1},above right=of X12,bb name = {\tiny$X_3$}] (X13) {};
\draw (X11_out1) to (X13_in1);
\draw (X11_out2) to (X12_in1);
\draw (X12_out1) to (X13_in2);
\node[bb={2}{2}, below right = -1 and 1.5 of X12, bb name = {\tiny$X_4$}] (X21) {};
\node[bb={1}{2}, above right=-1 and 1 of X21,bb name = {\tiny$X_5$}] (X22) {};
\draw (X21_out1) to (X22_in1);
\draw let \p1=(X22.north east), \p2=(X21.north west), \n1={\y1+\bby}, \n2=\bbportlen in
         (X22_out1) to[in=0] (\x1+\n2,\n1) -- (\x2-\n2,\n1) to[out=180] (X21_in1);
\node[bb={2}{2}, fit = {($(X11.north east)+(-1,3)$) (X12) (X13) ($(X21.south)$) ($(X22.east)+(.5,0)$)}, bb name ={\scriptsize $Y$}] (Z) {};
\draw (Z_in1') to (X11_in2);
\draw (Z_in2') to (X12_in2);
\draw (X12_out2) to (X21_in2);
\draw let \p1=(X22.south east),\n1={\y1-\bby}, \n2=\bbportlen in
  (X21_out2) to (\x1+\n2,\n1) to (Z_out2');
 \draw let \p1=(X12.south east), \p2=(X12.south west), \n1={\y1-\bby}, \n2=\bbportlen in
  (X12_out3) to[in=0] (\x1+\n2,\n1) -- (\x2-\n2,\n1) to[out=180] (X12_in3);
\draw let \p1=(X22.north east), \p2=(X11.north west), \n1={\y2+\bby}, \n2=\bbportlen in
  (X22_out2) to[in=0] (\x1+\n2,\n1) -- (\x2-\n2,\n1) to[out=180] (X11_in1);
\draw let \p1=(X13_out1), \p2=(X22.north east), \n2=\bbportlen in
 (X13_out1) to (\x1+\n2,\y1) -- (\x2+\n2,\y1) to (Z_out1');
\end{tikzpicture}
\end{displaymath}
(which really takes place in the underlying operad of $\mathcal{W}_\mathcal{C}$) and assign systems of a certain kind to all inner boxes; the lax monoidal and pseudo\-functorial structure of this assignment formally produce a system of the same kind for the outer box. 

Examples of such systems are discrete dynamical systems (Moore machines in the finite case), continuous dynamical systems but also more general systems with deterministic or total conditions; details can be found in the provided references. Since all these systems are lax monoidal pseudofunctors from the non-cocartesian monoidal category of wiring diagrams to $\Cat$, i.e.\ monoidal indexed categories, the monoidal Grothendieck construction \cref{thm:mainthm} induces a corresponding monoidal fibration in each system case, and this global structure does not reduce to a fibrewise one. 

For example, the algebra for discrete dynamical systems
\cite[\S 2.3]{AbstractMachines}
\begin{equation}\label{eq:DDS}
\mathrm{DDS}\colon\mathcal{W}_\Set\to\Cat  
\end{equation}
assigns to each box $X=(X^\mathrm{in},X^\mathrm{out})$
the category of all discrete dynamical systems with fixed input and output sets being $\prod_{x\in X^\mathrm{in}}x$ and $\prod_{y\in X^\mathrm{out}}y$ respectively. There exist morphisms between systems of the same input and output set, but not between those with different ones. To each morphism, i.e.\ wiring diagram as in \cref{eq:wiringdiagpic}, $\mathrm{DDS}$
produces a functor that maps an inner discrete dynamical system to a new outer one, with changed input and output sets accordingly.
(Pseudo)functoriality of this assignment 
allows the coherent zoom-in and zoom-out on dynamical systems built out of smaller dynamical systems, and monoidality allows the creation of new dynamical systems on parallel boxes.

Being a monoidal indexed category,  \cref{eq:DDS} gives rise to a monoidal opfibration over $\mathcal{W}_\Set$. Its  total category $\inta \mathrm{DDS}$ has objects all dynamical systems with arbitrary input and output sets,  morphisms that can now go between systems of different inputs/outputs, and also a natural tensor product inherited from that in $\mathcal{W}_\Set$ and the laxator of $\inta\mathrm{DDS}$. In a sense, this category has all the required flexibility for the direct communication (via morphisms in the total category) between any discrete dynamical system, or any composite of systems or parallel placement of them, whereas the wiring diagram algebra \cref{eq:DDS} focuses on the machinery of building new discrete dynamical systems systems from old.

This classic change of point of view also transfers over to maps of algebras, i.e.\ indexed monoidal 1-cells. As an example, see \cite[\S 5.1]{AbstractMachines}, discrete dynamical systems can naturally be viewed as general \emph{total} and \emph{deterministic} machines denoted by $\mathrm{Mch}^\mathrm{td}$,
via a monoidal pseudonatural transformation
\[
\begin{tikzcd}
    \mathcal{W}_\Set
    \arrow[dr, "\mathrm{DDS}"]
    \arrow[dd]
    \\
    \arrow[r, phantom, "\Downarrow"]
    &
    \Cat
    \\
    \mathcal{W}_{\widetilde{\mathrm{Int}_N}}
    \arrow[ur,"\mathrm{Mch}^\mathrm{td}",swap]
\end{tikzcd}\]
which also changes the type of input and output wires from sets to \emph{discrete interval sheaves} $\widetilde{\mathrm{Int}_N}$. This gives rise to a monoidal opfibred 1-cell
\begin{displaymath}
\begin{tikzcd}
    \inta\mathrm{DDS}
    \ar[r]
    \ar[d] 
    & 
    \inta\mathrm{Mch}^\mathrm{td}
    \ar[d] 
    \\
    \mathcal{W}_\Set
    \ar[r] 
    & 
\mathcal{W}_{\widetilde{\mathrm{Int}}_N}
\end{tikzcd}
\end{displaymath}
which provides a direct functorial translation between the one sort of system to the other in a way compatible with the monoidal structure.

As a final note, this method of modeling certain objects as algebras for a monoidal category (a.k.a. strict or general monoidal indexed categories) carries over to further contexts than systems and the wiring diagram category. Examples include hypergraph categories as algebras on cospans  \cite{HypergraphCats} and traced monoidal categories as algebras on cobordisms  \cite{TracedMonCatsAlg}. In all these cases, the monoidal Grothendieck construction gives a potentially fruitful change of perspective that should be further investigated.

\appendix

\section{Summary of structures}\label{sec:appendix}

The bulk of the main body of this paper is dedicated to proving various monoidal variations of the equivalence between fibrations and indexed categories, using general results in monoidal 2-category theory. In this appendix, we detail the descriptions of the (braided/symmetric) monoidal structures on the total category of the Grothendieck construction, assuming the appropriate data is present. We also provide a hands-on correspondence that underlies the proof of \cref{thm:fibrewise=global} regarding the transfer of monoidal structure from a functor to its target and vice versa. We hope this section can serve as a quick and clear reference on some fundamental constructions of this work.

\subsection{Monoidal structures}\label{sec:monoidal}

As sketched under \cref{cor:fixedbasemonoidalGr}, let $(\X, \otimes, I)$ be a monoid\-al category, and \[(\M, \mu, \mu_0) \maps (\X^\op, \otimes^\op, I) \to (\Cat, \times, \1)\] a monoidal indexed category, a.k.a. lax monoidal pseudofunctor.  Recall that $\mu$ is pseudonatural transformation consisting of functors $\mu_{x,y} \maps \M x \times \M y \to \M(x \otimes y)$ for any objects $x$ and $y$ of $\X$, and natural isomorphisms
\[
\begin{tikzcd}[column sep = 70]
    \M z \times \M w 
    \arrow[r, "\M f \times \M g"]
    \arrow[d, "\mu_{z,w}", swap]
    &
    \M x \times \M y
    \arrow[d, "\mu_{x,y}"]
    \arrow[dl, phantom, "{\scriptstyle\stackrel{\mu_{f,g}}{\cong}}"]
    \\
    \M(z \otimes w)
    \arrow[r, "\M(f \otimes g)", swap]
    &
    \M(x \otimes y)
\end{tikzcd}\]
for any arrows $f \maps x \to z$ and $g \maps y \to w$ in $\X$. Also the unique component of $\mu_0$ is the functor $\mu_0\colon\1\to\M(I)$.

The induced tensor product functor on the total category, denoted as $\otimes_\mu \maps \inta \M \times \inta \M \to \inta \M$, is given on objects by 
\begin{displaymath}
    (x,a) \otimes_\mu (y,b) = (x \otimes y, \mu_{x,y}(a,b))
\end{displaymath}
On morphisms $\left(f\colon x\to z, k\colon a\to(\M f)c\right)$ and $\left(g\colon y\to w,\ell\colon b\to(\M g)d\right)$, we get
\[
    (f, k) \otimes_\mu (g, \ell) = (x\ot y\xrightarrow{f \otimes g} z\ot w, \mu_{f,g}(\mu_{x, y}(k, \ell)))
\]
where the latter is the composite morphism
\begin{displaymath}
\mu_{x,y}(a,b)\xrightarrow{\mu_{x,y}(k,\ell)}\mu_{x,y}\left((\M f)(c),(\M g)(d)\right)\xrightarrow{\sim}\M(f\otimes g)(\mu_{z.w}(c,d))\textrm{ in }\M(x\otimes y).
\end{displaymath}
The monoidal unit is $I_\mu=(I, \mu_0)$.

If $a_{x,y,z} \maps (x \otimes y) \otimes z \to x \otimes (y \otimes z)$ denotes the associator in $\X$, the associator for $(\inta \M, \otimes_\mu, I_\mu)$ is given by
\begin{displaymath}
    \alpha_{(x,b), (y,c), (z,d)} = (\alpha_{x,y,z}, \omega_{x,y,z} (b,c,d))
\end{displaymath}
where $\omega$ is the invertible modification \cref{eq:omega}.

If $l_x \maps I \otimes x \to x$ and $r_x \maps x \otimes I \to x$ are the left and right unitors in $\X$, the unitors in $\inta \M$ are defined as
\begin{gather*}
\lambda_x = (l_x, \xi_x^{\text{-}1}(a))\colon (I,\mu_0)\otimes_\mu(x,a)\to(x,a) \\
\rho_x = (r_x, \zeta_x(a))\colon(x,a)\otimes_\mu(I,\mu_0)\to(x,a)
\end{gather*}
where $\zeta$ and $\xi$ are invertible modifications as in \cref{eq:omega}.

We now turn to the correspondence between 1-cells of \cref{thm:mainthm}:
given a monoidal indexed 1-cell
\[
\begin{tikzcd}
    (\X, \otimes, I)^\op
    \arrow[dr, "{(\M, \mu, \mu_0)}"]
    \arrow[dd, "{(F, \psi, \psi_0)^\op}", swap]
    \\
    \arrow[r, phantom, "\Downarrow{\scriptstyle\tau}"]
    &
    (\Cat, \times, \1)
    \\
    (\Y, \otimes, I)^\op
    \arrow[ur, "{(\N, \nu, \nu_0)}", swap]
\end{tikzcd}\]
where $\M$ and $\N$ are lax monoidal pseudofunctors and $F$ is a monoidal functor, as in \cref{prop:moni1cell},
we first of all obtain an ordinary fibred 1-cell $(P_\tau,F)\colon P_\M\to P_\N$ as explained above \cref{eq:inducedfibred1cell}
\begin{displaymath}
\begin{tikzcd}
    \inta \M
    \ar[r,"P_\tau"]
    \ar[d,"P_\M"'] 
    & 
    \inta
    \N
    \ar[d,"P_\N"] 
    \\
    \X
    \ar[r,"F"'] 
    & 
    \Y
\end{tikzcd}
\end{displaymath}
with $P_\tau (x, a) = (F x, \tau_x (a))$. The functor $F$ is already monoidal, and $P_\tau$ obtains a monoidal structure too: for example, there are isomorphisms
\begin{displaymath}
    P_\tau(x,a) \otimes_\nu P_\tau (y, b) \xrightarrow{\sim} P_\tau ((x, a) \otimes_\mu (y, b)) \quad \textrm{ in } \inta \N
\end{displaymath}
between the objects
\begin{align*}
    P_\tau (x, a) \otimes_\nu P_\tau (y, b)
    & =(Fx, \tau_x (a)) \ot_\nu (Fy, \tau_y (b) = (Fx \ot Fy, \nu_{Fx, Fy} (\tau_x (a), \tau_y (b)) \\
    P_\tau((x,a)\otimes_\mu(y,b))
    & =P_\tau(x\ot y,\mu_{x,y}(a,b))=(F(x\ot y),\tau_{x\ot y}(\mu_{x,y}(a,b)))
\end{align*}
given by $\psi_{x,y}\colon Fx\ot Fy\xrightarrow{\sim}F(x\ot y)$ and by 
\[\nu_{Fx,Fy}(\tau_x(a),\tau_y(b))\cong\N(\psi_{x,y})(\tau_{x\ot y}(\mu_{x,y}(a,b)))\]
essentially given by the monoidal pseudonatural isomorphism \cref{eq:monpseudocomponents}
for $\tau\colon\M\Rightarrow\N F^\op$. 
As a result, $(P_\tau,F)$ is indeed a monoidal fibred 1-cell as in \cref{prop:monoidalfibred1cell}.

Finally, it can be verified that starting with a monoidal indexed 2-cell as in \cref{prop:moni2cell}, the induced fibred 2-cell \cref{eq:inducedfibred2cell} is monoidal, i.e.\ $P_m$ satisfies the conditions of a monoidal natural transformation.

Regarding the induced braided and symmetric monoidal structures, 
suppose that $(\X,\ot,I)$ is a braided monoidal category, with braiding $b$ with components
 \[
\braid_{x,y} \maps x \otimes y \xrightarrow{\sim} y \otimes x;
\]
then $\X^\op$ is braided monoidal with the inverse braiding, namely $(\X^\op,\ot^\op,I,\braid^{-1})$.
Now if $(\M,\mu,\mu_0)\colon\X^\op\to\Cat$ is a \emph{braided} lax monoidal pseudofunctor, i.e.\ a braided monoidal indexed category,
by \cref{thm:mainthm} we have an induced braided monoidal structure on $(\inta\M, \ot_\mu, I_\mu)$, namely
\[B_{(x, a), (y, b)} \maps (x, a) \otimes_\mu (y,b)=(x\ot y,\mu_{x,y}(a,b)) \to (y,b) \otimes_\mu (x, a)=(y\ot x,\mu_{y,x}(b,a))\] are given by
$\braid_{x,y} \colon x \ot y \cong y \ot x$ in $\X$ and $(v_{x,y})_{(a,b)} \colon \mu_{x,y} (a,b) \cong \M (\braid^{-1}_{x,y}) (\mu_{y,x} (b,a))$, where $v$ is as in \cref{eq:brweakmonpseudo}. 

If $\M$ is a symmetric lax monoidal pseudofunctor, it can be verified that 
\[B_{(y, b),(x, a)} \circ B_{(x, a),(y, b)} = 1_{(x,a) \ot_\mu (y,b)}\] therefore $\inta \M$ is also symmetric monoidal, as is the monoidal fibration $P_\M \colon \inta \M \to \X$.

\subsection{Monoidal Indexed Categories as ordinary pseudofunctors}\label{monicat=imoncat}

Here we detail the correspondence between monoidal opindexed categories and a pseudofunctors into $\MonCat$ when the domain is a cocartesian monoidal category, as established by \cref{thm:fibrewise=global}; the one for indexed categories is of course similar. We denote by $\nabla_x \colon x+x \to x$ the induced natural components due to the universal property of coproduct, and $
\iota_x \colon x \to x+y$ the inclusion into a coproduct.

Start with a lax monoidal pseudofunctor $\M \colon (\X, +, 0) \to (\Cat, \times, \1)$ equipped with $\mu_{x,y}\colon \M(x) \times \M(y) \to \M(x + y)$ and $\mu_0 \colon \1 \to \M(0)$, which gives the global monoidal structure \cref{eq:globalmonstr} of the corresponding opfibration. There exists an induced monoidal structure on each fibre $\M(x)$ as follows:
\begin{gather}
\label{eq:explicitstructure1}
    \otimes_x \colon \M(x) \times \M(x) \xrightarrow{\mu_{x,x}} \M(x + x) \xrightarrow{\M(\nabla)} \M(x)
    \\
    I_x \colon \1 \xrightarrow{\mu_0} \M(0) \xrightarrow{\M(!)} \M(x) \nonumber
\end{gather}
Moreover, each $\M f \colon \M x \to \M y$ is a strong monoidal functor, with $\phi_{a,b} \colon (\M f)(a) \ot_y (\M f)(b) \xrightarrow{\sim} \M f(a \ot_xb)$ and $\phi_0 \colon I_y \xrightarrow{\sim} (\M f) I_x$ essentially given by the following isomorphisms
\begin{equation}\label{eq:strongmonreindex}
\begin{tikzcd}[row sep=.3in,column sep=.8in]
    \M x \times \M x
    \ar[r,"\M f\times\M f"]
    \ar[d,"\mu_{x,x}"']
    \ar[dr,phantom,"{\scriptstyle\stackrel{\mu^{f,f}}{\cong}}"description] 
    & 
    \M y \times \M y
    \ar[d,"\mu_{y,y}"] 
    \\
    \M(x+x)
    \ar[d,"\M(\nabla_x)"']
    \ar[r,"\M(f+f)"description]
    \ar[dr,phantom,"{\scriptstyle\cong}"description] 
    & \M(y+y)\ar[d,"\M(\nabla_y)"] 
    \\
    \M x \ar[r,"\M f"'] 
    & 
    \M y
\end{tikzcd}
\qquad
\begin{tikzcd}
    \1
    \ar[r,"\mu_0"]
    \ar[d,"\mu_0"']
    \ar[ddr,phantom,"{\scriptstyle\cong}"description] 
    & 
    \M(0)
    \ar[dd,"\M(!)"] 
    \\
    \M(0)
    \ar[d,"\M(!)"'] 
    &\\
    \M x 
    \ar[r,"\M f"']
    & 
    \M y
\end{tikzcd}
\end{equation}
since $\nabla$ and $!$ are natural and $\M$ is a pseudofunctor.

In the opposite direction, take an ordinary pseudofunctor $\M\colon\X\to\MonCat$ into the 2-category of monoidal categories, strong monoidal functors and monoidal natural transformations, with $\otimes_x\colon\M(x)\times\M(x)\to\M(x)$ and $I_x$ the fibrewise monoidal structures in every $\M x$. We can use those to endow $\M$ with a lax monoidal structure via 
\begin{gather*}
    \mu_{x,y} \colon \M(x) \times \M(y) \xrightarrow{\M(\iota_x) \times \M(\iota_y)} \M(x+y) \times \M(x+y) \xrightarrow{\otimes_{x+y}} \M(x+y) 
    \\
    \mu_0 \colon \1 \xrightarrow{I_0} \M(0)
\end{gather*}
The fact that all $\M f$ are strong monoidal imply that the above components form pseudonatural transformations, and all appropriate conditions are satisfied. 

\begin{rmk}\label{rmk:laxmonfun}
In the strict context, a weakly lax monoidal 2-functor $\M\colon(\X,+,0)\to(\Cat,\times,\1)$ with natural laxator and unitor bijectively corresponds to a functor $\X\to\MonCat_\textrm{st}$
since \cref{eq:strongmonreindex} are in fact strictly commutative, by naturality of $\mu,\mu_0$ and functoriality of $\M$.

In the even more special case of an ordinary lax monoidal functor $\M\colon(\X,+,0)\to(\Cat,\times,\1)$, the fibres $\M(x)$ turn out to be \emph{strict} monoidal. For example, strict associativity of the tensor is established by
\begin{displaymath}
\begin{tikzcd}[column sep=.15in]
\M x\times\M x\times\M x\ar[rr,"1\times\ot_x"]\ar[dr,"1\times\mu_{x,x}"']\ar[dddd] &\phantom{A}\ar[d,phantom,"\scriptstyle\cref{eq:explicitstructure1}"]& \M x\times \M x\ar[ddrr,bend left,"\ot_x","\cref{eq:explicitstructure1}"']\ar[dr,"\mu_{x,x}"description] && \\
&\M x\times \M (x+x)\ar[dd,phantom,"(*)"]\ar[ur,"1\times\M(\nabla)"description]\ar[dr,"\mu_{x,x+x}"description] && \M (x+x)\ar[dr,"\M\nabla"description] && \\
&& \M(x+x+x)\ar[ur,"\M (1+\nabla)"description]\ar[dr,"\M(\nabla+1)"description] && \M x \\
& \M(x+x)\times\M x\ar[d,phantom,"\scriptstyle\cref{eq:explicitstructure1}"]\ar[ur,"\mu_{x+x,x}"description]\ar[dr,"\M(\nabla)\times1"description] && \M(x+x)\ar[ur,"\M\nabla"description] & \\
\M x\times\M x\times\M x\ar[ur,"\mu_{x,x}\times1"]\ar[rr,"\ot_x\times1"'] &\phantom{A}&\M x\times\M x\ar[ur,"\mu_x"description]\ar[uurr,bend right,"\ot_x"',"\cref{eq:explicitstructure1}"] &&
\end{tikzcd}
\end{displaymath}
where the three diamond-shaped diagrams on the right commute due to naturality of $\mu$ as well as associativity of $\nabla$ and functoriality of $\M$ already in the monoidal strict opindexed case, whereas $(*)$ is in general $\omega$ from $\cref{eq:omega}$ which in this case is an identity.
\end{rmk}

\bibliographystyle{alpha}
\bibliography{references}

\end{document}